\documentclass{article}
\usepackage{amssymb, amsmath, amscd, latexsym, enumerate}
\usepackage{eepic}
\usepackage[dvips]{graphicx} 
\usepackage[usenames]{color}
\title{The Morse index of a triply periodic \\
minimal surface}
\author
 {Norio Ejiri   
    \thanks{Partly supported by JSPS Grant-in-Aid for 
    Scientific Research (C) 15K04859.}
 \and 
  Toshihiro Shoda 
    \thanks{Partly supported by JSPS Grant-in-Aid for Scientific Research (C) 16K05134.
  }
 }
\date{}
\usepackage{amsthm}
\theoremstyle{plain}
 \newtheorem{theorem}{Theorem}[section]
 \newtheorem{theorem*}{Main Theorem}
 \newtheorem*{conjecture}{Conjecture}

 \newtheorem{lemma}[theorem]{Lemma}
 \newtheorem{corollary}[theorem]{Corollary}
\theoremstyle{definition}

 \newtheorem{remark}[theorem]{Remark}

\renewcommand{\Re}{\operatorname{Re}}
\renewcommand{\Im}{\operatorname{Im}}

\numberwithin{equation}{section}
\numberwithin{figure}{section}
\numberwithin{table}{section}
\begin{document}
\maketitle

\begin{abstract}
In the previous work, the first author established an algorithm to compute 
the Morse index and the nullity of an $n$-periodic minimal surface in $\mathbb{R}^n$. 
In fact, the Morse index can be translated into the number of negative eigenvalues of a real symmetric matrix and 
the nullity can be translated into the number of zero-eigenvalue of a Hermitian matrix. The two key matrices consist of 
periods of the abelian differentials of the second kind on a minimal surface, and 
the signature of the Hermitian matrix gives a new invariant of a minimal surface. 
On the other hand, H family, rPD family, tP family, tD family, and tCLP family 
of triply periodic minimal surfaces in $\mathbb{R}^3$ 
have been studied in physics, chemistry, and crystallography. 
In this paper, we first determine the two key matrices 
for the five families explicitly. 
As its applications, by numerical arguments, we compute the Morse indices, nullities, and signatures 
for the five families. 
\end{abstract}

\section{Introduction}

The Morse index (resp. the nullity) of a compact oriented minimal submanifold in an oriented Riemannian manifold is defined as 
the sum of the dimensions of the eigenspaces corresponding to negative eigenvalues (resp. zero-eigenspace) of 
the Jacobi operator of the area. The purpose of this work is to compute the Morse index and the nullity of a compact oriented minimal surface in a flat three-torus. 
Moreover, we consider a signature of a minimal surface defined in \cite{Ej2}, and compute the signature 
of a compact oriented minimal surface in a flat three-torus. 
Now we refer to backgrounds. 

In 1968, Simons \cite{Sim} gave the second variational formula of the area and compute the Morse index and the nullity of a totally geodesic 
subsphere in the sphere. By his technique, we can see that the Morse index (resp. the nullity) of a totally geodesic subtorus in a flat 
three-torus is zero (resp. one). Next impressive developments were obtained by Montiel-Ros \cite{Mont-Ros} and 
Ross \cite{Ross}. 
Montiel-Ros considered the Dirichlet eigenvalue problem and the Neumann eigenvalue problem of the Laplacian to compute the Morse 
index and the nullity of a minimal surface.  
Their study includes the result that tCLP family consists of minimal surfaces with Morse index three and nullity three. 
Ross proved that Schwarz P surface, D surface, and Schoen's Gyroid are volume preserving stable, respectively. 
Applying his arguments, we find that each of the three minimal surfaces has Morse index one and nullity three. 
But the Morse index and the nullity have not been computed for other examples in the past two decades. 

Recently, the first author \cite{Ej1, Ej2} established a Moduli theory of compact oriented minimal surfaces in flat tori via 
the Morse index and the nullity. 
Also, he gave a procedure to compute the Morse index of a minimal surface with 
only trivial Jacobi fields. 
The procedure is to reduce computing the Morse index and 
the nullity of such minimal surfaces to fundamental arguments 
of eigenvalues in linear algebra. 
The main difficulty is to determine a canonical homology basis 
and the period matrix for each minimal surface explicitly. 
Our technique developed in this paper is to overcome this by using the method which is faithful to
the basics of compact Riemann surfaces, and shows that 
the procedure turns out to be practical. 
Recall that a normal vector field vanishing the Jacobi operator of the area is called 
a {\it Jacobi field} and the dimension of the space of Jacobi fields is equal to the nullity of a minimal surface. 
It is well-known that normal components of the Killing vector fields generated by translation on the torus give rise to Jacobi fields. 
So if we consider a non-totally geodesic compact oriented minimal surface in a flat $n$-torus $\mathbb{R}^n/\Lambda$, 
then it has nullity at least 
$n(=\dim \mathbb{R}^n/\Lambda)$. The Killing vector fields generated by translation on the torus is called {\it trivial Jacobi fields}, 
and a minimal surface has only trivial Jacobi fields if and only if its nullity is equal to $n$. 
Moreover, he introduced a new invariant for the Moduli theory of compact oriented minimal surfaces in 
flat tori which is called a signature of a minimal surface (see Section Two for the detail). 
The present work suggests that it might be easier to compute the signature than the Morse index. 

On the other hand, triply periodic minimal surfaces in $\mathbb{R}^3$ have been 
studied in physics, chemistry, and crystallography. 
Schr\"{o}der-Turk, Fogden, and Hyde \cite{TFH} studied one-parameter families of triply periodic minimal surfaces in $\mathbb{R}^3$. 
These one-parameter families which are called H family, rPD family, tP family, tD family, 
and tCLP family, 
contain many classical examples (Schwarz P surface, D surface, Schwarz H surface, 
and Schwarz CLP surface). 
Note that a triply periodic minimal surface properly immersed in $\mathbb{R}^3$ 
corresponds to a minimal immersion of a compact oriented surface into a flat three-torus. 
Hence the above one-parameter families are related to our works. 
Also, the above five families consist of minimal surfaces of genus three, and the following Ros' result clarifies an importance of 
such families. 
In fact, in this case, the minimum value of Morse index must be one (see Remark~\ref{rem-introduction}) and 
Ros \cite{Ros} proved that genus of a compact oriented minimal surface in a flat three-torus with Morse index one must be three. 
In the present paper, we compute the Morse index, the nullity, and the signature of each of the above one-parameter families. 

To state our main results, we review some fundamental arguments in the theory of minimal surfaces. 
Let $f:M\to \mathbb{R}^n/\Lambda$ be a minimal immersion of a compact oriented surface $M$ into a flat $n$-torus 
$\mathbb{R}^n/\Lambda$. With the induced conformal structure, $M$ is a compact Riemann surface and $f$ is called a 
{\it conformal minimal immersion}. Our object is then to study conformal minimal immersions of compact Riemann surfaces 
in flat $n$-tori. 
For a conformal minimal immersion, the following theorem is a basic tool. 
\begin{theorem}[Weierstrass representation formula]\label{W-rep}$\,$

Let $f:M \to \mathbb{R}^n/\Lambda$ be a conformal minimal 
immersion. Then, up to translations, $f$ can be represented by the following 
path-integrals{\rm :}
\[
f(p)=
\Re\int_{p_0}^p (
\omega_1, \,
\ldots ,\,
\omega_n )^t \quad {\rm mod}\>\>\> \Lambda, 
\]
where $p_0$ is a fixed point on $M$ and the $\omega_i$'s are holomorphic 
differentials on $M$ satisfying the following three conditions.  
\begin{align}
&\label{conf-cond} \omega_1^2+\cdots+\omega_n^2=0, \\
&\label{imm-cond} \omega_1,\,\ldots,\,\omega_n \>\> {\rm have}  \>\> {\rm no}  \>\>
{\rm common} \>\>{\rm zeros}, \\
&\label{period-cond} \left\{
\Re\int_{C} 
(\omega_1 ,\,
\ldots ,\,
\omega_n )^t \>\bigg|\>
C\in H_1(M,\,\mathbb{Z})
\right\} \>\>{\rm is}\>\>{\rm a}\>\>{\rm sublattice} \>\>{\rm of} 
\>\> \Lambda.
\end{align}
Conversely, the real part of path-integrals of holomorphic differentials satisfying 
the above three conditions defines a conformal minimal immersion of $M$ into a flat $n$-torus. 
\end{theorem}

Let $index_A$ (resp. $nullity_A$) denote the Morse index (resp. the nullity) of a minimal surface, 
and $(p,\,q)$ the signature of a minimal surface. 
Now we refer to our main results as follows. 
\begin{theorem*}[H family]\label{main-H}$\,$

For $a \in(0,\,1)$, let $M$ be the hyperelliptic Riemann surface of genus three 
defined by $w^2=z(z^3-a^3) \left( z^3-1/a^3 \right) $ and $f$ the conformal 
minimal immersion given by
\[f(p)=\Re \int^p_{p_0} i(1-z^2,\,i(1+z^2),\,2 z)^t \dfrac{d z}{w}.\]

Then there exist $0<a_1< a_2 <1$ satisfying the following properties{\rm :} \\
{\rm (i)} $index_A=2$, $nullity_A=3$, and $(p,\,q)=(5,\,4)$ for $a \in (0,\,a_1)$, \\
{\rm (ii)} $index_A=1$, $nullity_A=4$, and $(p,\,q)=(4,\,4)$ for $a = a_1$, \\
{\rm (iii)} $index_A=1$, $nullity_A=3$, and $(p,\,q)=(4,\,5)$ for $a \in (a_1,\,a_2)$, \\
{\rm (iv)} $index_A=1$, $nullity_A=5$, and $(p,\,q)=(4,\,3)$ for $a = a_2$, \\
{\rm (v)} $index_A=3$, $nullity_A=3$, and $(p,\,q)=(6,\,3)$ for $a \in (a_2,\,1)$, \\
where $a_1\approx 0.49701$, $a_2\approx 0.71479$. 
We obtain the similar results for $a>1$. 
\end{theorem*}

\begin{theorem*}[rPD family, Karcher TT surface]\label{main-rPD}$\,$

For $a \in(0,\,1]$, let $M$ be the hyperelliptic Riemann surface of genus three 
defined by $w^2=z(z^3-a^3) \left( z^3+1/a^3 \right) $ and $f$ the conformal 
minimal immersion given by
\[f(p)=\Re \int^p_{p_0} i(1-z^2,\,i(1+z^2),\,2 z)^t \dfrac{d z}{w}.\]
The case $a=1/\sqrt{2}$ corresponds to Schwarz {\rm P} surface. 

Then there exist $0<a_1<1$ satisfying the following properties{\rm :} \\
{\rm (i)} $index_A=2$, $nullity_A=3$, and $(p,\,q)=(5,\,4)$ for $a \in (0,\,a_1)$, \\
{\rm (ii)} $index_A=1$, $nullity_A=4$, and $(p,\,q)=(4,\,4)$ for $a = a_1$, \\
{\rm (iii)} $index_A=1$, $nullity_A=3$, and $(p,\,q)=(4,\,5)$ for $a \in (a_1,\,1]$, \\
where $a_1\approx 0.494722$. We obtain the similar results for $a\geq 1$. 
\end{theorem*}

\begin{remark}$\,$

$a_1$'s in Main Theorem~{\ref{main-H}} and Main Theorem~{\ref{main-rPD}} are considered from other point of view in \cite{TFH}. 
\end{remark}

\begin{theorem*}[tP family, tD family]\label{main-tP}$\,$

For $a \in(2,\,\infty)$, let $M$ be the hyperelliptic Riemann surface of genus three 
defined by $w^2=z^8+a z^4+1 $ and $f$ the conformal 
minimal immersion given by
\[f(p)=\Re \int^p_{p_0} (1-z^2,\,i(1+z^2),\,2 z)^t \dfrac{d z}{w}.\]
The case $a=14$ corresponds to Schwarz {\rm P} surface. 

Then there exist $2<a_1< 14 < a_2 <\infty$ satisfying the following properties{\rm :} \\
{\rm (i)} $index_A=2$, $nullity_A=3$, and $(p,\,q)=(5,\,4)$ for $a \in (2,\,a_1)$, \\
{\rm (ii)} $index_A=1$, $nullity_A=3$, and $(p,\,q)=(4,\,5)$ for $a \in (a_1,\,a_2)$, \\
{\rm (iii)} $index_A=2$, $nullity_A=3$, and $(p,\,q)=(5,\,4)$ for $a \in (a_2,\,\infty)$, \\
{\rm (iv)} $index_A=1$, $nullity_A=4$, and $(p,\,q)=(4,\,4)$ for $a = a_1,\,a_2 $, \\
where $a_1\approx 7.40284$, $a_2\approx 28.7783$. 
\end{theorem*}
\textbf{\it{tD family}} is defined as a family of conjugate surfaces of 
minimal surfaces which belong to tP family. Thus we obtain the same result for tD family as 
the above. 
\begin{theorem*}[tCLP family]\label{main-tCLP}$\,$

For $a \in(-2,\,2)$, let $M$ be the hyperelliptic Riemann surface of genus three 
defined by $w^2=z^8+a z^4+1 $ and $f$ the conformal 
minimal immersion given by
\[f(p)=\Re \int^p_{p_0} (1-z^2,\,i(1+z^2),\,2 z)^t \dfrac{d z}{w}.\]
The case $a=0$ corresponds to Schwarz {\rm CLP} surface. 

Then, for an arbitrary $a\in (-2,\,2)$, 
$index_A=3$, $nullity_A=3$, and $(p,\,q)=(6,\,3)$. 
\end{theorem*}

Our main theorems imply that every compact oriented minimal surface which belongs to the above five families 
has Morse index at most three, and we obtain the same results for other families as well \cite{ES2}. 
So we propose the following ``one-two-three conjecture": 
\begin{conjecture}$\,$

If a compact oriented minimal surface in a flat three-torus has genus three, then 
we have $1\leq index_A\leq 3$. 
\end{conjecture}

The outline of the paper as follows: The second section contains a short story of the 
procedure to compute the Morse index and the nullity. Section Three gives explicit descriptions 
of the key matrices to compute the Morse index and the nullity for each one-parameter family. Section Four contains 
details of numerical arguments of our main results by Mathematica, and finally in Section Five there is a 
collection of calculations to determine a canonical homology basis and the period matrix 
for each one-parameter family of minimal surfaces as appendix. 

The authors would like to thank Wayne Rossman and Shoichi Fujimori for useful conversations 
about Mathematica. 

\section{An algorithm to compute the Morse index and the nullity}

In this section, we refer to a story to compute the Morse index and the nullity of a minimal surface in a flat torus. 
The details of the contents are given in \cite{Ej1, Ej2} (see also \cite{ES} for its outline). 

Let $R$ be a compact oriented surface of genus $\gamma$ and 
$\phi:R\to \mathbb{R}^n/\Lambda$ a smooth map. 
Schoen and Yau \cite{Sch-Yau} defined an energy $E_{\phi}$ on the Teichm\"{u}ller space $\mathcal{T}_{\gamma}$ with 
base surface $R$. 
Recall that a point in $\mathcal{T}_{\gamma}$ is an equivalent class of pairs $[(M,\,h)]$ of a compact Riemann surface $M$ 
of genus $\gamma$ and an orientation preserving diffeomorphism $h:R\to M$. The following result suggests that the energy $E_{\phi}$ is one of 
important objects in Differential Geometry. 

\begin{theorem}[\cite{Sa-Uh, Sch-Yau}]

A critical point $p=(M,\,h) \in \mathcal{T}_{\gamma}$ of $E_{\phi}$ is corresponding to 
a conformal {\rm (}branched{\rm )} minimal immersion $f:M\to \mathbb{R}^n/\Lambda$. 
\end{theorem}

Since $\mathcal{T}_{\gamma}$ is diffeomorphic to $\mathbb{R}^{6\gamma-6}$, we can define the Morse index and the nullity of 
$E_{\phi}$ at a critical point by the Hessian. Let $index_{E}$ denote the Morse index of 
$E_{\phi}$ at a critical point and $nullity_E$ the nullity of 
$E_{\phi}$ at a critical point. 
We can translate $index_a$ (resp. $nullity_a$) into $index_E$ (resp. $nullity_E$) as follows. 

\begin{theorem}[Theorem~{3.4} in \cite{Ej1}]

Suppose that  $f:M\to \mathbb{R}^n/\Lambda$ is a conformal minimal immersion and 
$p=[(M,\,h)] \in \mathcal{T}_{\gamma}$ is the corresponding critical point of $E_{\phi}$. Then 
\[
index_a = index_E, \quad nullity_a= nullity_E + n. 
\]
\end{theorem}

As an immediate consequence, we have 
\begin{corollary}[Corollary~{3.20} in \cite{Ej1}]

Suppose that $f:M\to \mathbb{R}^n/\Lambda$ is a conformal minimal immersion and 
$p=[(M,\,h)] \in \mathcal{T}_{\gamma}$ is the corresponding critical point of $E_{\phi}$. 
If $f$ has only trivial Jacobi fields, then $E_{\phi}$ is non-degenerate at $p$. 
\end{corollary}

Before we refer to an algorithm to compute $index_E$ and $nullity_E$, we review some arguments. 

Let $\{A_j,\,B_j\}_{j=1}^{\gamma}$ be a canonical homology basis and $\{\varphi_1,\,\cdots,\, \varphi_{\gamma}\}$ 
a basis of the space of holomorphic differentials on a compact Riemann surface $M$. 
We set $\Phi=(\varphi_1,\,\cdots,\,\varphi_{\gamma})^t$. Then there exists a unique basis $\{\varphi_j\}_{j=1}^{\gamma}$ 
with the following property: 
\begin{equation}\label{riemann-mat}
\begin{pmatrix}
\displaystyle{\int_{A_1}} \Phi & \cdots & \displaystyle{\int_{A_{\gamma}}} \Phi  & 
\displaystyle{\int_{B_1}} \Phi & \cdots & \displaystyle{\int_{B_{\gamma}}} \Phi
\end{pmatrix}
=
\begin{pmatrix}
I_{\gamma} & \tau
\end{pmatrix}, 
\end{equation}
where $I_{\gamma}$ is the identity matrix of degree $\gamma$ and $\tau$ is a complex symmetric matrix of degree $\gamma$ 
with $\Im \tau >0$. This $\tau$ is called a {\it Riemann matrix} of $M$. 
Let $L_{m,\,n}$ be a set of real $(m,\,n)$ matrices and $K_{m,\,n}$ a set of complex $(m,\,n)$ matrices. 
Suppose that 
$ f(p)=\Re 
\begin{pmatrix}
\displaystyle{\int_{p_0}^p} \omega_1 & \cdots & \displaystyle{\int_{p_0}^p} \omega_n 
\end{pmatrix}^t$ is a conformal minimal immersion of a Riemann surface $M$ 
as in Theorem~\ref{W-rep}. Then we call 
\[
\Re
\begin{pmatrix}
\displaystyle{\int_{A_1}} \omega_1 & \cdots & \displaystyle{\int_{A_{\gamma}}} \omega_1 
& \displaystyle{\int_{B_1}} \omega_1 & \cdots & \displaystyle{\int_{B_{\gamma}}} \omega_1 \\
 &  \cdots &  &   & \cdots & \\ 
\displaystyle{\int_{A_1}} \omega_n & \cdots & \displaystyle{\int_{A_{\gamma}}} \omega_n 
& \displaystyle{\int_{B_1}} \omega_n & \cdots & \displaystyle{\int_{B_{\gamma}}} \omega_n
\end{pmatrix}
\in L_{n,\,2 \gamma} 
\]
a {\it real period matrix} and 
\[
\begin{pmatrix}
\displaystyle{\int_{A_1}} \omega_1 & \cdots & \displaystyle{\int_{A_{\gamma}}} \omega_1 
& \displaystyle{\int_{B_1}} \omega_1 & \cdots & \displaystyle{\int_{B_{\gamma}}} \omega_1 \\
 &  \cdots &  &   & \cdots & \\ 
\displaystyle{\int_{A_1}} \omega_n & \cdots & \displaystyle{\int_{A_{\gamma}}} \omega_n 
& \displaystyle{\int_{B_1}} \omega_n & \cdots & \displaystyle{\int_{B_{\gamma}}} \omega_n
\end{pmatrix}
\in K_{n,\,2\gamma}
\]
a {\it complex period matrix}. Let $L \in L_{n,\,2\gamma}$ be a real period matrix of $f$. 
By using a decomposition $L=(L_1,\, L_2)$ 
($L_j \in L_{n,\, \gamma}$), 
\[
\begin{pmatrix}
\omega_1 & \cdots & \omega_n
\end{pmatrix}^t =\dfrac{1}{2} (L_1+i [L_1 \Re\tau -L_2 ] (\Im \tau)^{-1} ) \Phi
\] 
holds (see \S~{7} in \cite{Ej1}, see also p.161 in \cite{ES}). 
Note that we have to assume \eqref{conf-cond} for the above. 
By setting 
\[
K(\tau,\,L)= \dfrac{1}{2} (L_1+i [L_1 \Re\tau -L_2 ] (\Im \tau)^{-1} ) 
\]
and \eqref{riemann-mat}, the complex period matrix can be written as 
\begin{equation}\label{cplx-period}
\begin{pmatrix}
K(\tau,\,L) & K(\tau,\,L) \tau
\end{pmatrix} 
\in K_{n,\,\gamma}\times K_{n,\,\gamma} = K_{n,\,2\gamma}. 
\end{equation}
Using \eqref{cplx-period}, we can construct a complex isotropic cone in 
$K_{n,\,\gamma}\times K_{n,\,\gamma}=K_{n,\,2\gamma}$. Moreover, 
if a minimal surface has only trivial Jacobi fields, then the complex isotropic cone defined by \eqref{cplx-period} must be 
complex Lagrangian (see Theorem~{8.1} in \cite{Ej2}, see also p.166 in \cite{ES}). 

To compute $index_E$ and $nullity_E$, we consider a basis of a tangent space of the complex 
isotropic cone by \eqref{cplx-period} derived from 
deformations of a complex structure of $M$ and an action of 
$SO(n,\,\mathbb{C}) \times (\mathbb{C} \setminus \{0\})$ (see p.169 in \cite{ES}). 
The basis are given by periods of the abelian differentials of the second kind. 
Recall that the abelian differentials of the second kind are meromorphic 
differentials with zero residues on a Riemann surface. 
Let $\{T_j\}_{j=1}^{n \gamma}$ be the basis of the tangent space of the complex 
isotropic cone by \eqref{cplx-period} in 
$ K_{n,\,\gamma}\times K_{n,\,\gamma}=K_{n,\,2\gamma}$, and we shall determine $\{T_j\}_{j=1}^{n \gamma}$ for each one-parameter 
family explicitly in the next section. Define 
\[
\eta ((Z_1,\,Z_2),\,(Z'_1,\,Z'_2)) := -i \, {\rm tr} ( (Z_2^t) \overline{Z'_1}-(Z_1^t) \overline{Z'_2}) 
\]
for $(Z_1,\,Z_2),\,(Z'_1,\,Z'_2) \in K_{n,\,\gamma}\times K_{n,\,\gamma}$ and $W := (\eta (T_i,\,T_j))$. 
Note that $\Re \eta$ is a pseudo K\"{a}hler metric given in \cite{Cort} (see \S~{9} in \cite{Ej2}, see also p.167 in \cite{ES}). 
$W$ is the Hermitian matrix of degree $n\gamma$ and it is one of the key matrices as we referred to in the abstract. 
We call a pair of the number of positive eigenvalues and that of 
negative eigenvalues of $W$ counted with multiplicities the {\it signature of a minimal surface}, and 
let $(p,\,q)$ denote the signature of a minimal surface. 
For a decomposition $T_j=(C_j,\,D_j)$ ($C_j,\,D_j \in K_{n,\,\gamma}$), we set 
\begin{align*}
K'_j &:= \Re C_j+i \{\Re C_j \Re\tau -\Re D_j \} (\Im \tau)^{-1}, \\ 
K''_j &:= \Re (i C_j)+i \{\Re (i C_j) \Re\tau -\Re (i D_j) \} (\Im \tau)^{-1}. 
\end{align*}
Define  
\begin{align*}
U_k&:= 
\begin{cases}
T_k & (1\leq k\leq n\gamma)\\
i T_{k-n\gamma} & (n\gamma+1\leq k \leq 2n\gamma)
\end{cases} \\ 
V_k&:= 
\begin{cases}
(K'_k,\,K'_k \tau) & (1\leq k\leq n\gamma)\\
(K''_{k-n\gamma},\,K''_{k-n\gamma} \tau)  & (n\gamma+1\leq k \leq 2n\gamma)
\end{cases}
\end{align*}
and $W_1:=((\Re \eta) (U_i,\,U_j))$, $W_2:=((\Re\eta) (V_i,\,V_j))$. 
Each $W_j$ is a real symmetric matrix of degree $2n\gamma$, and $W_2-W_1$ is another key matrix. 

Now we refer to an algorithm to compute the Morse index and the nullity of a minimal surface, that is, $index_E$ and $nullity_E$ 
for the one-parameter families in the introduction. 
It is well-known that, for the case $\gamma=n=3$, the Riemann surface $M$ is hyperelliptic (see Corollary~{3.2} in \cite{Mee}). 
Since we consider a family of minimal surfaces parametrized by 
$a$ which belongs to a suitable interval $I \subset \mathbb{R}$, 
$W$ and $W_2-W_1$ are also parametrized by $a$. $nullity_E$ is equal to 
the number of zero eigenvalues of $W$ counted with multiplicities (see Theorem 7.10 and \S~{14} in \cite{Ej1}, Theorem~{9.5} in \cite{Ej2}). 
Thus the set $\{a \in I \> | \> \det W=0 \}$ coincides with the set 
$\{a \in I \>|\> nullity_E\neq 0\}$ and it divides $I$ into some intervals. 
Note that $(p,\,q)$ and $index_E$ are constant on each divided interval. 
If the dimension of the zero eigenspace of $W_2-W_1$ is equal to $(2n-4)\gamma+2=8$, 
then $index_E-1$ is equal to the number of negative eigenvalues of $W_2-W_1$ counted with 
multiplicities (see Theorem 7.10, \S~{14} in \cite{Ej1}, Theorem~{9.6} in \cite{Ej2}, see also p.170 in \cite{ES}). 
Every family of minimal surfaces which we treat in this paper satisfies this assumption. 

We will use the notation $T_j$, $W$, $W_j$ below. 

\section{Key matrices}

\subsection{H family}

We start from the following Lemmas. 
\begin{lemma}\label{H-lemma1}
\begin{align}
d(z^{\alpha}w^{\beta})&=\dfrac{1}{2}z^{\alpha-1}w^{\beta-2}\{(2\alpha+7\beta)w^2+3\beta(a^3+1/a^3)z^4-6\beta z\}dz \label{eq:Hdiffs1}\\
&=\dfrac{1}{2}z^{\alpha-1}w^{\beta-2}\{(2\alpha+7\beta)z^7-(2\alpha+4\beta)(a^3+1/a^3)z^4 \label{eq:Hdiffs2}\\
&\nonumber \qquad\qquad\qquad\qquad\qquad\quad\qquad\qquad\qquad 
+(2\alpha+\beta)z\}dz.
\end{align}
\end{lemma}
\begin{proof}
Straightforward calculation. 
\end{proof}
\begin{lemma}\label{H-lemma2}$\,$

For an arbitrary $1${\rm -cycle} $\gamma$, we have the following formulae. 
\begin{align}
\int_{\gamma} \dfrac{z}{w^3}dz&=\dfrac{5}{6}\int_{\gamma}\dfrac{dz}{w}
+\dfrac{1}{2}\left(a^3+\dfrac{1}{a^3}\right)\int_{\gamma}\dfrac{z^4}{w^3}dz, 
\label{eq:Hintegrals1} \\
\int_{\gamma} \dfrac{z^2}{w^3}dz&=\dfrac{1}{2}\int_{\gamma}\dfrac{z}{w}dz
+\dfrac{1}{2}\left(a^3+\dfrac{1}{a^3}\right)\int_{\gamma}\dfrac{z^5}{w^3}dz,\label{eq:Hintegrals2}\\
\int_{\gamma} \dfrac{z^3}{w^3}dz&=\dfrac{1}{6}\int_{\gamma}\dfrac{z^2}{w}dz
+\dfrac{1}{2}\left(a^3+\dfrac{1}{a^3}\right)\int_{\gamma}\dfrac{z^6}{w^3}dz,\label{eq:Hintegrals3}\\
\int_{\gamma} \dfrac{dz}{w^3}&=
\dfrac{2}{3}\left(a^3+\dfrac{1}{a^3}\right)\int_{\gamma}\dfrac{z^2}{w}dz
+\left(2a^6+\dfrac{2}{a^6}-3\right)\int_{\gamma}\dfrac{z^6}{w^3}dz,\label{eq:Hintegrals4}\\
\int_{\gamma} \dfrac{z^7}{w^3}dz&=\dfrac{1}{6}\int_{\gamma}\dfrac{dz}{w}
+\dfrac{1}{2}\left(a^3+\dfrac{1}{a^3}\right)\int_{\gamma}\dfrac{z^4}{w^3}dz,\label{eq:Hintegrals5}\\
\int_{\gamma} \dfrac{z^8}{w^3}dz&=\int_{\gamma}\dfrac{z^2}{w^3}dz,\label{eq:Hintegrals6} \\
\int_{\gamma} \dfrac{z^9}{w^3}dz&=\dfrac{5}{6}\int_{\gamma}\dfrac{z^2}{w}dz
+\dfrac{1}{2}\left(a^3+\dfrac{1}{a^3}\right)\int_{\gamma}\dfrac{z^6}{w^3}dz.\label{eq:Hintegrals7}
\end{align}
\end{lemma}
\begin{proof}
Substituting $\alpha=1$, $\beta=-1$ to \eqref{eq:Hdiffs1} yields \eqref{eq:Hintegrals1}. 
Substituting $\alpha=2$, $\beta=-1$ to \eqref{eq:Hdiffs1} gives \eqref{eq:Hintegrals2}. 
Substituting $\alpha=3$, $\beta=-1$ \eqref{eq:Hdiffs1} implies \eqref{eq:Hintegrals3}. 
Substituting $\alpha=0$, $\beta=-1$ to \eqref{eq:Hdiffs2} yields 
\begin{align*}
\int_{\gamma} \dfrac{dz}{w^3}&=
4\left(a^3+\dfrac{1}{a^3}\right)\int_{\gamma}\dfrac{z^3}{w^3}dz
-7\int_{\gamma}\dfrac{z^6}{w^3}dz
\end{align*}
\begin{align*}
 &\underbrace{=}_{\eqref{eq:Hintegrals3}}
\dfrac{2}{3}\left(a^3+\dfrac{1}{a^3}\right)\int_{\gamma}\dfrac{z^2}{w}dz
+\left(2a^6+\dfrac{2}{a^6}-3\right)\int_{\gamma}\dfrac{z^6}{w^3}dz.
\end{align*}
Substituting $\alpha=1$, $\beta=-1$ to \eqref{eq:Hdiffs2} gives 
\begin{align*}
\int_{\gamma} \dfrac{z^7}{w^3}dz&=\dfrac{1}{5}\int_{\gamma}\dfrac{z}{w^3}dz
+\dfrac{2}{5}\left(a^3+\dfrac{1}{a^3}\right)\int_{\gamma}\dfrac{z^4}{w^3}dz \\
&\underbrace{=}_{\eqref{eq:Hintegrals1}}\dfrac{1}{6}\int_{\gamma}\dfrac{dz}{w}
+\dfrac{1}{2}\left(a^3+\dfrac{1}{a^3}\right)\int_{\gamma}\dfrac{z^4}{w^3}dz.
\end{align*}
Substituting $\alpha=2$, $\beta=-1$ to \eqref{eq:Hdiffs2} implies \eqref{eq:Hintegrals6}. 
Finally, we have 
\begin{align*}
\int_{\gamma} \dfrac{z^9}{w^3}dz&=
\int_{\gamma} \dfrac{z^2}{w^3}\left\{w^2+ 
\left(a^3+\dfrac{1}{a^3}\right)z^4-z \right\}dz\\
&=\int_{\gamma}\dfrac{z^2}{w}dz
+\left(a^3+\dfrac{1}{a^3}\right)\int_{\gamma}\dfrac{z^6}{w^3}dz
-\int_{\gamma}\dfrac{z^3}{w^3}dz\\
&\underbrace{=}_{\eqref{eq:Hintegrals3}}\dfrac{5}{6}\int_{\gamma}\dfrac{z^2}{w}dz
+\dfrac{1}{2}\left(a^3+\dfrac{1}{a^3}\right)\int_{\gamma}\dfrac{z^6}{w^3}dz.
\end{align*}
\end{proof}
Next we consider the basis of the tangent space of the complex Lagrangian cone via 
the complex period matrix. First we consider deformations of 
the complex structure of $M$. 
Let $a_1,\,a_2,\,\dots,\,a_6$ be six points in $\mathbb{C}\setminus \{0\}$ 
satisfying 
\[
w^2=z(z^3-a^3)\left(z^3-\dfrac{1}{a^3}\right)=z(z-a_1)(z-a_2)\cdots (z-a_6). 
\]
We now deform $M$ as follows: 
$w^2=z(z-z_1)(z-z_2)\cdots (z-z_6)$. 
Since 
\begin{align*}
2 w \dfrac{\partial w}{\partial z_i}\bigg|_{(z_1,\,\dots,\,z_6)=(a_1,\,\dots,\,a_6)} &= 
-z(z-a_1)\cdots (z-a_{i-1})(z-a_{i+1})\cdots (z-a_6) 
= -\dfrac{w^2}{z-a_i}, 
\end{align*}
we have 
\begin{align*}
\dfrac{\partial}{\partial z_i}\bigg|_{(z_1,\,\dots,\,z_6)=(a_1,\,\dots,\,a_6)} \int_{\gamma } \dfrac{1-z^2}{w}dz =
\int_{\gamma} \dfrac{\partial}{\partial z_i}\bigg|_{z=a_i}\dfrac{1-z^2}{w}dz =
\dfrac{1}{2} \int_{\gamma} \dfrac{1-z^2}{w (z-a_i)}dz
\end{align*}
for an arbitrary $1$-cycle $\gamma$. Applying the same technique, we find 
\begin{align*}
\dfrac{\partial}{\partial z_i}&\bigg|_{(z_1,\,\dots,\,z_6)=(a_1,\,\dots,\,a_6)}
\left( 
\int_{A_1}
\begin{pmatrix}
\dfrac{1-z^2}{w}dz  \\
\dfrac{i(1+z^2)}{w}dz  \\
\dfrac{2 z}{w}dz 
\end{pmatrix},\, 
\dots ,\, 
\int_{B_3}
\begin{pmatrix}
\dfrac{1-z^2}{w}dz  \\
\dfrac{i(1+z^2)}{w}dz  \\
\dfrac{2 z}{w}dz 
\end{pmatrix} 
\right) 
\end{align*}
\begin{align*}
&=
\dfrac{1}{2} \left(
\int_{A_1}
\begin{pmatrix}
\dfrac{1-z^2}{w (z-a_i)}dz  \\
\dfrac{i(1+z^2)}{w (z-a_i)}dz  \\
\dfrac{2 z}{w (z-a_i)}dz 
\end{pmatrix},\, \dots,\, 
\int_{B_3}
\begin{pmatrix}
\dfrac{1-z^2}{w (z-a_i)}dz  \\
\dfrac{i(1+z^2)}{w (z-a_i)}dz  \\
\dfrac{2 z}{w (z-a_i)}dz 
\end{pmatrix}
\right). 
\end{align*}
To describe 
\[
\int_{\gamma}\dfrac{1-z^2}{w (z-a_i)}dz, \quad \int_{\gamma}\dfrac{i(1+z^2)}{w (z-a_i)}dz, \quad 
\int_{\gamma}\dfrac{2 z}{w (z-a_i)}dz 
\]
via periods of the abelian differentials of the second kind, we set 
\begin{align*}
w^2&= z(z-a_i)(z^5+\alpha_{i 4} z^4 +\alpha_{i 3} z^3+\alpha_{i 2} z^2+\alpha_{i 1} z+\alpha_{i 0} ) \\
   &= z(z^3-a^3)\left(z^3-\dfrac{1}{a^3}\right), 
\end{align*}
that is, 
\[
\alpha_{i 0}=-\dfrac{1}{a_i},\>\> \alpha_{i 1}=-\dfrac{1}{a_i^2},\>\>
\alpha_{i 2}=-\dfrac{1}{a_i^3},\>\> \alpha_{i 3}=a_i^2,\>\> \alpha_{i 4}=a_i. 
\]
By Lemma~\ref{H-lemma2}, we have 
\begin{align*}
&\int_{\gamma} \dfrac{dz}{w(z-a_i)} = \int_{\gamma} \dfrac{1}{w^3}z(z^5+\alpha_{i 4} z^4 +\alpha_{i 3} z^3+\alpha_{i 2} z^2
+\alpha_{i 1} z+\alpha_{i 0}) dz \\
&=\dfrac{5}{6} \alpha_{i 0} \int_{\gamma} \dfrac{dz}{w}+ \dfrac{1}{2} \alpha_{i 1} \int_{\gamma} \dfrac{z}{w}dz +
\dfrac{1}{6} \alpha_{i 2} \int_{\gamma} \dfrac{z^2}{w}dz \\
&+ \left\{\dfrac{\alpha_{i 0}}{2} \left( a^3+\dfrac{1}{a^3} \right)  + \alpha_{i 3} \right\} \int_{\gamma} \dfrac{z^4}{w^3}dz 
+ \left\{\dfrac{\alpha_{i 1}}{2} \left( a^3+\dfrac{1}{a^3} \right) + \alpha_{i 4} \right\} \int_{\gamma} \dfrac{z^5}{w^3}dz \\
&+ \left\{\dfrac{\alpha_{i 2}}{2} \left( a^3+\dfrac{1}{a^3} \right) + 1 \right\} \int_{\gamma} \dfrac{z^6}{w^3}dz, \\
&\int_{\gamma} \dfrac{z}{w(z-a_i)} dz = 
\dfrac{1}{6} \int_{\gamma} \dfrac{dz}{w}+ \dfrac{1}{2} \alpha_{i 0} \int_{\gamma} \dfrac{z}{w}dz +
\dfrac{1}{6} \alpha_{i 1} \int_{\gamma} \dfrac{z^2}{w}dz \\
&+ \left\{\dfrac{1}{2} \left( a^3+\dfrac{1}{a^3} \right) + \alpha_{i 2} \right\} \int_{\gamma} \dfrac{z^4}{w^3}dz 
+ \left\{\dfrac{\alpha_{i 0}}{2} \left( a^3+\dfrac{1}{a^3} \right)  + \alpha_{i 3} \right\} \int_{\gamma} \dfrac{z^5}{w^3}dz \\
& + \left\{\dfrac{\alpha_{i 1}}{2} \left( a^3+\dfrac{1}{a^3} \right)  + \alpha_{i 4} \right\} \int_{\gamma} \dfrac{z^6}{w^3}dz, \\
&\int_{\gamma} \dfrac{z^2}{w(z-a_i)} dz = 
\dfrac{1}{6} \alpha_{i 4} \int_{\gamma} \dfrac{dz}{w}+ \dfrac{1}{2} \int_{\gamma} \dfrac{z}{w}dz +
\dfrac{1}{6} \alpha_{i 0} \int_{\gamma} \dfrac{z^2}{w}dz \\
&+ \left\{\dfrac{\alpha_{i 4}}{2} \left( a^3+\dfrac{1}{a^3} \right) + \alpha_{i 1} \right\} \int_{\gamma} \dfrac{z^4}{w^3}dz 
+ \left\{\dfrac{1}{2} \left( a^3+\dfrac{1}{a^3} \right) + \alpha_{i 2} \right\} \int_{\gamma} \dfrac{z^5}{w^3}dz \\
& + \left\{\dfrac{\alpha_{i 0}}{2} \left( a^3+\dfrac{1}{a^3} \right) + \alpha_{i 3} \right\} \int_{\gamma} \dfrac{z^6}{w^3}dz.
\end{align*}
Thus, setting
\begin{align*}
P_1 & = 
\begin{pmatrix}
1 & 0 & -1 \\
i & 0 & i \\
0 & 2 & 0
\end{pmatrix}, \quad 
P_2 = \begin{pmatrix}
\dfrac{1}{2} & -\dfrac{i}{2} & 0 & 0 & 0 & 0 \\
0 & 0 & \dfrac{1}{2} & 0 & 0 & 0 \\
-\dfrac{1}{2} & -\dfrac{i}{2} & 0 & 0 & 0 & 0 \\
0 & 0 & 0 & \dfrac{1}{2} & -\dfrac{i}{2} & 0 \\
0 & 0 & 0 & 0 & 0 & 1 \\
0 & 0 & 0 & -\dfrac{1}{2} & -\dfrac{i}{2} & 0 
\end{pmatrix}, 
\end{align*}
and 
\begin{align*}
P_{a_i} &= 
\begin{pmatrix}
-\dfrac{5}{6 a_i} & -\dfrac{1}{2 a_i^2} & -\dfrac{1}{6 a_i^3} & \dfrac{1}{2} \left( a_i^2-\dfrac{1}{a_i^4} \right) 
& \dfrac{1}{2} \left( a_i-\dfrac{1}{a_i^5} \right) & \dfrac{1}{2} \left( 1-\dfrac{1}{a_i^6} \right) \\
\dfrac{1}{6} & -\dfrac{1}{2 a_i} & -\dfrac{1}{6 a_i^2} & \dfrac{1}{2} \left( a_i^3-\dfrac{1}{a_i^3} \right) 
& \dfrac{1}{2} \left( a_i^2-\dfrac{1}{a_i^4} \right) & \dfrac{1}{2} \left( a_i-\dfrac{1}{a_i^5} \right) \\
\dfrac{a_i}{6} & \dfrac{1}{2} & -\dfrac{1}{6 a_i} & \dfrac{1}{2} \left(a_i^4-\dfrac{1}{a_i^2} \right) 
& \dfrac{1}{2} \left( a_i^3-\dfrac{1}{a_i^3} \right) & \dfrac{1}{2}\left( a_i^2-\dfrac{1}{a_i^4} \right)
\end{pmatrix}, 
\end{align*}
we find 
\[
\begin{pmatrix}
\int_{\gamma}\dfrac{1-z^2}{w (z-a_i)}dz      \\
\int_{\gamma}\dfrac{i(1+z^2)}{w (z-a_i)}dz   \\
\int_{\gamma}\dfrac{2 z}{w (z-a_i)}dz
\end{pmatrix}
=P_1 P_{a_i} P_2
\begin{pmatrix}
\int_{\gamma}\dfrac{1-z^2}{w}  dz    \\
\int_{\gamma}\dfrac{i(1+z^2)}{w} dz     \\
\int_{\gamma}\dfrac{2 z}{w} dz     \\
\int_{\gamma}\dfrac{z^4-z^6}{w^3} dz     \\
\int_{\gamma}\dfrac{i(z^4+z^6)}{w^3} dz    \\
\int_{\gamma}\dfrac{z^5}{w^3} dz     
\end{pmatrix}. 
\]
Let $\Omega_{{\rm H}}$ be the complex period matrix of the abelian differentials of the second kind (see \S~{\ref{H-detail}}), that is, 
\[
i
\begin{pmatrix}
0  &  \dfrac{\sqrt{3}}{2} (A+i B)  &   0  &  -\sqrt{3} A  &  -2\sqrt{3} A  &  -\sqrt{3} A  \\
2 A  &   \dfrac{-3 A+ i B}{2}   &  A- i B  & i B  &  0   &  A  \\
-i D  &   -C  &   2C+i D  &  C  &  0  &  i D  \\
0  &  -\dfrac{\sqrt{3}}{2} (E+i F)   &  0  &  \sqrt{3} E  &  2\sqrt{3} E  &  \sqrt{3} E  \\
-2 E  &   \dfrac{3 E- i F}{2}   &   -E+iF  &  -i F  &  0  &  -E  \\
i I  &  H   &   -2H-i I   &  -H  &  0  &  -i I
\end{pmatrix}.
\]
Then, choosing $a_1=a$, $a_2=e^{\frac{2}{3}\pi i} a$, $a_3=e^{\frac{4}{3}\pi i} a$, $a_4=1/a$, 
and $a_5=e^{\frac{2}{3}\pi i}/a $, 
we obtain five tangent vectors $\{T_i\}_{i=1}^5:= \left\{\frac{1}{2} P_1 P_{a_i} P_2 \Omega_{{\rm H}}\right\}_{i=1}^5$ in $K_{3,\,6}$. 
Moreover, we shall consider the tangent vectors via an action of 
$SO(3,\,\mathbb{C}) \times (\mathbb{C} \setminus \{0\})$. 
Setting $C_1$ and $C_2$ are complex matrices of degree $3$ given by 
\[
(C_1,\,C_2)=
i
\begin{pmatrix}
0  &  \dfrac{\sqrt{3}}{2} (A+i B)  &   0  &  -\sqrt{3} A  &  -2\sqrt{3} A  &  -\sqrt{3} A  \\
2 A  &   \dfrac{-3 A+ i B}{2}   &  A- i B  & i B  &  0   &  A  \\
-i D  &   -C  &   2C+i D  &  C  &  0  &  i D  
\end{pmatrix}, 
\]
we have the Riemann matrix $\tau=C_1^{-1}C_2$ and 
the following four tangent vectors in $K_{3,\,6}$: 
\begin{align*}
&T_6:=(C_1,\,C_2),\quad 
T_7:=\begin{pmatrix}
0 & 1 & 0 \\
-1 & 0 & 0 \\
0 & 0 & 0
\end{pmatrix}
(C_1,\,C_2), \\
&T_8:=\begin{pmatrix}
0 & 0 & 1 \\
0 & 0 & 0 \\
-1 & 0 & 0
\end{pmatrix}
(C_1,\,C_2),\quad 
T_9:=\begin{pmatrix}
0 & 0 & 0 \\
0 & 0 & 1 \\
0 & -1 & 0
\end{pmatrix}
(C_1,\,C_2). 
\end{align*}

\subsection{rPD family}

We begin by giving some lemmas. Choosing $i z$, $e^{-\frac{\pi}{4}i}w$, and $i a$ instead of $z$, $w$, and $a$ in 
Lemma~\ref{H-lemma1} 
and Lemma~\ref{H-lemma2}, respectively, we find the following two Lemmas. 
\begin{lemma}\label{rPD-lemma1}
\begin{align}
&d(z^{\alpha} w^{\beta})
=\dfrac{z^{\alpha-1} w^{\beta-2}}{2} 
\left\{ (2\alpha +7\beta) w^2 +3 \beta \left(a^3-\dfrac{1}{a^3}\right)
z^4+6\beta z\right\}dz \label{eq:TTexact1}\\
&=
\dfrac{z^{\alpha} w^{\beta-2}}{2} 
\left\{ (2\alpha +7\beta) z^6 +2(\alpha+2\beta)\left(-a^3+\dfrac{1}{a^3}\right)
z^3-(2\alpha+\beta)
\right\}dz \label{eq:TTexact2}
\end{align}
\end{lemma}
\begin{lemma}\label{rPD-lemma2}
For an arbitrary $1${\rm -cycle} $\gamma$, we have 
\begin{align}
\int_{\gamma} \dfrac{z}{w^3}dz &= -\dfrac{5}{6} \int_{\gamma} \dfrac{dz}{w}
+\dfrac{1}{2}\left(-a^3+\dfrac{1}{a^3}\right)\int_{\gamma} \dfrac{z^4}{w^3}dz, 
\label{eq:TTintegral1} \\
\int_{\gamma} \dfrac{z^2}{w^3}dz &= -\dfrac{1}{2} \int_{\gamma} \dfrac{z}{w}dz
+\dfrac{1}{2}\left(-a^3+\dfrac{1}{a^3}\right)\int_{\gamma} \dfrac{z^5}{w^3}dz, 
\label{eq:TTintegral2}\\
\int_{\gamma} \dfrac{z^3}{w^3}dz &= -\dfrac{1}{6} \int_{\gamma} \dfrac{z^2}{w}dz
+\dfrac{1}{2}\left(-a^3+\dfrac{1}{a^3}\right)\int_{\gamma} \dfrac{z^6}{w^3}dz, 
\label{eq:TTintegral3} \\
\int_{\gamma} \dfrac{dz}{w^3} &= \dfrac{2}{3} 
\left(a^3-\dfrac{1}{a^3}\right) \int_{\gamma} \dfrac{z^2}{w}dz
+\left(2 a^6+\dfrac{2}{a^6}+3\right)\int_{\gamma} \dfrac{z^6}{w^3}dz, 
\label{eq:TTintegral4} \\
\int_{\gamma} \dfrac{z^7}{w^3}dz &= \dfrac{1}{6} \int_{\gamma} \dfrac{dz}{w}
+\dfrac{1}{2}\left(a^3-\dfrac{1}{a^3}\right)\int_{\gamma} \dfrac{z^4}{w^3}dz, 
\label{eq:TTintegral5} \\
\int_{\gamma} \dfrac{z^8}{w^3}dz &= -\int_{\gamma} \dfrac{z^2}{w^3}dz, 
\label{eq:TTintegral6} 
\end{align}
\begin{align}
\int_{\gamma} \dfrac{z^9}{w^3}dz &= \dfrac{5}{6} \int_{\gamma} \dfrac{z^2}{w}dz
+\dfrac{1}{2}\left(a^3-\dfrac{1}{a^3}\right)\int_{\gamma} \dfrac{z^6}{w^3}dz. 
\label{eq:TTintegral7}
\end{align}
\end{lemma}
Next we consider the basis of the tangent space of the complex Lagrangian cone via 
the complex period matrix. First we consider deformations of 
the complex structure of $M$. 
Let $a_1,\,a_2,\,\dots,\,a_6$ be six points in $\mathbb{C} \setminus \{0\}$ 
satisfying  
\[
w^2=z(z^3-a^3)\left(z^3+\dfrac{1}{a^3}\right)=z(z-a_1)(z-a_2)\cdots (z-a_6). 
\]
We now deform $M$ as follows: $w^2=z(z-z_1)(z-z_2)\cdots (z-z_6)$. 
Then, for an arbitrary $1$-cycle $\gamma$, we find 
\begin{align*}
\dfrac{\partial}{\partial z_i}&\bigg|_{(z_1,\,\dots,\,z_6)=(a_1,\,\dots,\,a_6)}
\left( 
\int_{A_1}
\begin{pmatrix}
\dfrac{1-z^2}{w}dz  \\
\dfrac{i(1+z^2)}{w}dz  \\
\dfrac{2 z}{w}dz 
\end{pmatrix},\, 
\dots ,\, 
\int_{B_3}
\begin{pmatrix}
\dfrac{1-z^2}{w}dz  \\
\dfrac{i(1+z^2)}{w}dz  \\
\dfrac{2 z}{w}dz 
\end{pmatrix} 
\right) \\
&=
\dfrac{1}{2} \left(
\int_{A_1}
\begin{pmatrix}
\dfrac{1-z^2}{w (z-a_i)}dz  \\
\dfrac{i(1+z^2)}{w (z-a_i)}dz  \\
\dfrac{2 z}{w (z-a_i)}dz 
\end{pmatrix},\, \dots,\, 
\int_{B_3}
\begin{pmatrix}
\dfrac{1-z^2}{w (z-a_i)}dz  \\
\dfrac{i(1+z^2)}{w (z-a_i)}dz  \\
\dfrac{2 z}{w (z-a_i)}dz 
\end{pmatrix}
\right). 
\end{align*}
To describe 
\[
\int_{\gamma}\dfrac{1-z^2}{w (z-a_i)}dz, \quad \int_{\gamma}\dfrac{i(1+z^2)}{w (z-a_i)}dz, \quad 
\int_{\gamma}\dfrac{2 z}{w (z-a_i)}dz 
\]
via periods of the abelian differentials of the second kind, we set 
\begin{align*}
w^2&= z(z-a_i)(z^5+\alpha_{i 4} z^4 +\alpha_{i 3} z^3+\alpha_{i 2} z^2+\alpha_{i 1} z+\alpha_{i 0} ) 
= z(z^3-a^3)\left(z^3+\dfrac{1}{a^3}\right), 
\end{align*}
that is, 
\[
\alpha_{i 0}=\dfrac{1}{a_i},\>\> \alpha_{i 1}=\dfrac{1}{a_i^2},\>\>
\alpha_{i 2}=\dfrac{1}{a_i^3},\>\> \alpha_{i 3}=a_i^2,\>\> \alpha_{i 4}=a_i. 
\]
Thus, setting 
\begin{align*}
P_1 & = 
\begin{pmatrix}
1 & 0 & -1 \\
i & 0 & i \\
0 & 2 & 0
\end{pmatrix}, \quad 
P_2 = \begin{pmatrix}
\dfrac{1}{2} & -\dfrac{i}{2} & 0 & 0 & 0 & 0 \\
0 & 0 & \dfrac{1}{2} & 0 & 0 & 0 \\
-\dfrac{1}{2} & -\dfrac{i}{2} & 0 & 0 & 0 & 0 \\
0 & 0 & 0 & \dfrac{1}{2} & -\dfrac{i}{2} & 0 \\
0 & 0 & 0 & 0 & 0 & 1 \\
0 & 0 & 0 & -\dfrac{1}{2} & -\dfrac{i}{2} & 0 
\end{pmatrix}, 
\end{align*}
\begin{align*}
P_{a_i} &= 
\begin{pmatrix}
-\dfrac{5}{6 a_i} & -\dfrac{1}{2 a_i^2} & -\dfrac{1}{6 a_i^3} & \dfrac{1}{2} \left( a_i^2+\dfrac{1}{a_i^4} \right) 
& \dfrac{1}{2} \left( a_i+\dfrac{1}{a_i^5} \right) & \dfrac{1}{2} \left( 1+\dfrac{1}{a_i^6} \right) \\
\dfrac{1}{6} & -\dfrac{1}{2 a_i} & -\dfrac{1}{6 a_i^2} & \dfrac{1}{2} \left( a_i^3+\dfrac{1}{a_i^3} \right) & 
\dfrac{1}{2} \left( a_i^2+\dfrac{1}{a_i^4} \right)  & \dfrac{1}{2} \left( a_i+\dfrac{1}{a_i^5} \right) \\
\dfrac{a_i}{6} & \dfrac{1}{2} & -\dfrac{1}{6 a_i} & \dfrac{1}{2} \left( a_i^4+\dfrac{1}{a_i^2} \right)  & 
\dfrac{1}{2} \left( a_i^3+\dfrac{1}{a_i^3} \right)  & \dfrac{1}{2} \left( a_i^2+\dfrac{1}{a_i^4} \right) 
\end{pmatrix}, 
\end{align*}
we find 
\[
\begin{pmatrix}
\int_{\gamma}\dfrac{1-z^2}{w (z-a_i)}dz      \\
\int_{\gamma}\dfrac{i(1+z^2)}{w (z-a_i)}dz   \\
\int_{\gamma}\dfrac{2 z}{w (z-a_i)}dz
\end{pmatrix}
=P_1 P_{a_i} P_2 
\begin{pmatrix}
\int_{\gamma}\dfrac{1-z^2}{w}  dz    \\
\int_{\gamma}\dfrac{i(1+z^2)}{w} dz     \\
\int_{\gamma}\dfrac{2 z}{w} dz     \\
\int_{\gamma}\dfrac{z^4-z^6}{w^3} dz     \\
\int_{\gamma}\dfrac{i(z^4+z^6)}{w^3} dz    \\
\int_{\gamma}\dfrac{z^5}{w^3} dz     
\end{pmatrix} 
\]
by Lemma~\ref{rPD-lemma2}. 
Let $\Omega_{{\rm rPD}}$ be the period matrix of the abelian differentials of the second 
kind (see \S~{\ref{rPD-detail}}), that is, 
\[i
\begin{pmatrix}
2 i B   &   -2 (A+i B)  &  -(A+i B)  & 2 A  &   3(A-i B)   &  2(A-i B) \\
-2\sqrt{3} A  &   0   &  \sqrt{3}  (A+i B)  &  -2\sqrt{3} i B  &  \sqrt{3} (A-i B)  &  0 \\
i D   &   C- i D  &  -C + i D  &  -C  &  0  &  -(C+i D) \\
- 2 i F   &   2 (-E+ i F)  &  -E+i F  &  2 E  & 3(E+i F)  &   2(E+i F) \\ 
-2\sqrt{3} E  &  0   &  \sqrt{3} (E-i F)  &  2\sqrt{3}  i  F  & \sqrt{3} (E+i F)  &  0  \\ 
i I  &   H - i I  &   -H+i I   &  -H  &  0  &  -(H+i I)
\end{pmatrix}. 
\]
Then, choosing $a_1=a$, $a_2=e^{\frac{2}{3}\pi i} a$, $a_3=e^{\frac{4}{3}\pi i} a$, $a_4=-1/a$, 
and $a_5=-e^{\frac{2}{3}\pi i}/a $, 
we obtain five tangent vectors $\{T_i\}_{i=1}^5:= \left\{\frac{1}{2} P_1 P_{a_i} P_2 \Omega_{{\rm rPD}}\right\}_{i=1}^5$ in $K_{3,\,6}$. 
Moreover, we shall consider the tangent vectors via an action of 
$SO(3,\,\mathbb{C}) \times (\mathbb{C} \setminus \{0\})$. 
Setting $C_1$ and $C_2$ are complex matrices of degree $3$ given by 
\[
(C_1,\,C_2)=
i
\begin{pmatrix}
2 i B   &   -2 (A+i B)  &  -(A+i B)  & 2 A  &   3(A-i B)   &  2(A-i B) \\
-2\sqrt{3} A  &   0   &  \sqrt{3}  (A+i B)  &  -2\sqrt{3} i B  &  \sqrt{3} (A-i B)  &  0 \\
i D   &   C- i D  &  -C + i D  &  -C  &  0  &  -(C+i D) 
\end{pmatrix}, 
\]
we have the Riemann matrix $\tau=C_1^{-1} C_2$ and 
the following four tangent vectors in $K_{3,\,6}$: 
\begin{align*}
&T_6:=(C_1,\,C_2),\quad 
T_7:=\begin{pmatrix}
0 & 1 & 0 \\
-1 & 0 & 0 \\
0 & 0 & 0
\end{pmatrix}
(C_1,\,C_2),\\
&T_8:=\begin{pmatrix}
0 & 0 & 1 \\
0 & 0 & 0 \\
-1 & 0 & 0
\end{pmatrix}
(C_1,\,C_2),\quad 
T_9:=\begin{pmatrix}
0 & 0 & 0 \\
0 & 0 & 1 \\
0 & -1 & 0
\end{pmatrix}
(C_1,\,C_2). 
\end{align*}

\subsection{tP family, tD family}\label{tP-section}

Recall that tD family is parametrized by $a\in (-\infty,\,-2)$ instead of 
$a \in (2,\,\infty)$. 
Using the reparametrization $(u,\,v)=(e^{-\frac{\pi}{4} i} z,\, w)$, we transform tD family to a family of conjugate 
surfaces of minimal surfaces which belong to tP family. 
Note that the Morse index and the nullity of a minimal surface depend only on 
its Gauss map (see for instance \cite{Mont-Ros}), and the signature of a minimal surface coincides with 
that of the associated surface (Theorem~{9.1} in \cite{Ej2}). 
Hence 
the Morse index of a minimal surface is equal to that of its conjugate surface and so are the nullity 
and the signature. 
Thus we treat only tP family. 

We begin by giving some lemmas. 
\begin{lemma}\label{tP-lemma1}
\begin{align}
d(z^{\alpha}w^{\beta})&=z^{\alpha-1}w^{\beta-2}\{(\alpha+4\beta)z^8
+a (\alpha+2\beta)z^4+\alpha\}dz \label{eq:tPdiffs1} \\
&=z^{\alpha-1}w^{\beta-2}\{(\alpha+4\beta)w^2-2 a \beta z^4-4\beta\}dz \label{eq:tPdiffs2}.
\end{align}
\end{lemma}
\begin{proof}
Straightforward calculation. 
\end{proof}
\begin{lemma}\label{tP-lemma2}
For an arbitrary $1${\rm -cycle} $\gamma$, we have 
\begin{align}
\int_{\gamma}\dfrac{dz}{w^3}&=\dfrac{3}{4}\int_{\gamma}\dfrac{dz}{w}
-\dfrac{a}{2} \int_{\gamma}z^4\dfrac{dz}{w^3}, \label{eq:tPintegrals1} \\
\int_{\gamma}z\dfrac{dz}{w^3}&=\dfrac{1}{2}\int_{\gamma}z\dfrac{dz}{w}
-\dfrac{a}{2} \int_{\gamma}z^5\dfrac{dz}{w^3}, \label{eq:tPintegrals2} \\
\int_{\gamma}z^2\dfrac{dz}{w^3}&=\dfrac{1}{4}\int_{\gamma}z^2\dfrac{dz}{w}
-\dfrac{a}{2} \int_{\gamma}z^6\dfrac{dz}{w^3}, \label{eq:tPintegrals3} \\
\int_{\gamma}z^3\dfrac{dz}{w^3}&=\int_{\gamma}z^7\dfrac{dz}{w^3}=0, \label{eq:tPintegrals4} \\
\int_{\gamma}z^8\dfrac{dz}{w^3}&=\dfrac{1}{4}\int_{\gamma}\dfrac{dz}{w}
-\dfrac{a}{2}\int_{\gamma}z^4\dfrac{dz}{w^3}, \label{eq:tPintegrals5} \\
\int_{\gamma}z^9\dfrac{dz}{w^3}&=\int_{\gamma}z\dfrac{dz}{w^3}, \label{eq:tPintegrals6} \\
\int_{\gamma}z^{10}\dfrac{dz}{w^3}&=\dfrac{3}{4}\int_{\gamma}z^2,
\dfrac{dz}{w}-\dfrac{a}{2}\int_{\gamma}z^6\dfrac{dz}{w^3}.\label{eq:tPintegrals7} 
\end{align}
\end{lemma}
\begin{proof}
Substituting $\alpha=1$, $\beta=-1$ to \eqref{eq:tPdiffs2} yields \eqref{eq:tPintegrals1}. 
Substituting $\alpha=2$, $\beta=-1$ to \eqref{eq:tPdiffs2} gives \eqref{eq:tPintegrals2}. 
Substituting $\alpha=3$, $\beta=-1$ \eqref{eq:tPdiffs2} implies \eqref{eq:tPintegrals3}. 
Substituting $(\alpha,\,\beta)=(0,\,-1),\,(4,\,-1)$ to \eqref{eq:tPdiffs1} yields \eqref{eq:tPintegrals4}. 
Substituting $\alpha=1$, $\beta=-1$ to \eqref{eq:tPdiffs1} and \eqref{eq:tPintegrals1} give \eqref{eq:tPintegrals5}. 
Substituting $\alpha=2$, $\beta=-1$ to \eqref{eq:tPdiffs1} implies \eqref{eq:tPintegrals6}. 
Combining the equation $w^2=z^8+a z^4+1$ and \eqref{eq:tPintegrals3} yields \eqref{eq:tPintegrals7}. 
\end{proof}
Next we consider the basis of the tangent space of the complex Lagrangian cone via 
the complex period matrix. First we consider deformations of 
the complex structure of $M$. 
We assume $a\in (2,\,\infty)$. Let $a_1,\,a_2,\,\dots,\,a_8$ be eight points in 
$\mathbb{C} \setminus \{0\}$: 
\[
w^2=z^8+a z^4+1=(z-a_1)(z-a_2)\cdots (z-a_8). 
\]
We now deform $M$ as follows: $w^2=(z-z_1)(z-z_2)\cdots (z-z_8)$. 
Then, for an arbitrary $1$-cycle $\gamma$, we find 
\begin{align*}
\dfrac{\partial}{\partial z_i}&\bigg|_{(z_1,\,\dots,\,z_8)=(a_1,\,\dots,\,a_8)}
\left( 
\int_{A_1}
\begin{pmatrix}
\dfrac{1-z^2}{w}dz  \\
\dfrac{i(1+z^2)}{w}dz  \\
\dfrac{2 z}{w}dz 
\end{pmatrix},\, 
\dots ,\, 
\int_{B_3}
\begin{pmatrix}
\dfrac{1-z^2}{w}dz  \\
\dfrac{i(1+z^2)}{w}dz  \\
\dfrac{2 z}{w}dz 
\end{pmatrix} 
\right) \\
&=
\dfrac{1}{2} \left(
\int_{A_1}
\begin{pmatrix}
\dfrac{1-z^2}{w (z-a_i)}dz  \\
\dfrac{i(1+z^2)}{w (z-a_i)}dz  \\
\dfrac{2 z}{w (z-a_i)}dz 
\end{pmatrix},\, \dots,\, 
\int_{B_3}
\begin{pmatrix}
\dfrac{1-z^2}{w (z-a_i)}dz  \\
\dfrac{i(1+z^2)}{w (z-a_i)}dz  \\
\dfrac{2 z}{w (z-a_i)}dz 
\end{pmatrix}
\right). 
\end{align*}
To describe 
\[
\int_{\gamma}\dfrac{1-z^2}{w (z-a_i)}dz, \quad \int_{\gamma}\dfrac{i(1+z^2)}{w (z-a_i)}dz, \quad 
\int_{\gamma}\dfrac{2 z}{w (z-a_i)}dz 
\]
via periods of the abelian differentials of the second kind, we set 
\begin{align*}
w^2&= (z-a_i)(z^7+\alpha_{i 6} z^6+\alpha_{i 5} z^5 +\alpha_{i 4} z^4+\alpha_{i 3} z^3+\alpha_{i 2} z^2+\alpha_{i 1} z+\alpha_{i 0}) \\
   &= z^8+a z^4+1, 
\end{align*}
that is, 
\begin{align*}
\alpha_{i 0} &= -\dfrac{1}{a_i},\>\> \alpha_{i 1}=-\dfrac{1}{a_i^2},\>\> \alpha_{i 2}=-\dfrac{1}{a_i^3},\>\>
\alpha_{i 3}=-\dfrac{1}{a_i^4},\\
\alpha_{i 4} &= a_i^3,\>\> \alpha_{i 5}=a_i^2,\>\> \alpha_{i 6}=a_i. 
\end{align*}
Thus, setting 
\begin{align*}
P_1 & = 
\begin{pmatrix}
1 & 0 & -1 \\
i & 0 & i \\
0 & 2 & 0
\end{pmatrix}, \quad 
P_2 = \begin{pmatrix}
\dfrac{1}{2} & -\dfrac{i}{2} & 0 & 0 & 0 & 0 \\
0 & 0 & \dfrac{1}{2} & 0 & 0 & 0 \\
-\dfrac{1}{2} & -\dfrac{i}{2} & 0 & 0 & 0 & 0 \\
0 & 0 & 0 & \dfrac{1}{2} & -\dfrac{i}{2} & 0 \\
0 & 0 & 0 & 0 & 0 & 1 \\
0 & 0 & 0 & -\dfrac{1}{2} & -\dfrac{i}{2} & 0 
\end{pmatrix}, 
\end{align*}
and 
\begin{align*}
P_{a_i} &= 
\begin{pmatrix}
-\dfrac{3}{4 a_i} & -\dfrac{1}{2 a_i^2} & -\dfrac{1}{4 a_i^3} & \dfrac{a}{2 a_i} + a_i^3 
& \dfrac{a}{2 a_i^2} + a_i^2  & \dfrac{a}{2 a_i^3} + a_i \\
\dfrac{1}{4} & -\dfrac{1}{2 a_i} & -\dfrac{1}{4 a_i^2} & -\dfrac{a}{2} - \dfrac{1}{a_i^4} & 
\dfrac{a}{2 a_i} + a_i^3 & \dfrac{a}{2 a_i^2} + a_i^2 \\
\dfrac{a_i}{4} & \dfrac{1}{2} & -\dfrac{1}{4 a_i} & -\dfrac{a}{2} a_i -\dfrac{1}{a_i^3}   & 
-\dfrac{a}{2} - \dfrac{1}{a_i^4}  & \dfrac{a}{2 a_i} + a_i^3  
\end{pmatrix}, 
\end{align*}
we find 
\[
\begin{pmatrix}
\int_{\gamma}\dfrac{1-z^2}{w (z-a_i)}dz      \\
\int_{\gamma}\dfrac{i(1+z^2)}{w (z-a_i)}dz   \\
\int_{\gamma}\dfrac{2 z}{w (z-a_i)}dz
\end{pmatrix}
=P_1 P_{a_i} P_2 
\begin{pmatrix}
\int_{\gamma}\dfrac{1-z^2}{w}  dz    \\
\int_{\gamma}\dfrac{i(1+z^2)}{w} dz     \\
\int_{\gamma}\dfrac{2 z}{w} dz     \\
\int_{\gamma}\dfrac{z^4-z^6}{w^3} dz     \\
\int_{\gamma}\dfrac{i(z^4+z^6)}{w^3} dz    \\
\int_{\gamma}\dfrac{z^5}{w^3} dz     
\end{pmatrix} 
\]
by Lemma~\ref{tP-lemma2}. 
Let $\Omega_{{\rm tP}}$ be the period matrix of the abelian differentials of the second 
kind (see \S~{\ref{tP-detail}}), that is, 
\[
\begin{pmatrix}
- i B   &   -A  &  i B  & -i B  &   -2 i B   &  -i B \\
A  &   i B  &  -A  &  i B  &  0  &  -i B \\
-i D   &   i D  &  -i D  &  C  &  0  &  C \\
- i F   &   -E  &  i F  &  -i F  & -2 i F  &   -i F \\ 
E  &  i F   &  -E  &  i F  & 0  &  -i F  \\ 
-i I  &   i I  &   -i I   &  H  &  0  &  H
\end{pmatrix}. 
\]
Set $\alpha := \sqrt{\frac{\sqrt{a+2}+\sqrt{a-2}}{2}}>1$. Then, choosing $a_1=e^{\frac{\pi}{4} i} \alpha$, 
$a_2=e^{\frac{3}{4}\pi i} \alpha$, 
$a_3=e^{-\frac{\pi}{4} i} \alpha$, 
$a_4=e^{-\frac{3}{4}\pi i} \alpha$, and 
$a_5=e^{\frac{\pi}{4} i} / \alpha$, 
we obtain five tangent vectors $\{T_i\}_{i=1}^5:= \left\{\frac{1}{2} P_1 P_{a_i} P_2 \Omega_{{\rm tP}}\right\}_{i=1}^5$ in $K_{3,\,6}$. 
Moreover, we shall consider the tangent vectors via an action of 
$SO(3,\,\mathbb{C}) \times (\mathbb{C} \setminus \{0\})$. 
Setting $C_1$ and $C_2$ are complex matrices of degree $3$ given by 
\[
(C_1,\,C_2)=
\begin{pmatrix}
- i B   &   -A  &  i B  & -i B  &   -2 i B   &  -i B \\
A  &   i B  &  -A  &  i B  &  0  &  -i B \\
-i D   &   i D  &  -i D  &  C  &  0  &  C 
\end{pmatrix}, 
\]
we have the Riemann matrix $\tau=C_1^{-1}C_2$ and 
the following four tangent vectors in $K_{3,\,6}$: 
\begin{align*}
&T_6:=(C_1,\,C_2),\quad 
T_7:=\begin{pmatrix}
0 & 1 & 0 \\
-1 & 0 & 0 \\
0 & 0 & 0
\end{pmatrix}
(C_1,\,C_2),
\end{align*}
\begin{align*}
&T_8:=\begin{pmatrix}
0 & 0 & 1 \\
0 & 0 & 0 \\
-1 & 0 & 0
\end{pmatrix}
(C_1,\,C_2),\quad 
T_9:=\begin{pmatrix}
0 & 0 & 0 \\
0 & 0 & 1 \\
0 & -1 & 0
\end{pmatrix}
(C_1,\,C_2). 
\end{align*}

\subsection{tCLP family} 

Since $M$ is defined by the same equation as tP family, we can take the same $P_1$, $P_{a_i}$, and $P_2$ 
as in \S~\ref{tP-section} for $a\in (-2,\,2)$. We now assume $a\in [0,\,2)$ because every minimal surface for $a\in (-2,\,0]$ 
can be transformed to its conjugate surface for $-a\in [0,\,2)$. 
Let $\Omega_{{\rm tCLP}}$ be the period matrix of the abelian differentials of the second kind (see \S~{\ref{tCLP-detail}}), that is, 
\[
\begin{pmatrix}
- i B   &   i B  &  i B  & 0  &   -A   &  -A \\
- i B  &   -i B  &  i B  &  A  &  A  &  0 \\
- C   &   C  &  - C  &  - i D  &  0  &  -i D \\
- i F   &   i F  &  i F  &  0  & E  &   E \\ 
- i F  &  -i F   &  i F  &  -E  & -E  &  0  \\ 
- H &   H  &   -H  &  i I  &  0  &  i I
\end{pmatrix}. 
\]
Set $e^{i\alpha} := -\frac{a}{2}+i \frac{\sqrt{4-a^2}}{2}\in S^1\subset \mathbb{C}$ ($\alpha \in [\pi/2,\,\pi) $). 
Choosing $a_1=e^{\frac{\alpha}{4} i}$, $a_2=i e^{\frac{\alpha}{4} i}$, $a_3=-e^{\frac{\alpha}{4} i}$, 
$a_4=-i e^{\frac{\alpha}{4} i}$, and $a_5=e^{-\frac{\alpha}{4} i}$, 
we obtain five tangent vectors $\{T_i\}_{i=1}^5:= \left\{\frac{1}{2} P_1 P_{a_i} P_2 \Omega_{{\rm tCLP}}\right\}_{i=1}^5$ in $K_{3,\,6}$. 
Moreover, setting $C_1$ and $C_2$ are complex matrices of degree $3$ given by 
\[
(C_1,\,C_2)=
\begin{pmatrix}
- i B   &   i B  &  i B  & 0  &   -A   &  -A \\
- i B  &   -i B  &  i B  &  A  &  A  &  0 \\
- C   &   C  &  - C  &  - i D  &  0  &  -i D
\end{pmatrix}, 
\]
we have the Riemann matrix $\tau=C_1^{-1} C_2$ and 
the following four tangent vectors in $K_{3,\,6}$: 
\begin{align*}
&T_6:=(C_1,\,C_2),\quad 
T_7:=\begin{pmatrix}
0 & 1 & 0 \\
-1 & 0 & 0 \\
0 & 0 & 0
\end{pmatrix}
(C_1,\,C_2), \\
&T_8:=\begin{pmatrix}
0 & 0 & 1 \\
0 & 0 & 0 \\
-1 & 0 & 0
\end{pmatrix}
(C_1,\,C_2),\quad 
T_9:=\begin{pmatrix}
0 & 0 & 0 \\
0 & 0 & 1 \\
0 & -1 & 0
\end{pmatrix}
(C_1,\,C_2). 
\end{align*}

\section{Morse indices and nullities}

In this section, we will show our main results. 
Every family of minimal surfaces is parametrized by $a$ which belongs to a 
suitable interval. 
First, we find $\{a\>|\> \det W=0\}$. By these points, each domain interval 
can be divided into some intervals. 
After that, we shall consider the negative eigenvalues of $W_2-W_1$ on each 
divided interval. Our main results follow from the procedure. 
\begin{remark}\label{rem-introduction} 
A minimal surface with $index_E=0$ must be the totally geodesic subtorus. 
Thus every minimal surface in this paper satisfies $index_E \geq 1$. 
Suppose that a family of minimal surfaces parametrized by $a\in (\alpha,\,\beta)$ 
satisfies $index_E = 1$. Then, 
by a fundamental eigenvalues argument, we have $index_E = 1$ for $a=\alpha,\,\beta$. 
In fact, let 
\[
\underbrace{\lambda_1,\,\cdots,\,\lambda_k}_{<0},\, 0,\,\cdots,\, 0,\, 
\underbrace{\lambda_{k+1},\,\cdots}_{>0} 
\]
be eigenvalues of $E_{\phi}$ at $a=\alpha$. In a neighborhood of $a=\alpha$, we have 
$index_E \geq k$. If $index_E = 1$ on $(\alpha,\,\beta)$, then $k\leq 1 $, that is, 
$k=1$ holds. So is the case $a=\beta$. 
\end{remark}
Since each $T_j$, $W$, $W_j$ are too complicated, we use numerical arguments by 
Mathematica. 

\subsection{H family}

The curve $\det W$ may meet the real axis at two points $a_1$, $a_2$ 
(see Figure~\ref{H-det}). 
Thus we consider three intervals $(0,\,a_1)$, $(a_1,\,a_2)$, $(a_2,\,1)$. 
\begin{figure}[htbp] \label{H-det}
\begin{center}
\begin{tabular}{cc}
 \includegraphics[width=.4\linewidth]{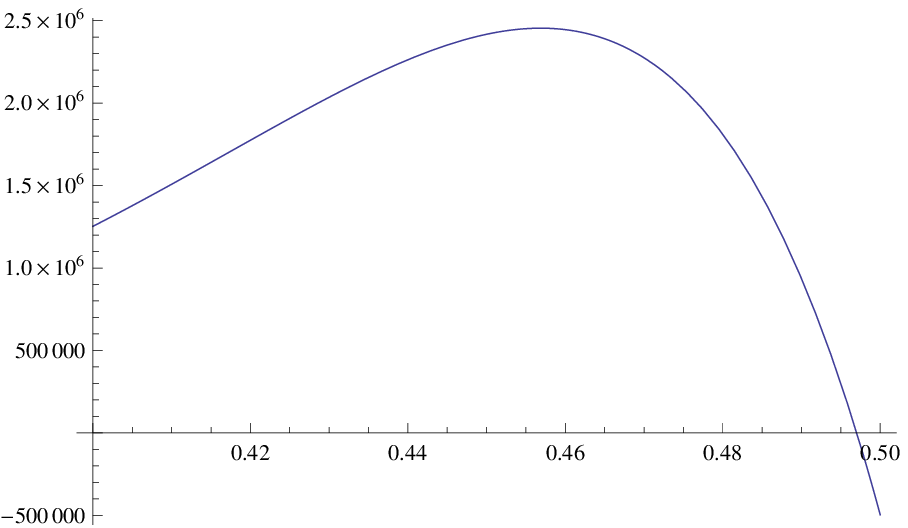} &
 \includegraphics[width=.5\linewidth]{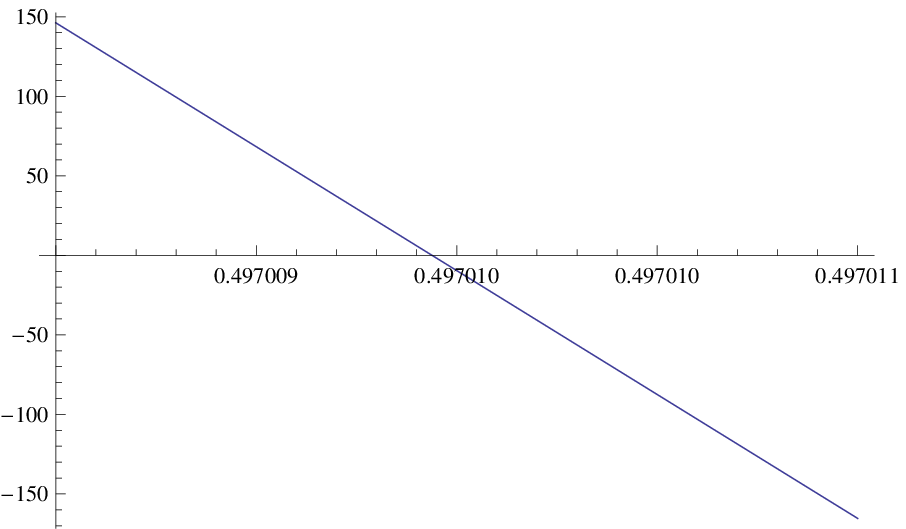} \\
 \includegraphics[width=.4\linewidth]{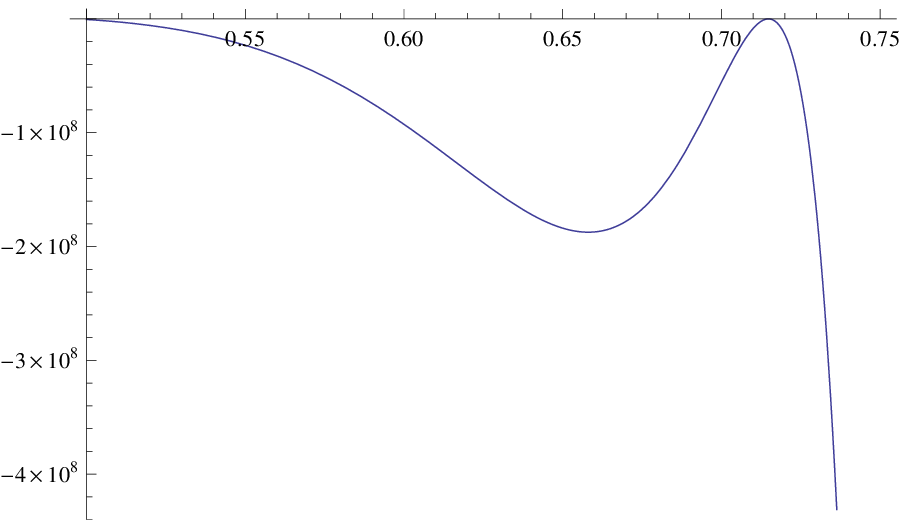} &
 \includegraphics[width=.5\linewidth]{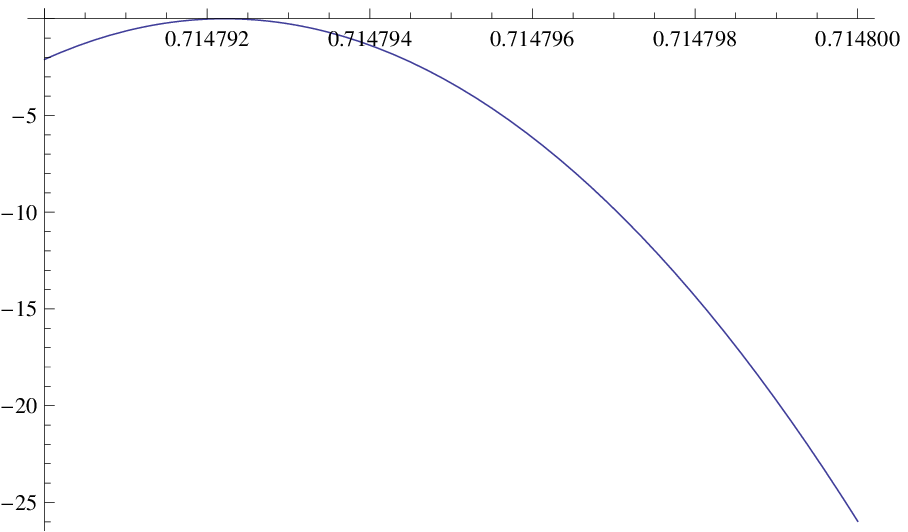} \\
\end{tabular}
\caption{ $\det W$ (the real axis is defined by $a$)}
\end{center}
\end{figure} 
Sets of eigenvalues of $W$ for $a= 0.49700993839805,\, 0.714792215373045$ 
are given by 
\begin{align*}
&\{ 160.732, 92.2764, 83.1167, 71.3028, -4.65805, -3.62774, -0.180182, \\
&\qquad -0.0540145, 0 \}, \\
& \{ 231.28, 147.299, 131.781, 125.929, -7.19325, -6.91484, -5.16664, \\
&\qquad  0, 0 \}, 
\end{align*}
respectively. Hence $a_1 \approx 0.49701$ and $a_2 \approx 0.71479$. 

Substituting $a=0.3,\,0.5,\,0.8$ to $W$, we obtain sets of the eigenvalues as follows: 
\begin{align*}
&\{ 144.683, 53.6184, 39.0519, 34.6726, -7.31714, -2.4839, 1.01569, \\
&\qquad -0.0461945, -0.00828724 \}, \\
&\{ 161.373, 92.8348, 83.8803, 71.9439, -4.64655, -3.64618, -0.18237, \\
& \qquad -0.0647386, -0.0274727 \}, \\
&\{ 275.447, 181.897, 155.676, 149.321, -14.6043, -12.4717, -8.58915, \\
& \qquad 0.405355, 0.108547 \}.
\end{align*}
Also, substituting $a=0.3,\,0.5,\,0.8$ to $W_2-W_1$, we have the following sets of the 
eigenvalues: 
\begin{align*}
& \{17.7972, 17.7971, 11.0343, 11.0248, 5.34828, -1.75067, 0.119967, 
0.085154, \\
&\qquad 0.0207832, 0.014763, 0, 0, 0, 0, 0, 0, 0, 0\}, \\
& \{14.8116, 14.7502, 10.5123, 9.83969, 9.68212, 0.819204, 0.324313, 
0.190585, \\
&\qquad 0.0728135, 0.050784, 0, 0, 0, 0, 0, 0, 0, 0\}, \\
& \{44.9919, 38.0399, 36.0209, 29.0891, 24.1326, 20.2568, 3.46409, 
1.12237, \\
& \qquad -0.384964, -0.130142, 0, 0, 0, 0, 0, 0, 0, 0\}. 
\end{align*}
Therefore, we conclude that 
\[
\begin{cases}
(p,\, q)=(5,\,4), \>\>\> index_E=2,\>\>\> nullity_E=0 & (a\in (0,\,a_1)) \\
(p,\, q)=(4,\,4), \>\>\> index_E=1,\>\>\> nullity_E=1 & (a = a_1) \\
(p,\,q)=(4,\,5), \>\>\> index_E=1, \>\>\> nullity_E=0 & (a\in (a_1,\,a_2)) \\
(p,\,q)=(4,\,3), \>\>\> index_E=1, \>\>\> nullity_E = 2 & (a = a_2) \\
(p,\,q)=(6,\,3), \>\>\> index_E=3, \>\>\> nullity_E = 0 & (a\in (a_2,\,1))
\end{cases}
\]

\subsection{rPD family}

The curve $\det W$ may meet the real axis at a point $a_1$ (see Figure~\ref{rPD-det}). 
Thus we consider two intervals $(0,\,a_1)$, $(a_1,\,1]$. 
\begin{figure}[htbp]\label{rPD-det} 
\begin{center}
\begin{tabular}{cc}
 \includegraphics[width=.4\linewidth]{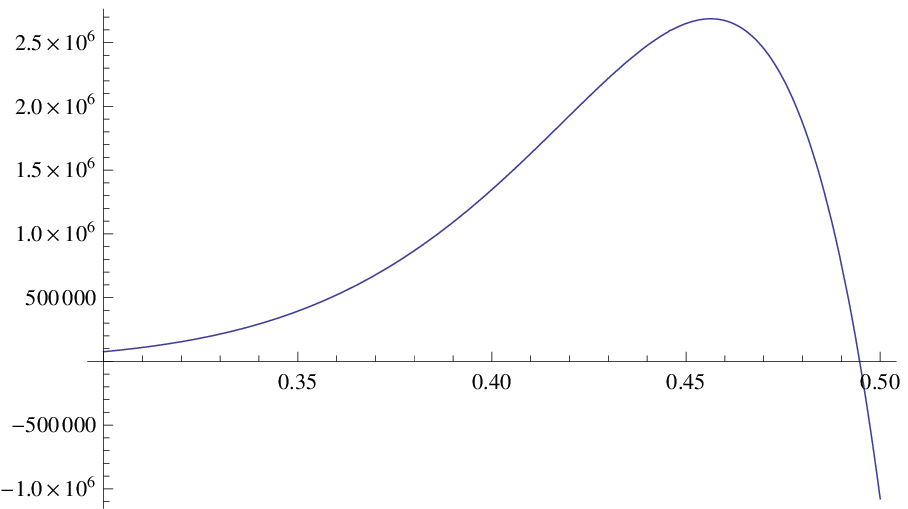} &
 \includegraphics[width=.5\linewidth]{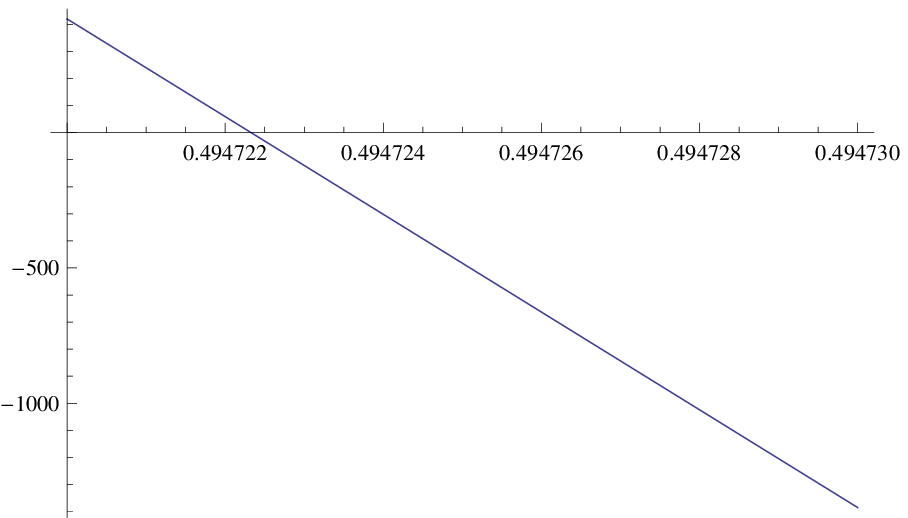} \\
 \includegraphics[width=.4\linewidth]{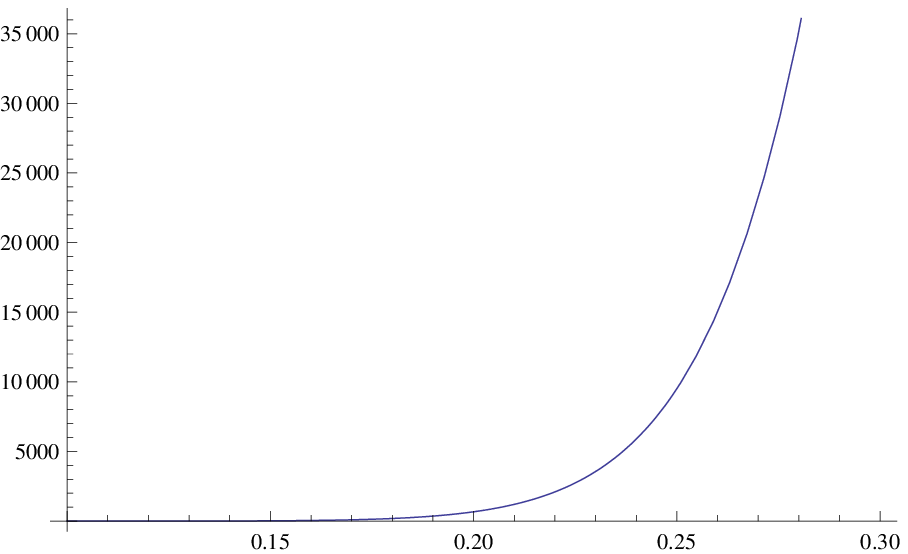} &
 \includegraphics[width=.4\linewidth]{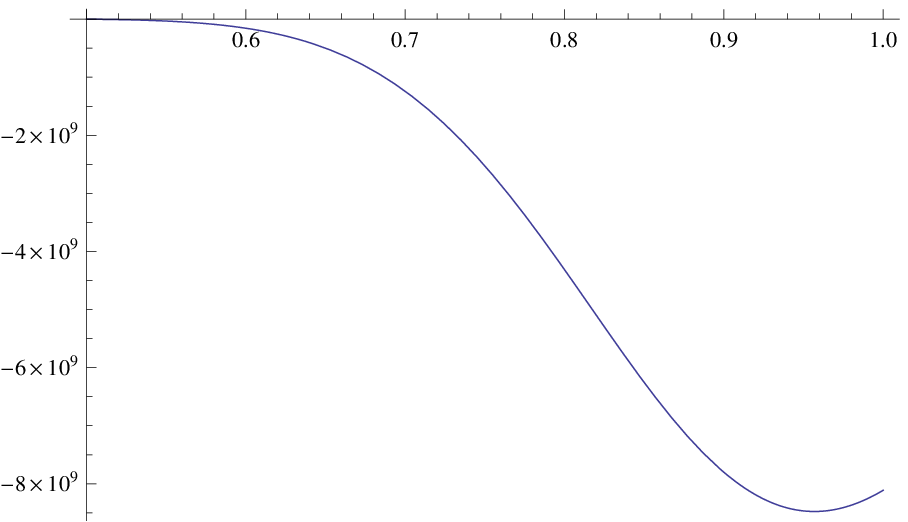} \\
\end{tabular}
\caption{ $\det W$ (the real axis is defined by $a$)}
\end{center}
\end{figure} 
A set of eigenvalues of $W$ for $a= 0.494722327827355$ is given by 
\begin{align*}
&\{ 160.726, 90.8344, 81.34, 70.8198, -4.32264, -3.35846, -0.218914, \\ 
&\qquad -0.0654375, 0 \}. 
\end{align*}
Hence $a_1 \approx 0.494722$. 

Substituting $a=0.3,\,0.5$ to $W$, we obtain sets of the eigenvalues as follows: 
\begin{align*}
&\{144.749, 53.5907, 39.027, 34.7134, -7.29118, -2.47246, 1.01161, \\
& \qquad -0.0471324, -0.00845948\}, \\
&\{161.851, 91.747, 82.5866, 71.9293, -4.28067, -3.37296, -0.226271, \\
& \qquad -0.0900042, -0.041516\}. 
\end{align*}
Also, substituting $a=0.3,\,0.5$ to $W_2-W_1$, we have the following sets of the 
eigenvalues: 
\begin{align*}
& \{17.7528, 17.752, 11.0056, 11.0005, 5.33466, -1.74533, 0.117449, 
0.0869349, \\
&\qquad 0.0207703, 0.01538, 0, 0, 0, 0, 0, 0, 0, 0\}, \\
& \{14.0747, 13.9953, 9.92996, 9.22538, 9.22362, 0.719009, 0.359877, 
0.214733, \\
&\qquad 0.100869, 0.0860548, 0, 0, 0, 0, 0, 0, 0, 0\}. 
\end{align*}
Therefore, we conclude that 
\[
\begin{cases}
(p,\,q)=(5,\,4), \>\>\> index_E=2, \>\>\> nullity_E = 0 & (a\in (0,\,a_1)) \\
(p,\,q)=(4,\,4), \>\>\> index_E=1, \>\>\> nullity_E = 1 & (a = a_1) \\
(p,\,q)=(4,\,5), \>\>\> index_E=1, \>\>\> nullity_E = 0 & (a\in (a_1,\,1]) 
\end{cases}
\]

\subsection{tP family, tD family}

The curve $\det W$ may meet the real axis at two points $a_1$, $a_2$ 
(see Figure~\ref{tP-det}). 
Thus we consider three intervals $(2,\,a_1)$, $(a_1,\,a_2)$, $(a_2,\,\infty)$. 
\begin{figure}[htbp]\label{tP-det} 
\begin{center}
\begin{tabular}{cc}
 \includegraphics[width=.4\linewidth]{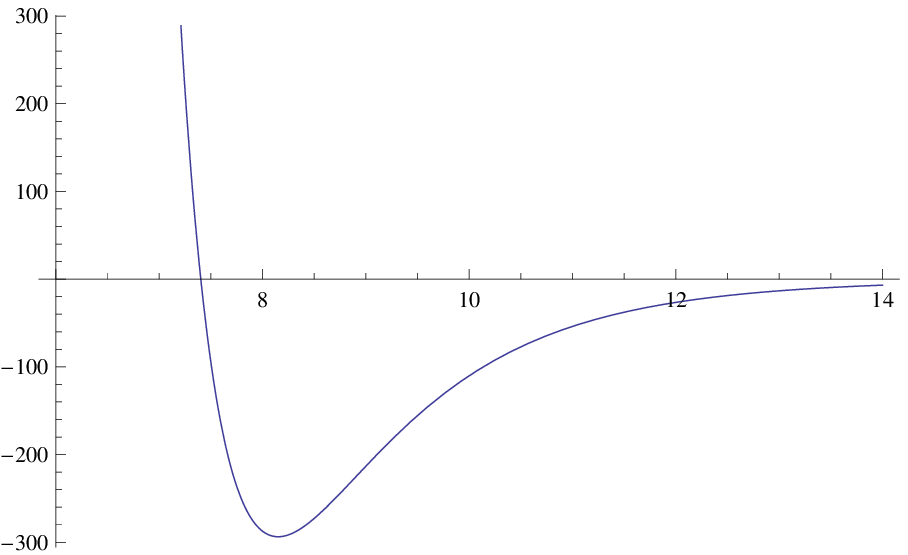} &
 \includegraphics[width=.5\linewidth]{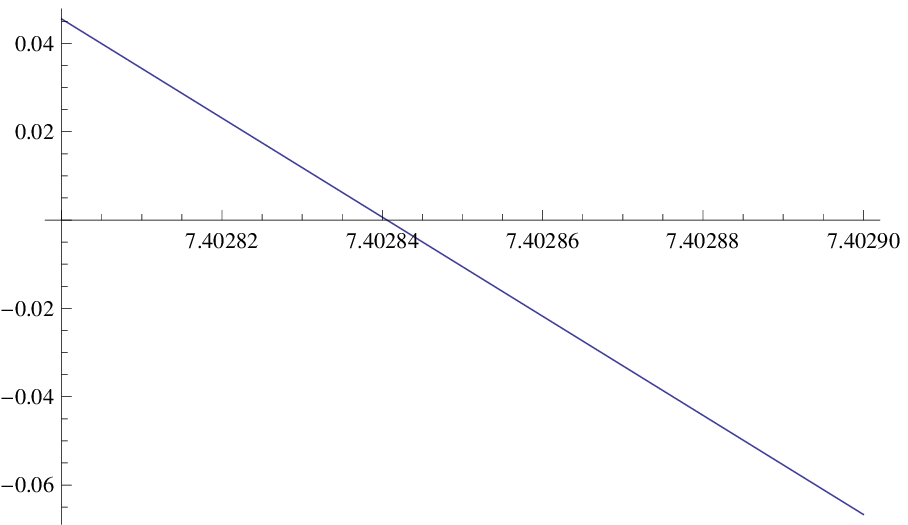} \\
 \includegraphics[width=.4\linewidth]{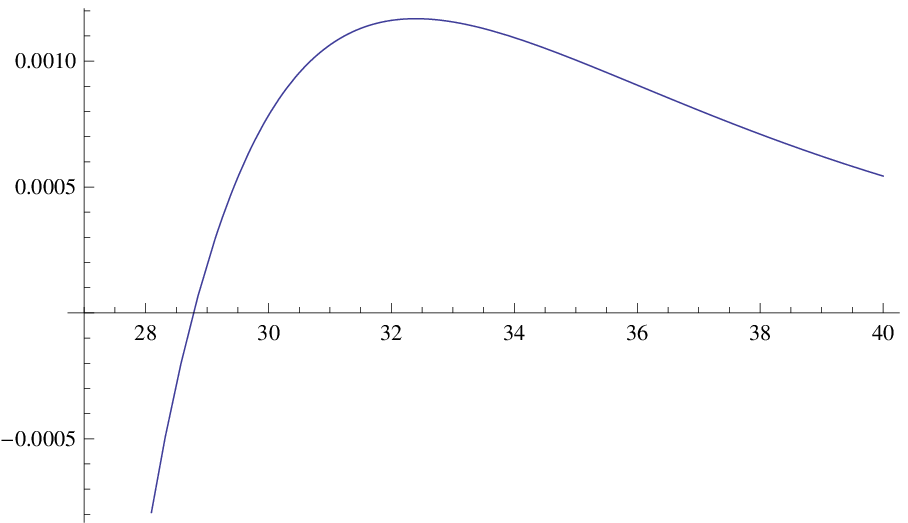} &
 \includegraphics[width=.5\linewidth]{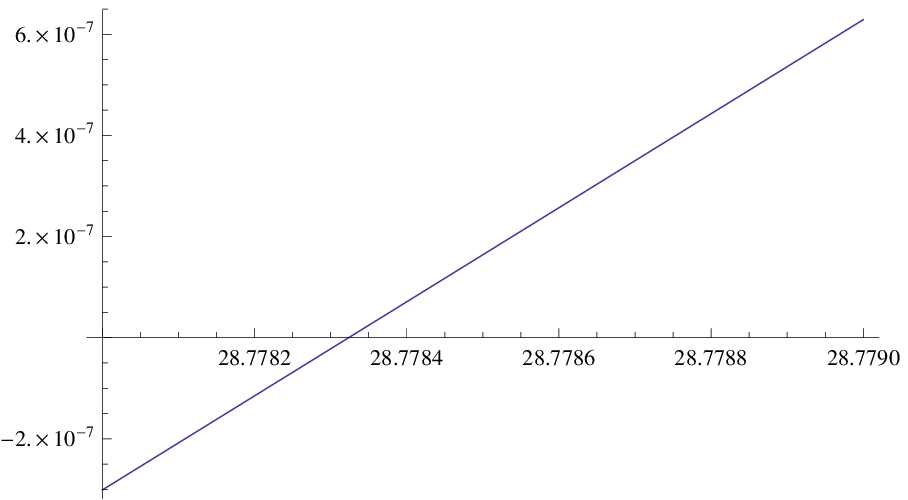} 
\end{tabular}
\caption{ $\det W$ (the real axis is defined by $a$)}
\end{center}
\end{figure} 
Sets of eigenvalues of $W$ for $a= 7.4028405832965,\, 28.7783236867029$ 
are given by 
\begin{align*}
&\{ 80.9577, 48.893, 45.7276, 40.9177, -1.74202, -0.459907, -0.182299, \\
&\qquad -0.124155, 0 \}, \\
& \{ 34.5772, 22.6075, 19.4385, 18.7545, -0.858857, -0.0828255, -0.0605294, \\
&\qquad -0.00120275, 0 \}, 
\end{align*}
respectively. Hence $a_1\approx 7.40284$ and $a_2\approx 28.7783$. 

Substituting $a=7,\,14,\,30$ to $W$, we obtain sets of the eigenvalues as follows: 
\begin{align*}
&\{84.0549, 50.6269, 47.5056, 42.1118, -1.82758, -0.502372, -0.195009, \\
&\qquad -0.131043, 0.00417493\}, \\
&\{53.6301, 33.4006, 29.9484, 29.3214, -1.15324, -0.175836, -0.133389, \\
& \qquad -0.0334135, -0.00481829\}, \\
&\{33.7419, 22.1296, 18.9583, 18.2612, -0.8469, -0.080107, -0.056963, \\
& \qquad -0.00109614, 0.000716405\}.
\end{align*}
Also, substituting $a=7,\,14,\,30$ to $W_2-W_1$, we have the following sets of the 
eigenvalues: 
\begin{align*}
& \{6.0379, 5.37072, 1.57278, 1.28479, 1.17983, 0.510969, 0.329481, \
0.207524, \\
&\qquad 0.0869196, -0.00442925, 0, 0, 0, 0, 0, 0, 0, 0\}, \\
& \{4.01549, 3.82079, 0.660116, 0.450696, 0.404925, 0.209065, 0.146808, \
0.0499664, \\
&\qquad 0.0163956, 0.00721722, 0, 0, 0, 0, 0, 0, 0, 0\}, \\
& \{3.06952, 3.02215, 0.269261, 0.151882, 0.13806, 0.116333, 0.0594526, \
0.00268534, \\
&\qquad 0.00188477, -0.000794133, 0, 0, 0, 0, 0, 0, 0, 0\}. 
\end{align*}
Therefore, we conclude that 
\[
\begin{cases}
(p,\,q)=(5,\,4), \>\>\> index_E=2, \>\>\> nullity_E = 0 & (a\in (2,\,a_1)) \\
(p,\,q)=(4,\,5), \>\>\> index_E=1, \>\>\> nullity_E=0 & (a\in (a_1,\,a_2)) \\
(p,\,q)=(5,\,4), \>\>\> index_E=2, \>\>\> nullity_E = 0 & (a\in (a_2,\, \infty )) \\
(p,\,q)=(4,\,4), \>\>\> index_E=1, \>\>\> nullity_E = 1 & (a = a_1,\,a_2) 
\end{cases}
\]

\subsection{tCLP family}

\begin{figure}[htbp]\label{tCLP-det} 
\begin{center}
\begin{tabular}{cc}
 \includegraphics[width=.45\linewidth]{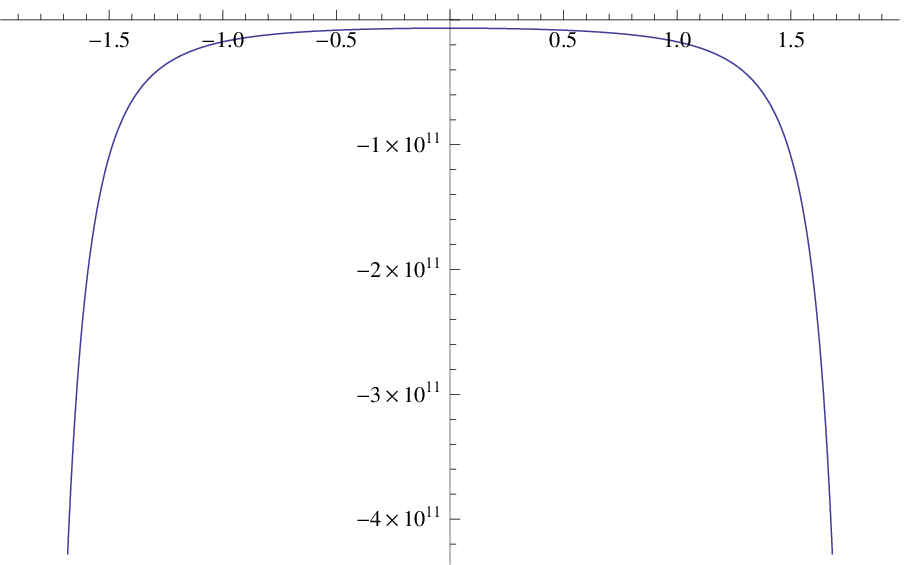} &
 \includegraphics[width=.45\linewidth]{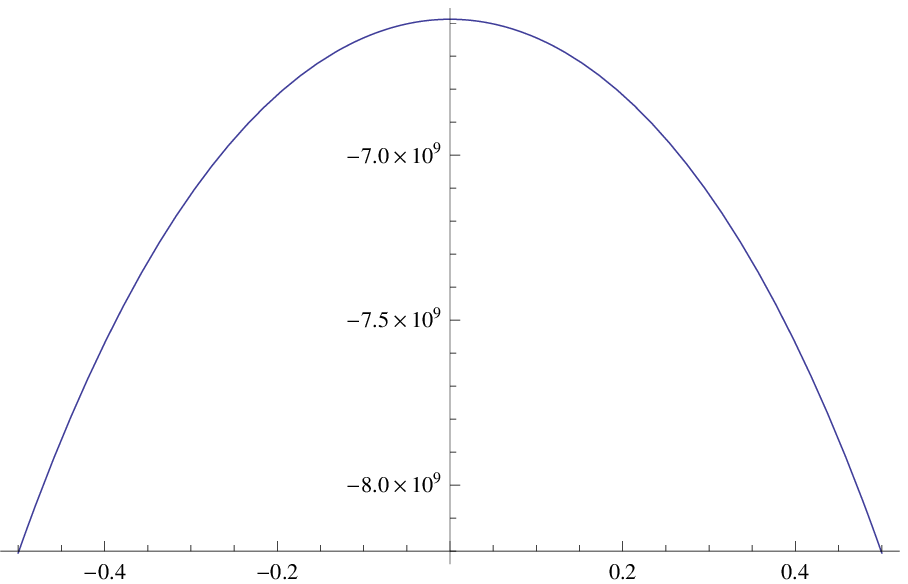} \\
\end{tabular}
\caption{ $\det W$ (the real axis is defined by $a$)}
\end{center}
\end{figure} 

The curve $\det W$ may not meet the real axis (see Figure~\ref{tCLP-det}). 
Substituting $a=0$ to $W$, we obtain a set of the eigenvalues as follows: 
\begin{align*}
&\{ 169.074, 106.977, 101.864, 74.8337, -7.06191, -3.56035, -3.24383, \\
&\qquad 2.28679, 0.25619\}. 
\end{align*}
Also, substituting $a=0$ to $W_2-W_1$, we have the following set of the 
eigenvalues: 
\begin{align*}
& \{ 15.791, 15.791, 8.75673, 8.75673, 7.27068, 7.27068, 5.06466, 5.06466, \\
&\qquad -0.529658, -0.529658, 0, 0, 0, 0, 0, 0, 0, 0 \}. 
\end{align*}
Therefore, we conclude that 
\[
(p,\,q)=(6,\,3), \>\>\> index_E=3, \>\>\> nullity_E = 0
\]
for $a\in (-2,\,2)$. 

\section{Appendix (canonical homology bases and period matrices)}

We shall give a canonical homology basis and calculate periods of the abelian 
differentials of the second kind on each Riemann surface defined in \S 1.

\subsection{H family}\label{H-detail}

\subsubsection{Canonical homology basis}\label{H-canonical}

Let $M$ be a hyperelliptic Riemann surface of genus $3$ defined as the completion of 
$\{(z,\,w)\,|\, w^2=z(z^3-a^3) \left( z^3-\frac{1}{a^3} \right)\} \subset \mathbb{C}^2$ for $a\in (0,\,1)$. 
The three differentials 
\[\dfrac{dz}{w},\,z\dfrac{dz}{w},\,z^2\dfrac{dz}{w}\]
form a basis for the abelian differentials of the first kind, that is, holomorphic differentials (see. p.255 in \cite{Grif-Har}). 
Up to exact forms, the abelian differentials of the second kind are given by the following six differentials 
(see p.460 in \cite{Grif-Har}): 
\[\dfrac{dz}{w},\,z\dfrac{dz}{w},\,z^2\dfrac{dz}{w},\,\dfrac{z^4}{w^3}dz,\,\dfrac{z^5}{w^3}dz,\,\dfrac{z^6}{w^3}dz.\]
In fact, we have the following divisors: 
\begin{align*}
\left(\dfrac{z^4}{w^3}dz\right)&=6(0,\,0)-2(a,\,0)-2(e^{\frac{2}{3}\pi i}a,\,0)-2(e^{-\frac{2}{3}\pi i}a,\,0) \\
&\quad -2(1/a,\,0)-2(e^{\frac{2}{3}\pi i}/a,\,0)-2(e^{-\frac{2}{3}\pi i}/a,\,0)+10(\infty,\,\infty), \\
\left(\dfrac{z^5}{w^3}dz\right)&=8(0,\,0)-2(a,\,0)-2(e^{\frac{2}{3}\pi i}a,\,0)-2(e^{-\frac{2}{3}\pi i}a,\,0) \\
&\quad -2(1/a,\,0)-2(e^{\frac{2}{3}\pi i}/a,\,0)-2(e^{-\frac{2}{3}\pi i}/a,\,0)+8(\infty,\,\infty), \\
\left(\dfrac{z^6}{w^3}dz\right)&=10(0,\,0)-2(a,\,0)-2(e^{\frac{2}{3}\pi i}a,\,0)-2(e^{-\frac{2}{3}\pi i}a,\,0) \\
&\quad -2(1/a,\,0)-2(e^{\frac{2}{3}\pi i}/a,\,0)-2(e^{-\frac{2}{3}\pi i}/a,\,0)+6(\infty,\,\infty).
\end{align*}
Thus, the above three differentials are meromorphic. We now see that they have no residues. 
To do this, we only consider residues at the pole 
$(a,\,0)$ because others can be showed in the same way. 
The implicit function theorem assures us that $w$ is a holomorphic coordinate around poles. 
We can write 
\[z=a+a_2 w^2+a_4 w^4+\cdots\qquad \left(a_2\neq 0 \right)\]
around $(a,\,0)$ and find 
\[\dfrac{z^j}{w^3}dz=\dfrac{(a+a_2 w^2+a_4 w^4+\cdots)^j(2 a_2 +4 a_4 w^2+\cdots)}{w^2}dw\]
for $j=4,\,5,\,6$. So they have no residues at $(a,\,0)$. 

Next, let 
\[G=i \left(\dfrac{1-z^2}{w},\,\dfrac{i(1+z^2)}{w},\,\dfrac{2z}{w},\,
\dfrac{z^4-z^6}{w^3},\,\dfrac{i(z^4+z^6)}{w^3},\,\dfrac{z^5}{w^3}\right)^t dz \]
and consider the biholomorphisms 
\begin{align*}
j(z,\,w)&=(z,\,-w),\>\>\varphi(z,\,w)=(e^{\frac{2\pi}{3}i}z,\,e^{\frac{\pi}{3}i}w)
\end{align*}
on $M$. Then it is straightforward to compute that 
\begin{align*}
j^{*}G=-G, \quad 
\varphi^{*}G & =
\begin{pmatrix}
\dfrac{1}{2} & \dfrac{\sqrt{3}}{2}  & 0 & 0 & 0 & 0 \\
-\dfrac{\sqrt{3}}{2} & \dfrac{1}{2}  & 0 & 0 & 0 & 0 \\
0 & 0 & -1 & 0 & 0 & 0 \\
0 & 0 & 0 & \dfrac{1}{2} & \dfrac{\sqrt{3}}{2}  & 0 \\
0 & 0 & 0 & -\dfrac{\sqrt{3}}{2} & \dfrac{1}{2}  & 0 \\
0 & 0 & 0 & 0 & 0 & -1 
\end{pmatrix}G.  
\end{align*}
Now we determine a canonical homology basis on $M$. Recall that 
\begin{align*}
\pi_{{\rm H}}:&\quad M \longrightarrow \overline{\mathbb{C}}:=\mathbb{C} \cup \{\infty\} \\
& (z,\,w) \> \longmapsto \> z
\end{align*}
defines a two-sheeted branched covering having branch locus
\[
\{ (0,\,0),\>(\infty,\,\infty),\> (a,\,0),\> (e^{\pm \frac{2}{3}\pi i}a,\,0),\>  
(1/a,\,0),\> (e^{\pm \frac{2}{3}\pi i}/a,\,0) \}
\]
and $j$ is its deck transformation. 
So $M$ can be expressed as a $2$-sheeted branched cover of $\overline{\mathbb{C}}$ in the following way. \\

\begin{picture}(340,150)

\put(10,90){\line(1,0){140}} \put(170,90){\line(1,0){140}}
\put(70,30){\line(0,1){120}} \put(230,30){\line(0,1){120}}

\put(70,90){\circle*{2}} \put(90,90){\circle*{2}} \put(120,90){\circle*{2}} 
\put(57,108){\circle*{2}} \put(40,131){\circle*{2}} \put(57,72){\circle*{2}} \put(40,49){\circle*{2}} 

\put(230,90){\circle*{2}} \put(250,90){\circle*{2}} \put(280,90){\circle*{2}} 
\put(217,108){\circle*{2}} \put(200,131){\circle*{2}} \put(217,72){\circle*{2}} \put(200,49){\circle*{2}} 
 
\put(60,80){$O$} \put(88,96){$a$} \put(110,96){$1/a$} \put(60,112){$e^{\frac{2}{3} \pi i} a$} \put(5,130){$e^{\frac{2}{3} \pi i} /a$}
\put(20,70){$e^{-\frac{2}{3} \pi i} a$} \put(0,50){$e^{-\frac{2}{3} \pi i} /a$}

\put(70,90){\line(-13,18){13}} \put(40,49){\line(-17,-23){12}}
\put(230,90){\line(-13,18){13}} \put(200,49){\line(-17,-23){12}}
\qbezier(40,131)(30,90)(57,72) \qbezier(200,131)(190,90)(217,72)

\put(75,94){$+$} \put(75,82){$-$} \put(130,94){$+$} \put(130,82){$-$} \put(50,120){$+$} \put(41,111){$-$} 
\put(40,60){$+$} \put(48,52){$-$} 
\put(235,94){$+$} \put(235,82){$-$} \put(290,94){$+$} \put(290,82){$-$} \put(210,120){$+$} \put(201,111){$-$} 
\put(200,60){$+$} \put(208,52){$-$} 

\put(120,140){{\rm (i) }} \put(280,140){{\rm (ii) }}

\put(143,96){$(\infty)$} \put(303,96){$(\infty)$} \put(15,23){$(\infty)$} \put(175,23){$(\infty)$} 

\put(140,5){{\rm \textbf{figure (H)}}}

\thicklines

\put(70,90){\line(1,0){20}} \put(120,90){\line(1,0){30}} \put(57,108){\line(-17,23){17}} \put(57,72){\line(-17,-23){17}} 
\put(230,90){\line(1,0){20}} \put(280,90){\line(1,0){30}} \put(217,108){\line(-17,23){17}} \put(217,72){\line(-17,-23){17}} 
 
\end{picture}
$\,$\\
We prepare two copies of $\overline{\mathbb{C}}$ and slit them along the thick lines in {\rm figure (H)}. 
Identifying each of the upper (resp. lower) edges of the thick lines in {\rm (i)} with 
each of the lower (resp. upper) edges of the thick lines in (ii), we obtain the hyperelliptic Riemann surface $M$ of genus $3$ 
(see the following figure). 
Note that each of thin lines joining two branch points in {\rm figure (H)} is corresponding to each of thick lines 
joining two branch points in the following figure. 

\begin{picture}(340,160)

\put(135,150){$(\infty,\infty)$} \put(133,122){$(1/a,\,0)$} \put(138,105){$(a,\,0)$} 
\put(140,82){$(0,\,0)$} \put(130,65){$(e^{\frac{2}{3}\pi i} a ,\,0)$} \put(130,40){$(e^{\frac{2}{3}\pi i}/a,\,0)$}
\put(130,24){$(e^{-\frac{2}{3} \pi i} a ,\,0)$} \put(128,0){$(e^{-\frac{2}{3} \pi i}/a,\,0)$} 

\put(122,134){$-$} \put(122,93){$-$} \put(122,53){$-$} \put(122,13){$-$} 
\put(212,133){$-$} \put(212,93){$-$} \put(212,53){$-$} \put(212,13){$-$} 

\put(90,133){$+$} \put(90,93){$+$} \put(90,53){$+$} \put(90,13){$+$} 
\put(180,133){$+$} \put(180,93){$+$} \put(180,53){$+$} \put(180,13){$+$} 

\put(110,150){\circle*{2}} \put(110,120){\circle*{2}} \put(110,110){\circle*{2}} \put(110,80){\circle*{2}} 
\put(110,70){\circle*{2}} \put(110,40){\circle*{2}} \put(110,30){\circle*{2}} \put(110,0){\circle*{2}}
\put(200,150){\circle*{2}} \put(200,120){\circle*{2}} \put(200,110){\circle*{2}} \put(200,80){\circle*{2}} 
\put(200,70){\circle*{2}} \put(200,40){\circle*{2}} \put(200,30){\circle*{2}} \put(200,0){\circle*{2}}

\qbezier(110,150)(100,150)(100,135) \qbezier(100,135)(100,120)(110,120)
\qbezier[10](110,150)(120,150)(120,135) \qbezier[10](120,135)(120,120)(110,120)

\qbezier(110,110)(100,110)(100,95) \qbezier(100,95)(100,80)(110,80)
\qbezier[10](110,110)(120,110)(120,95) \qbezier[10](120,95)(120,80)(110,80)

\qbezier(110,70)(100,70)(100,55) \qbezier(100,55)(100,40)(110,40)
\qbezier[10](110,70)(120,70)(120,55) \qbezier[10](120,55)(120,40)(110,40)

\qbezier(110,30)(100,30)(100,15) \qbezier(100,15)(100,0)(110,0)
\qbezier[10](110,30)(120,30)(120,15) \qbezier[10](120,15)(120,0)(110,0)

\qbezier(200,150)(190,150)(190,135) \qbezier(190,135)(190,120)(200,120)
\qbezier[10](200,150)(210,150)(210,135) \qbezier[10](210,135)(210,120)(200,120)

\qbezier(200,110)(190,110)(190,95) \qbezier(190,95)(190,80)(200,80)
\qbezier[10](200,110)(210,110)(210,95) \qbezier[10](210,95)(210,80)(200,80)

\qbezier(200,70)(190,70)(190,55) \qbezier(190,55)(190,40)(200,40)
\qbezier[10](200,70)(210,70)(210,55) \qbezier[10](210,55)(210,40)(200,40)

\qbezier(200,30)(190,30)(190,15) \qbezier(190,15)(190,0)(200,0)
\qbezier[10](200,30)(210,30)(210,15) \qbezier[10](210,15)(210,0)(200,0)

\put(50,150){(i)} \put(245,150){(ii)}

\thicklines

\qbezier(110,150)(30,150)(30,75) \qbezier(30,75)(30,0)(110,0)

\qbezier(200,150)(280,150)(280,75) \qbezier(280,75)(280,0)(200,0)

\qbezier(110,120)(80,115)(110,110) \qbezier(110,80)(80,75)(110,70)
\qbezier(110,40)(80,35)(110,30) 

\qbezier(200,120)(230,115)(200,110) \qbezier(200,80)(230,75)(200,70)
\qbezier(200,40)(230,35)(200,30) 

\end{picture}
$\,$\\
To describe $1$-cycles on $M$, we consider the following key paths:
\begin{align*}
C_1&=\{(z,\,w)=(a t,\,a^2 \sqrt{t(1-t^3)(1/a^3-a^3 t^3)})\>|\>t: 0\to 1,\> \sqrt{*}>0 \}, \\
C_2&=\{(z,\,w)=(t,\,-i \sqrt{t(t^3-a^3)(1/a^3-t^3)})\>|\>t: a\to 1/a,\> \sqrt{*}>0 \}. 
\end{align*}
We first choose $C_1$ in the following figure. After that, we shall see a relation between $C_1$ and $C_2$. 

\begin{picture}(340,100)

\put(10,10){\line(1,0){140}} 
\put(70,0){\line(0,1){80}} 

\put(70,10){\circle*{2}} \put(90,10){\circle*{2}} \put(120,10){\circle*{2}} 

\put(75,19){$+$} \put(75,2){$-$} \put(120,60){{\rm (i) }} \put(90,17){$C_1$} \put(245,35){$C_1$}

\put(283,45){$(a,\,0)$} \put(285,22){$(0,\,0)$}

\put(270,50){\circle*{2}} \put(270,20){\circle*{2}} 

\thicklines

\put(70,10){\line(1,0){20}} \put(120,10){\line(1,0){30}} 

\put(70,14){\line(1,0){20}} \put(83,14){\line(-2,1){5}} \put(83,14){\line(-2,-1){5}}

\thinlines

\qbezier(270,90)(260,90)(260,75) \qbezier(260,75)(260,60)(270,60)
\qbezier[10](270,90)(280,90)(280,75) \qbezier[10](280,75)(280,60)(270,60)

\qbezier[10](270,50)(280,50)(280,35) \qbezier[10](280,35)(280,20)(270,20)

\put(210,90){(i)} 

\qbezier(270,90)(190,90)(190,15) \qbezier(270,60)(240,55)(270,50) \qbezier(270,20)(240,15)(270,10)

\thicklines

\qbezier(270,50)(260,50)(260,35) \qbezier(260,35)(260,20)(270,20)
\put(260,38){\line(-1,-2){3}} \put(260,38){\line(1,-2){3}} 

\end{picture}

$\,$\\
To do this, we introduce the three paths: 
\begin{align*}
C'_2&=\{(z,\,w)=(t,\,-i \sqrt{t(t^3-a^3)(1/a^3-t^3)})\>|\>t: a\to 1,\> \sqrt{*}>0 \}, \\
C_3&=\{(z,\,w)=(e^{it},\,w(t))\>|\>t: 0 \to \pi/3,\>\> w(0)\in -i\mathbb{R}_{>0}\}, \\
C_4&=\{(z,\,w)=(e^{\frac{\pi}{3}i}t,\,e^{\frac{\pi}{6}i}\sqrt{t(t^3+a^3)(t^3+1/a^3)})\>|\>t:1\to 0, \> \sqrt{*}>0\}, 
\end{align*}
and claim that $C_1\cup (C'_2 \cup C_3\cup C_4)$ is homotopic to zero by using path-integrals of the holomorphic 
differential $\dfrac{1-z^2}{w} dz$. Note that $C'_2 \cap C_3=\{(1,\,0)\}$. 

Straightforward calculations yield
\begin{align}
\int_{C_1}\dfrac{1-z^2}{w} dz &=\dfrac{1}{a} \int_0^1 \dfrac{1-a^2 t^2}{\sqrt{t(1-t^3)(\frac{1}{a^3}-a^3 t^3)}} dt, \label{H-inform1} 
\end{align}
\begin{align}
\int_{C'_2}\dfrac{1-z^2}{w} dz &= i \int_a^1 \dfrac{1-t^2}{\sqrt{t(t^3-a^3)(\frac{1}{a^3}-t^3)}} dt, \label{H-inform2} \\
\int_{C_4}\dfrac{1-z^2}{w} dz &=-\dfrac{\sqrt{3}}{2} \int_0^1 \dfrac{1+ t^2}{\sqrt{t(t^3+a^3)(t^3+\frac{1}{a^3})}} dt  
-\dfrac{i}{2} \int_0^1 \dfrac{1- t^2}{\sqrt{t(t^3+a^3)(t^3+\frac{1}{a^3})}} dt. \label{H-inform3}
\end{align}
Now we set $x=(z+1/z)/2=\cos t:1\to 1/2$ along $C_3$. Then $\dfrac{1-z^2}{w}dz=-2\dfrac{z^2}{w} dx$. 
$z+1/z=2x$ implies that $z^2+1/z^2=4x^2-2$, and thus $z^3+1/z^3=(z+1/z)(z^2-1+1/z^2)=8x^3-6x$. Hence, 
\begin{align*}
\left( \dfrac{z^2}{w}  \right)^2&=\dfrac{z^4}{z(z^3-a^3)(z^3-\frac{1}{a^3})}=
-\dfrac{1}{a^3+\frac{1}{a^3}-(z^3+\frac{1}{z^3})} \\
&=-\dfrac{1}{a^3+\frac{1}{a^3}+6x-8x^3}<0.
\end{align*}
To choose a suitable branch, we substitute $t=0$ to $z^2/w$. Then 
$\dfrac{z^2}{w}(t=0)=\dfrac{1}{w(0)}\in i \mathbb{R}_{>0}$. As a result, 
\begin{align}
\dfrac{z^2}{w} = \dfrac{i}{\sqrt{ a^3+\frac{1}{a^3}+6x-8x^3 }} \in i \mathbb{R}_{>0}.  \label{H-branch1}
\end{align}
Consequently, we have 
\begin{align}
\int_{C_3} \dfrac{1-z^2}{w}dz=2 i \int^1_{\frac{1}{2}}  \dfrac{dx}{\sqrt{ a^3+\frac{1}{a^3}+6x-8x^3 }}. \label{H-inform4}
\end{align}
Moreover, combining \eqref{H-branch1} and substituting $t=\pi/3$ to $z^2/w$, we find 
$\dfrac{e^{\frac{2}{3} \pi i}}{w(\pi/3)} \in i \, \mathbb{R}_{>0}$. 
So $w(\pi/3) \in e^{\frac{\pi }{6} i } \mathbb{R}_{>0}$. 
Thus, $C_3 \cap C_4=\{(e^{\frac{\pi}{3}i},\,e^{\frac{\pi}{6}i}\sqrt{(1+a^3)(1+1/a^3)})\} $, and therefore, 
$C'_2 \cup C_3\cup C_4$ defines a connected path. 

Since $\pi_{{\rm H}} (C_1)\cup \pi_{{\rm H}} ( C'_2 \cup C_3\cup C_4)$ is homotopic to zero, 
there exists $n\in \{ 0,\,1 \}$ such that $C_1 \cup j^n (C'_2 \cup C_3\cup C_4)$ is homotopic to zero. 
It follows from \eqref{H-inform1}, \eqref{H-inform2}, \eqref{H-inform3}, and \eqref{H-inform4} that 
\begin{align*}
&\dfrac{1}{a} \int_0^1 \dfrac{1-a^2 t^2}{\sqrt{t(1-t^3)(\frac{1}{a^3}-a^3 t^3)}} dt+(-1)^n \Bigg\{ 
i \int_a^1 \dfrac{1-t^2}{\sqrt{t(t^3-a^3)(\frac{1}{a^3}-t^3)}} dt \\
&\>-\dfrac{\sqrt{3}}{2} \int_0^1 \dfrac{1+ t^2}{\sqrt{t(t^3+a^3)(t^3+\frac{1}{a^3})}} dt  
-\dfrac{i}{2} \int_0^1 \dfrac{1- t^2}{\sqrt{t(t^3+a^3)(t^3+\frac{1}{a^3})}} dt 
\end{align*}
\begin{align*}
&\quad +2 i \int^1_{\frac{1}{2}}  \dfrac{dx}{\sqrt{ a^3+\frac{1}{a^3}+6x-8x^3 }} \Bigg\} =0. 
\end{align*}
The real part of the above implies $n=0$. Hence $C_1\cup (C'_2 \cup C_3\cup C_4)$ is homotopic to zero (see the following figure). 
So we obtain $C_2$ as follows. 

\begin{picture}(340,110)

\put(30,20){\line(1,0){90}} 
\put(40,10){\line(0,1){80}} 

\put(40,20){\circle*{2}} \put(60,20){\circle*{2}} \put(90,20){\circle*{2}}

\put(248,55){$(a,\,0)$} \put(250,32){$(0,\,0)$} 

\put(242,43){$-$}  

\put(230,60){\circle*{2}} \put(230,30){\circle*{2}}

\qbezier(230,100)(220,100)(220,85) \qbezier(220,85)(220,70)(230,70)
\qbezier[10](230,100)(240,100)(240,85) \qbezier[10](240,85)(240,70)(230,70)

\qbezier[10](230,60)(240,60)(240,45) \qbezier[10](240,45)(240,30)(230,30)

\put(170,100){(i)} 

\qbezier(230,100)(150,100)(150,25) 

\qbezier(230,70)(200,65)(230,60) 

\qbezier(230,30)(200,25)(230,20)

\put(205,40){$C_1$} \put(203,70){$C'_2$} \put(165,59){$C_3$} \put(190,20){$C_4$} 

\put(45,10){$C_1$} \put(65,10){$C'_2$} \put(247,70){$C_2$} \put(75,50){$C_3$} \put(43,54){$C_4$} 

\put(90,80){{\rm (i)}}

\thicklines

\qbezier(245,70)(215,65)(245,60) \put(239,69){\line(-1,0){6}} \put(239,69){\line(-1,-1){4}} 

\qbezier(230,60)(220,60)(220,45) \qbezier(220,45)(220,30)(230,30) \put(220,45){\line(-1,-2){3}} \put(220,45){\line(1,-2){3}} 
\qbezier(215,65)(217,62)(230,60) \put(219,62.5){\line(1,-1){4}} \put(219,62.5){\line(1,0){7}}

\qbezier(215,65)(170,65)(170,40) \put(180,58){\line(1,0){6}} \put(180,58){\line(2,3){3}}
\qbezier(170,40)(170,30)(230,30) \put(200,31){\line(-3,-1){7}} \put(200,31){\line(-3,1){7}}

\put(40,20){\line(1,0){20}} \put(90,20){\line(1,0){30}}

\put(40,24){\line(1,0){20}} \put(53,24){\line(-2,1){5}} \put(53,24){\line(-2,-1){5}}

\put(60,24){\line(1,0){20}} \put(70,24){\line(-2,1){5}} \put(70,24){\line(-2,-1){5}}

\qbezier(80,24)(80,50)(60,60)

\put(60,60){\line(-20,-36){20}}

\put(73,49){\line(3,-2){6}} \put(73,49){\line(0,-1){7}}

\put(50,42){\line(0,1){6}} \put(50,42){\line(3,2){5}} 

\end{picture}

In the process, we have also showed that the two equations 
\begin{align*}
&\dfrac{1}{a} \int_0^1 \dfrac{1-a^2t^2}{\sqrt{t(1-t^3)(\frac{1}{a^3}-a^3t^3)}}dt = 
\dfrac{\sqrt{3}}{2} \int_0^1 \dfrac{1+t^2}{\sqrt{t(t^3+a^3)(t^3+\frac{1}{a^3})}}dt, \\
&\int_0^1 \dfrac{1-t^2}{\sqrt{t(t^3+a^3)(t^3+\frac{1}{a^3})}} dt = 
2 \int_a^1 \dfrac{1-t^2}{\sqrt{t(t^3-a^3)(\frac{1}{a^3}-t^3)}} dt + 4 \int_{\frac{1}{2}}^1 \dfrac{dt}{\sqrt{a^3+\frac{1}{a^3}+6t-8t^3}}.
\end{align*}
To determine other paths, we extend $C_3$ in the following way: 
\[
C'_3:=\{(z,\,w)=(e^{it},\,w(t))\>|\>t: 0 \to 2\pi/3,\>\> w(0)\in -i\mathbb{R}_{>0}\}. 
\]

\begin{picture}(340,160)

\put(10,70){\line(1,0){140}} 
\put(70,10){\line(0,1){120}} 

\put(70,70){\circle*{2}} \put(90,70){\circle*{2}} \put(120,70){\circle*{2}} 
\put(57,88){\circle*{2}} \put(40,111){\circle*{2}} \put(57,52){\circle*{2}} \put(40,29){\circle*{2}} 

\put(70,70){\line(-13,18){13}} \put(40,29){\line(-17,-23){12}}
\qbezier(40,111)(30,70)(57,52) 

\put(75,74){$+$} \put(75,62){$-$} \put(120,120){{\rm (i) }} \put(120,77){$C_2$} \put(245,95){$C_1$} \put(10,120){$j(\varphi(C_2))$} 
\put(90,110){$C'_3$} \put(197,70){$C'_3$} \put(220,45){$j(\varphi(C_2))$} 

\put(288,105){$(a,\,0)$} \put(285,82){$(0,\,0)$}

\put(270,110){\circle*{2}} \put(270,80){\circle*{2}} \put(270,70){\circle*{2}} \put(270,40){\circle*{2}} 

\thicklines

\put(70,70){\line(1,0){20}} \put(120,70){\line(1,0){30}} \put(57,88){\line(-17,23){17}} \put(57,52){\line(-17,-23){17}} 

\put(90,74){\line(1,0){30}} \put(103,74){\line(-2,1){5}} \put(103,74){\line(-2,-1){5}}

\put(62,90){\line(-17,23){17}} \put(52,103.5){\line(0,-1){6}} \put(52,103.5){\line(2,-1){6}}

\qbezier(110,70)(95,125)(45,105) \put(89,105){\line(3,-1){6}} \put(89,105){\line(1,-3){2}} 

\thinlines

\qbezier(270,150)(260,150)(260,135) \qbezier(260,135)(260,120)(270,120)
\qbezier[10](270,150)(280,150)(280,135) \qbezier[10](280,135)(280,120)(270,120)

\qbezier[10](270,110)(280,110)(280,95) \qbezier[10](280,95)(280,80)(270,80)

\qbezier[10](270,70)(280,70)(280,55) \qbezier[10](280,55)(280,40)(270,40)

\qbezier(270,30)(260,30)(260,15) \qbezier(260,15)(260,0)(270,0)
\qbezier[10](270,30)(280,30)(280,15) \qbezier[10](280,15)(280,0)(270,0)

\put(210,150){(i)} 

\qbezier(270,150)(190,150)(190,75) \qbezier(190,75)(190,0)(270,0)

\qbezier(270,120)(240,115)(270,110) \qbezier(270,80)(240,75)(270,70)
\qbezier(270,40)(240,35)(270,30) 
\qbezier(270,110)(260,110)(260,95) \qbezier(260,95)(260,80)(270,80)

\put(283,123){$C_2$} \put(285,65){$(e^{\frac{2}{3}\pi i} a ,\,0)$} \put(285,40){$(e^{\frac{2}{3}\pi i}/a,\,0)$}

\thicklines

\qbezier(270,70)(260,70)(260,55) \qbezier(260,55)(260,40)(270,40)
\put(260,55){\line(-1,2){3}} \put(260,55){\line(1,2){3}} 

\qbezier(210,90)(210,55)(260,55) \qbezier(255,115)(210,115)(210,90) 
\put(210,85){\line(-1,2){3}} \put(210,85){\line(1,2){3}} 
\qbezier(285,120)(255,115)(285,110) \put(279,119){\line(-1,0){7}} \put(279,119){\line(-1,-1){4}}

\end{picture}
$\,$\\
\eqref{H-branch1} yields $w(2\pi/3) \in i e^{\frac{\pi}{3} i} \mathbb{R}_{>0}$, and thus $w(2\pi/3)=j(\varphi (w(0)))$. 
It follows that $C'_3 \cup j(\varphi(C_2))=\{ (e^{\frac{2}{3} \pi i},\, i e^{\frac{\pi}{3} i} \sqrt{(1-a^3)(1/a^3-1)} ) \}$. 
As a result, we obtain $j(\varphi(C_2))$ in the previous figure. 

\begin{picture}(340,160)

\put(0,70){\line(1,0){80}} \put(60,10){\line(0,1){120}}

\put(60,70){\circle*{2}} \put(47,88){\circle*{2}} \put(30,111){\circle*{2}} \put(47,52){\circle*{2}} \put(30,29){\circle*{2}} 
 
\put(60,70){\line(-13,18){13}} \put(30,29){\line(-17,-23){12}} \qbezier(30,111)(20,70)(47,52)

\put(65,74){$+$} \put(65,62){$-$} \put(35,110){$+$} \put(28,95){$-$} \put(30,43){$+$} \put(40,34){$-$} 

\put(10,140){{\rm (i) }}

\put(45,105){$j(\varphi (C'_2))$} \put(88,100){$j(\varphi (C_3))$} 
\put(145,100){$j(\varphi (C'_2))$} \put(156,80){$j(\varphi (C_1))$} \put(100,60){$j(\varphi (C_4))$}

\put(100,70){\line(1,0){80}} \put(160,10){\line(0,1){120}}

\put(160,70){\circle*{2}} \put(147,88){\circle*{2}} \put(130,111){\circle*{2}} \put(147,52){\circle*{2}} \put(130,29){\circle*{2}} 
 
\put(160,70){\line(-13,18){13}} \put(130,29){\line(-17,-23){12}} \qbezier(130,111)(120,70)(147,52)

\put(165,72){$+$} \put(165,63){$-$} \put(133,108){$+$} \put(128,100){$-$} \put(130,43){$+$} \put(140,34){$-$} 

\put(110,140){{\rm (ii) }}

\put(195,150){$(\infty,\infty)$} \put(196,122){$(1/a,\,0)$} \put(202,105){$(a,\,0)$} 
\put(202,82){$(0,\,0)$} \put(193,65){$(e^{\frac{2}{3}\pi i} a ,\,0)$} \put(191,40){$(e^{\frac{2}{3}\pi i}/a,\,0)$}
\put(191,24){$(e^{-\frac{2}{3} \pi i} a ,\,0)$} \put(188,0){$(e^{-\frac{2}{3} \pi i}/a,\,0)$} 

\put(262,133){$-$} \put(262,93){$-$} \put(262,56){$-$} \put(262,13){$-$} 

\put(230,133){$+$} \put(230,93){$+$} \put(230,53){$+$} \put(230,13){$+$}

\put(250,150){\circle*{2}} \put(250,120){\circle*{2}} \put(250,110){\circle*{2}} \put(250,80){\circle*{2}} 
\put(250,70){\circle*{2}} \put(250,40){\circle*{2}} \put(250,30){\circle*{2}} \put(250,0){\circle*{2}}

\qbezier(250,150)(240,150)(240,135) \qbezier(240,135)(240,120)(250,120)
\qbezier[10](250,150)(260,150)(260,135) \qbezier[10](260,135)(260,120)(250,120)

\qbezier(250,110)(240,110)(240,95) \qbezier(240,95)(240,80)(250,80)
\qbezier[10](250,110)(260,110)(260,95) \qbezier[10](260,95)(260,80)(250,80)

\qbezier(250,70)(240,70)(240,55) \qbezier(240,55)(240,40)(250,40)
\qbezier[10](250,70)(260,70)(260,55) \qbezier[10](260,55)(260,40)(250,40)

\qbezier(250,30)(240,30)(240,15) \qbezier(240,15)(240,0)(250,0)
\qbezier[10](250,30)(260,30)(260,15) \qbezier[10](260,15)(260,0)(250,0)

\put(295,150){(ii)} 

\qbezier(259,62)(273,64)(285,58)

\qbezier(273,103)(210,80)(251,65)

\thicklines

\qbezier(250,150)(330,150)(330,75) \qbezier(330,75)(330,0)(250,0)

\qbezier(250,120)(280,115)(250,110) \qbezier(250,80)(280,75)(250,70)
\qbezier(250,40)(280,35)(250,30)

\put(160,70){\line(1,0){20}} \put(147,88){\line(-17,23){17}} \put(147,52){\line(-17,-23){17}} 

\put(60,70){\line(1,0){20}} \put(47,88){\line(-17,23){17}} \put(47,52){\line(-17,-23){17}} 

\put(49,91){\line(-17,23){11}} \put(42.5,100){\line(3,-2){6}} \put(42.5,100){\line(1,-3){2}}

\put(144,86){\line(-17,23){11}} \put(137.5,95){\line(3,-2){6}} \put(137.5,95){\line(1,-3){2.5}}

\qbezier(133,100)(120,92)(120,74) \put(120,74){\line(1,0){33}} \put(144,86){\line(17,-23){9}} 

\put(147,82){\line(3,-2){5}} \put(147,82){\line(1,-3){2}} 

\put(123,89){\line(0,1){6}} \put(123,89){\line(2,1){5}} 

\put(137,74){\line(-2,1){5}} \put(137,74){\line(-2,-1){5}}

\qbezier(250,70)(260,70)(260,55) 

\qbezier(320,70)(320,80)(250,80) \qbezier(260,55)(262,40)(264,37) \qbezier(264,34)(320,40)(320,70) 

\put(300,43){\line(-1,0){7}} \put(300,43){\line(-2,-3){4}} \put(280,28){$j(\varphi( C_3 ))$}

\put(290,79){\line(3,1){6}} \put(290,79){\line(3,-2){5}} \put(288,86){$j(\varphi( C_4 ))$}

\put(257,71.5){\line(1,1){4}} \put(257,71.5){\line(4,-1){6}} \put(265,67){$j(\varphi( C_1 ))$}

\put(259,63){\line(-2,1){5}} \put(259,63){\line(1,3){2}} \put(270,49){$j( \varphi(C'_2) )$}

\qbezier(243,70)(253,70)(253,55) \qbezier(253,55)(253,40)(243,40) 
\put(253,55){\line(-1,2){3}} \put(253,55){\line(1,2){3}}

\put(275,102){$j( \varphi(C_2) )$}

\end{picture}

$\,$\\
Note that $j(\varphi(C_1)) \cup j(\varphi ( C'_2 \cup C_3 \cup C_4 ))$ is homotopic to zero. 
Since the upper edges of the thick lines in {\rm (i)} is identified with 
the lower edges of the thick lines in (ii) in {\rm figure (H)}, $j(\varphi(C_1))$, $j(\varphi  (C'_2))$, $j(\varphi( C_3 ))$, 
$ j(\varphi(C_4 ))$ lie in {\rm (ii)} (see the above figure). 

Next, $j(\varphi(C'_3))$ lies in {\rm (ii)} (see the following figure). 
In the following left figure, $j(\varphi(C'_3))$ meets the thin line joining two branch points. 
Hence, $j(\varphi(C'_3))$ is given in the following right figure. Moreover, it is easy to verify that $j(\varphi(C'_3))$ meets 
$\varphi^2 (C_2)(=(j \circ \varphi) (j(\varphi(C_2))))$ at the end in a similar way as the above. 

\begin{picture}(340,160)

\put(0,70){\line(1,0){80}} \put(60,10){\line(0,1){120}}

\put(60,70){\circle*{2}} \put(47,88){\circle*{2}} \put(30,111){\circle*{2}} \put(47,52){\circle*{2}} \put(30,29){\circle*{2}} 
 
\put(60,70){\line(-13,18){13}} \put(30,29){\line(-17,-23){12}} \qbezier(30,111)(20,70)(47,52)

\put(65,74){$+$} \put(65,62){$-$} \put(40,100){$+$} \put(31,91){$-$} \put(32,50){$+$} \put(40,34){$-$} 

\put(-5,24){$\varphi^2 (C_2)$}

\put(10,120){{\rm (ii) }}

\put(195,150){$(\infty,\infty)$} \put(196,122){$(1/a,\,0)$} \put(202,105){$(a,\,0)$} 
\put(202,82){$(0,\,0)$} \put(193,65){$(e^{\frac{2}{3}\pi i} a ,\,0)$} \put(191,40){$(e^{\frac{2}{3}\pi i}/a,\,0)$}
\put(191,24){$(e^{-\frac{2}{3} \pi i} a ,\,0)$} \put(188,0){$(e^{-\frac{2}{3} \pi i}/a,\,0)$} 

\put(192,134){$-$} \put(192,93){$-$} \put(192,53){$-$} \put(192,13){$-$} 
\put(262,133){$-$} \put(262,93){$-$} \put(262,56){$-$} \put(262,13){$-$} 

\put(160,133){$+$} \put(160,93){$+$} \put(160,58){$+$} \put(160,13){$+$} 
\put(230,133){$+$} \put(230,93){$+$} \put(230,53){$+$} \put(230,13){$+$} 

\put(180,150){\circle*{2}} \put(180,120){\circle*{2}} \put(180,110){\circle*{2}} \put(180,80){\circle*{2}} 
\put(180,70){\circle*{2}} \put(180,40){\circle*{2}} \put(180,30){\circle*{2}} \put(180,0){\circle*{2}}

\put(250,150){\circle*{2}} \put(250,120){\circle*{2}} \put(250,110){\circle*{2}} \put(250,80){\circle*{2}} 
\put(250,70){\circle*{2}} \put(250,40){\circle*{2}} \put(250,30){\circle*{2}} \put(250,0){\circle*{2}}

\qbezier(180,150)(170,150)(170,135) \qbezier(170,135)(170,120)(180,120)
\qbezier[10](180,150)(190,150)(190,135) \qbezier[10](190,135)(190,120)(180,120)

\qbezier(180,110)(170,110)(170,95) \qbezier(170,95)(170,80)(180,80)
\qbezier[10](180,110)(190,110)(190,95) \qbezier[10](190,95)(190,80)(180,80)

\qbezier(180,70)(170,70)(170,55) \qbezier(170,55)(170,40)(180,40)
\qbezier[10](180,70)(190,70)(190,55) \qbezier[10](190,55)(190,40)(180,40)

\qbezier(180,30)(170,30)(170,15) \qbezier(170,15)(170,0)(180,0)
\qbezier[10](180,30)(190,30)(190,15) \qbezier[10](190,15)(190,0)(180,0)

\qbezier(250,150)(240,150)(240,135) \qbezier(240,135)(240,120)(250,120)
\qbezier[10](250,150)(260,150)(260,135) \qbezier[10](260,135)(260,120)(250,120)

\qbezier(250,110)(240,110)(240,95) \qbezier(240,95)(240,80)(250,80)
\qbezier[10](250,110)(260,110)(260,95) \qbezier[10](260,95)(260,80)(250,80)

\qbezier(250,70)(240,70)(240,55) \qbezier(240,55)(240,40)(250,40)
\qbezier[10](250,70)(260,70)(260,55) \qbezier[10](260,55)(260,40)(250,40)

\qbezier(250,30)(240,30)(240,15) \qbezier(240,15)(240,0)(250,0)
\qbezier[10](250,30)(260,30)(260,15) \qbezier[10](260,15)(260,0)(250,0)

\put(120,150){(i)} \put(295,150){(ii)} \put(118,70){$C'_3$} \put(283,52){$j(\varphi(C'_3))$}

\qbezier(240,10)(260,0)(280,17) 

\thicklines

\qbezier(180,150)(100,150)(100,75) \qbezier(100,75)(100,0)(180,0)

\qbezier(250,150)(330,150)(330,75) \qbezier(330,75)(330,0)(250,0)

\qbezier(180,120)(150,115)(180,110) \qbezier(180,80)(150,75)(180,70)
\qbezier(180,40)(150,35)(180,30) 

\qbezier(250,120)(280,115)(250,110) \qbezier(250,80)(280,75)(250,70)
\qbezier(250,40)(280,35)(250,30)

\qbezier(260,55)(300,50)(265,35) \qbezier[15](265,35)(260,20)(240,15)
\put(270,37){\line(2,3){4}} \put(270,37){\line(1,0){7}}

\qbezier(250,30)(240,30)(240,15) \qbezier(240,15)(240,0)(250,0)
\put(240,15){\line(-1,2){3}} \put(240,15){\line(1,2){3}} 

\qbezier(35,104)(0,70)(35,36)

\put(17.5,70){\line(-1,2){3}} \put(17.5,70){\line(1,2){3}}

\qbezier(130,90)(130,55)(170,55) \qbezier(165,115)(130,115)(130,90) 
\put(130,85){\line(-1,2){3}} \put(130,85){\line(1,2){3}} 

\put(-11,97){$j(\varphi(C'_3))$}

\put(280,22){$\varphi^2 (C_2)$}

\put(60,70){\line(1,0){20}} \put(47,88){\line(-17,23){17}} \put(47,52){\line(-17,-23){17}} 

\put(44,54){\line(-17,-23){17}} \put(35,42){\line(0,1){5}} \put(35,42){\line(2,1){5}}

\end{picture}

$\,$\\
Finally, $\varphi^2 (C_1)$, $\varphi^2 (C'_2)$, $\varphi^2 (C_3)$, $\varphi^2 (C_4)$ lie in {\rm (i)} and   
$\varphi^2 (C_1) \cup \varphi^2 ( C'_2\cup C_3 \cup C_4 )$ is homotopic to zero. So we obtain the 
following figure. 

\begin{picture}(340,160)

\put(-10,90){\line(1,0){100}} \put(40,30){\line(0,1){120}} 

\put(40,90){\circle*{2}} \put(60,90){\circle*{2}} \put(90,90){\circle*{2}} 
\put(27,108){\circle*{2}} \put(10,131){\circle*{2}} \put(27,72){\circle*{2}} \put(10,49){\circle*{2}}

\put(25,80){$O$} 

\put(40,90){\line(-13,18){13}} \put(10,49){\line(-17,-23){12}}
\qbezier(10,131)(0,90)(27,72) 

\put(47,94){$+$} \put(47,84){$-$} \put(20,120){$+$} \put(11,111){$-$} 
\put(5,53){$+$} \put(13,46){$-$} 

\put(60,140){{\rm (i) }}

\put(100,90){\line(1,0){70}} \put(150,30){\line(0,1){120}} 

\put(150,90){\circle*{2}} \put(170,90){\circle*{2}}  
\put(137,108){\circle*{2}} \put(120,131){\circle*{2}} \put(137,72){\circle*{2}} \put(120,49){\circle*{2}}

\put(140,80){$O$} 

\put(150,90){\line(-13,18){13}} \put(120,49){\line(-17,-23){12}}
\qbezier(120,131)(110,90)(137,72) 

\put(155,94){$+$} \put(155,82){$-$} \put(130,120){$+$} \put(121,111){$-$} 
\put(120,69){$+$} \put(128,52){$-$} 

\put(130,140){{\rm (ii) }} 

\qbezier(46,32)(35,35)(30,56) \qbezier(0,65)(5,55)(25,62) \qbezier(34,75)(35,105)(55,105)

\thicklines

\put(40,90){\line(1,0){20}} \put(27,108){\line(-17,23){17}} \put(27,72){\line(-17,-23){17}} 
 
\put(150,90){\line(1,0){20}} \put(137,108){\line(-17,23){17}} \put(137,72){\line(-17,-23){17}} 
  
\put(133,74){\line(-17,-23){17}} \put(122,59.5){\line(1,4){1.5}} \put(122,59.5){\line(2,1){6}} 

\put(90,60){$\varphi^2 (C_2)$}

\qbezier(13,53)(65,45)(85,90) \put(65,63){\line(-1,0){6}} \put(65,63){\line(-1,-2){3}} 
\put(50,47){$\varphi^2 (C'_3)$}

\put(32,72){\line(-17,-23){11}} \put(36,77){\line(0,1){6}} \put(36,77){\line(2,1){6}} 
\put(32,72){\line(17,23){11}} \put(26,64){\line(0,1){6}} \put(26,64){\line(2,1){6}} 

\qbezier(21,57)(37,53)(55,63) \put(42,57.5){\line(-5,1){6}} \put(42,57.5){\line(-1,-1){4}}

\put(55,63){\line(-1,2){12}} \put(49,75){\line(0,-1){6}} \put(49,75){\line(3,-2){5}} 

\put(50,78){$\varphi^2 (C_4)$} \put(50,30){$\varphi^2 (C_3)$} 
\put(-10,70){$\varphi^2 (C'_2)$} \put(60,104){$\varphi^2 (C_1)$}

\thinlines

\put(200,150){{\rm (i) }} 

\put(275,150){$(\infty,\infty)$} \put(276,122){$(1/a,\,0)$} \put(282,105){$(a,\,0)$} 
\put(282,82){$(0,\,0)$} \put(273,65){$(e^{\frac{2}{3}\pi i} a ,\,0)$} \put(271,40){$(e^{\frac{2}{3}\pi i}/a,\,0)$}
\put(271,24){$(e^{-\frac{2}{3} \pi i} a ,\,0)$} \put(268,0){$(e^{-\frac{2}{3} \pi i}/a,\,0)$} 

\put(272,134){$-$} \put(272,93){$-$} \put(272,53){$-$} \put(272,13){$-$} 
\put(240,133){$+$} \put(240,93){$+$} \put(240,58){$+$} \put(240,13){$+$} 

\put(260,150){\circle*{2}} \put(260,120){\circle*{2}} \put(260,110){\circle*{2}} \put(260,80){\circle*{2}} 
\put(260,70){\circle*{2}} \put(260,40){\circle*{2}} \put(260,30){\circle*{2}} \put(260,0){\circle*{2}}

\qbezier(260,150)(250,150)(250,135) \qbezier(250,135)(250,120)(260,120)
\qbezier[10](260,150)(270,150)(270,135) \qbezier[10](270,135)(270,120)(260,120)

\qbezier(260,110)(250,110)(250,95) \qbezier(250,95)(250,80)(260,80)
\qbezier[10](260,110)(270,110)(270,95) \qbezier[10](270,95)(270,80)(260,80)

\qbezier(260,70)(250,70)(250,55) \qbezier(250,55)(250,40)(260,40)
\qbezier[10](260,70)(270,70)(270,55) \qbezier[10](270,55)(270,40)(260,40)

\qbezier(260,30)(250,30)(250,15) \qbezier(250,15)(250,0)(260,0)
\qbezier[10](260,30)(270,30)(270,15) \qbezier[10](270,15)(270,0)(260,0)

\thicklines

\qbezier(260,150)(180,150)(180,75) \qbezier(180,75)(180,0)(260,0)

\qbezier(260,120)(230,115)(260,110) \qbezier(260,80)(230,75)(260,70)
\qbezier(260,40)(230,35)(260,30) 

\qbezier[20](260,30)(230,25)(230,55) \qbezier[20](230,55)(230,82)(260,80) 
\qbezier(260,30)(270,30)(270,15) \qbezier(270,15)(270,10)(251,10)
\qbezier[25](251,10)(190,10)(190,55) \qbezier[25](190,55)(190,90)(260,80)

\qbezier[6](230,50)(229,52)(227,57) \qbezier[6](230,50)(231,52)(233,57) \put(200,40){$\varphi^2 (C_1)$}
\qbezier[4](219,81)(217,82)(213,84) \qbezier[5](219,81)(217,81)(213,77) \put(200,89){$\varphi^2 (C_4)$}
\qbezier[5](217,15)(220,16)(223,17) \qbezier[5](217,15)(219,12)(221,9) \put(170,0){$\varphi^2 (C_3)$}
\put(269,23){\line(0,1){7}} \put(269,23){\line(-2,1){6}} \put(155,25){$\varphi^2 (C'_2)$}

\thinlines

\qbezier(187,27)(230,27)(270,20) \qbezier(202,3)(210,5)(212,17)

\end{picture}

$\,$\\
Therefore, setting 
\begin{align*}
A_3 &:= -\varphi (C_2)-\varphi (C_1)+\varphi^2 (C_1) +j (\varphi (C_2))+j(\varphi (C_1))-j(\varphi^2 (C_1)), \\
B_3 &:= \varphi^2 (C_2) -j(\varphi^2 (C_2)),\quad 
A_2 :=   \varphi (C_1) -j(\varphi (C_1)), \quad A_1 :=  -C_2 +j(C_2), \\
B_2 &:=  \varphi (C_2) -j(\varphi (C_2))+B_3, 
B_1 := C_1 -j(C_1)+B_2, 
\end{align*}
we obtain a canonical homology basis as follows. 

\begin{picture}(340,100)

\qbezier(20,50)(20,90)(160,90) \qbezier(160,90)(300,90)(300,50) 
\qbezier(20,50)(20,10)(160,10) \qbezier(160,10)(300,10)(300,50) 

\qbezier(55,55)(80,35)(105,55) \qbezier(135,55)(160,35)(185,55) \qbezier(215,55)(240,35)(265,55) 
\qbezier(60,52)(80,62)(100,52) \qbezier(140,52)(160,62)(180,52) \qbezier(220,52)(240,62)(260,52)

\put(30,85){{\rm (ii)}} \put(30,15){{\rm (i)}} 

\qbezier(45,50)(45,30)(80,30) \qbezier(115,50)(115,30)(80,30) 
\qbezier(45,50)(45,70)(80,70) \qbezier(115,50)(115,70)(80,70) 

\qbezier(125,50)(125,30)(160,30) \qbezier(195,50)(195,30)(160,30) 
\qbezier(125,50)(125,70)(160,70) \qbezier(195,50)(195,70)(160,70) 

\qbezier(205,50)(205,30)(240,30) \qbezier(275,50)(275,30)(240,30) 
\qbezier(205,50)(205,70)(240,70) \qbezier(275,50)(275,70)(240,70) 

\qbezier(71,46)(55,38)(55,20) \qbezier[15](55,20)(75,25)(71,46) 

\qbezier(151,46)(135,35)(140,10) \qbezier[15](140,10)(160,25)(151,46) 

\qbezier(249,46)(265,38)(265,20) \qbezier[15](249,46)(243,25)(265,20)

\put(80,30){\line(-2,1){5}} \put(80,30){\line(-2,-1){5}} \put(100,25){$A_1$}
\put(160,30){\line(-2,1){5}} \put(160,30){\line(-2,-1){5}} \put(180,25){$A_2$}
\put(240,30){\line(-2,1){5}} \put(240,30){\line(-2,-1){5}} \put(275,35){$A_3$}

\put(62.5,40){\line(-2,-1){5}}  \put(62.5,40){\line(0,-1){5}}  \put(50,10){$B_1$}
\put(144.5,40){\line(-2,-1){5}} \put(144.5,40){\line(0,-1){5}} \put(135,0){$B_2$}
\put(257.5,40){\line(2,-1){5}} \put(257.5,40){\line(0,-1){5}} \put(260,10){$B_3$}

\end{picture}

\subsubsection{Period matrix}\label{H-periods-arg}

Key $1$-cycles of the canonical homology basis as in \S~\ref{H-canonical} are given by $C_1 \cup \{-j(C_1)\}$ and 
$C_2 \cup \{-j(C_2)\}$. Since the two key $1$-cycles meet the pole $(a,\,0)$ of the $z^j/w^3 dz$'s, we introduce two useful $1$-cycles. 
First we extend $C_3$ in the following way. \\

\begin{picture}(340,100)

\put(30,50){\line(1,0){90}} 
\put(40,0){\line(0,1){100}} 

\put(40,50){\circle*{2}} \put(60,50){\circle*{2}} \put(90,50){\circle*{2}}

\put(248,45){$(a,\,0)$} \put(250,22){$(0,\,0)$} 

\put(242,33){$-$}  

\put(230,50){\circle*{2}} \put(230,20){\circle*{2}}

\qbezier(230,90)(220,90)(220,75) \qbezier(220,75)(220,60)(230,60)
\qbezier[10](230,90)(240,90)(240,75) \qbezier[10](240,75)(240,60)(230,60)

\qbezier[10](230,50)(240,50)(240,35) \qbezier[10](240,35)(240,20)(230,20)

\put(170,90){(i)} 

\qbezier(230,90)(150,90)(150,15) 

\qbezier(230,60)(200,55)(230,50) 

\qbezier(230,20)(200,15)(230,10)

\put(205,30){$C_1$} \put(165,49){$C''_3$} \put(190,25){$C_4$} 

\put(45,40){$C_1$} \put(75,80){$C''_3$} \put(43,84){$C_4$} \put(20,14){$\varphi^2 (C_4)$} \put(180,4){$\varphi^2 (C_4)$} 

\put(90,70){{\rm (i)}}

\thicklines

\qbezier(230,50)(220,50)(220,35) \qbezier(220,35)(220,20)(230,20) \put(220,35){\line(-1,-2){3}} \put(220,35){\line(1,-2){3}} 

\qbezier(215,55)(170,55)(170,30) \put(180,48){\line(1,0){6}} \put(180,48){\line(2,3){3}}
\qbezier(170,30)(170,20)(230,20) \put(200,21){\line(-3,-1){7}} \put(200,21){\line(-3,1){7}}

\put(40,50){\line(1,0){20}} \put(90,50){\line(1,0){30}}

\put(40,54){\line(1,0){20}} \put(53,54){\line(-2,1){5}} \put(53,54){\line(-2,-1){5}}

\put(60,54){\line(1,0){20}} \put(70,54){\line(-2,1){5}} \put(70,54){\line(-2,-1){5}}

\qbezier(80,54)(80,80)(60,90)

\put(60,90){\line(-20,-36){20}}

\put(73,79){\line(3,-2){6}} \put(73,79){\line(0,-1){7}}

\put(50,72){\line(0,1){6}} \put(50,72){\line(3,2){5}}

\qbezier(80,46)(80,20)(60,10)

\put(60,10){\line(-20,36){20}}

\put(73,21){\line(0,-1){7}} \put(73,21){\line(-2,-1){7}}

\put(50,28){\line(0,-1){6}} \put(50,28){\line(3,-2){5}}

\qbezier[20](230,20)(190,13)(170,15) \qbezier[20](215,55)(160,30)(170,15) 

\qbezier[3](186,41)(183,41)(181,41) \qbezier[3](187,39)(186,36)(185,34) 
\qbezier[3](188,15)(185,16)(182,17) \qbezier[3](188,14)(183,13)(181,12) 

\end{picture}

\[
C''_3 := \{(z,\,w)=(e^{it},\,w(t))\>|\>t: -\pi/3 \to \pi/3,\>\> w(0)\in -i\mathbb{R}_{>0}\}. 
\]
$\,$\\
By the arguments as in \S~\ref{H-canonical}, we conclude that $C_1 \cup \{-j(C_1)\}$ is homotopic to 
$(-C_4) \cup (-C''_3) \cup \varphi^2 (C_4)$. 
Next we introduce the following extension of $C_4$. 
\[
C'_4=\{(z,\,w)=(e^{\frac{\pi}{3}i}t,\,e^{\frac{\pi}{6}i}\sqrt{t(t^3+a^3)(t^3+1/a^3)})\>|\>t:\infty\to 0, \> \sqrt{*}>0\}. 
\]

\begin{picture}(340,120)

\put(0,70){\line(1,0){90}} \put(10,65){\line(0,1){55}} 

\put(0,50){\line(1,0){90}} \put(10,0){\line(0,1){55}} 

\put(10,70){\circle*{2}} \put(30,70){\circle*{2}} \put(60,70){\circle*{2}} 

\put(10,50){\circle*{2}} \put(30,50){\circle*{2}} \put(60,50){\circle*{2}} 

\put(198,45){$(a,\,0)$} \put(200,22){$(0,\,0)$} \put(195,58){$(1/a,\,0)$} \put(196,85){$(\infty,\,\infty)$} 

\put(192,33){$-$}  

\put(180,50){\circle*{2}} \put(180,20){\circle*{2}}  \put(180,60){\circle*{2}} \put(180,90){\circle*{2}}  

\put(250,50){\circle*{2}} \put(250,20){\circle*{2}}  \put(250,60){\circle*{2}} \put(250,90){\circle*{2}}  

\qbezier(180,90)(170,90)(170,75) \qbezier(170,75)(170,60)(180,60)
\qbezier[10](180,90)(190,90)(190,75) \qbezier[10](190,75)(190,60)(180,60)
\qbezier[10](180,50)(190,50)(190,35) \qbezier[10](190,35)(190,20)(180,20)

\qbezier(250,90)(240,90)(240,75) \qbezier(240,75)(240,60)(250,60)
\qbezier[10](250,90)(260,90)(260,75) \qbezier[10](260,75)(260,60)(250,60)
\qbezier[10](250,50)(260,50)(260,35) \qbezier[10](260,35)(260,20)(250,20)

\put(130,90){(i)} \put(60,35){(ii)} \put(300,90){(ii)}

\qbezier(180,90)(100,90)(100,15) \qbezier(250,90)(330,90)(330,15) 

\qbezier(180,60)(150,55)(180,50) \qbezier(250,60)(280,55)(250,50) 

\qbezier(180,20)(150,15)(180,10) \qbezier(250,20)(280,15)(250,10)

\put(125,37){$C'_4$} 

\put(35,60){$C_2$} \put(13,104){$C'_4$} \put(35,15){$j(\varphi^2 (C'_4))$} \put(150,54){$C_2$} \put(270,54){$j(C_2)$} 
\put(260,28){$j(\varphi^2 (C'_4))$} 

\put(60,90){{\rm (i)}}

\qbezier(180,50)(170,50)(170,35) \qbezier(170,35)(170,20)(180,20) 

\qbezier(250,50)(240,50)(240,35) \qbezier(240,35)(240,20)(250,20)

\thicklines

\qbezier(180,60)(150,55)(180,50) \put(168,53){\line(1,-2){2}} \put(168,53){\line(5,1){5}} 

\qbezier(250,60)(280,55)(250,50) \put(262,53){\line(-1,-2){2}} \put(262,53){\line(-5,1){5}} 

\qbezier(180,90)(120,90)(120,30) \put(123,55){\line(0,1){8}} \put(123,55){\line(1,1){6}}
\qbezier(120,30)(120,20)(180,20) 

\qbezier(250,90)(310,90)(310,30) \put(307,55){\line(0,1){8}} \put(307,55){\line(-1,1){6}}
\qbezier(310,30)(310,20)(250,20) 

\put(10,70){\line(1,0){20}} \put(10,50){\line(1,0){20}} 

\put(30,74){\line(1,0){30}} \put(50,74){\line(-2,1){5}} \put(50,74){\line(-2,-1){5}}

\put(30,105){\line(-20,-36){20}} \put(30,105){\line(20,36){8}}

\put(20,87){\line(0,1){6}} \put(20,87){\line(3,2){5}} 

\put(30,15){\line(-20,36){20}} \put(30,15){\line(20,-36){8}}

\put(20,32){\line(0,-1){6}} \put(20,32){\line(3,-2){5}} 

\end{picture}

$\,$\\
The previous arguments as in \S~\ref{H-canonical} imply that $C_2 \cup \{-j(C_2)\}$ is homotopic to 
$(-C'_4) \cup j (\varphi^2(C'_4))$. 
Hence we may consider $(-C_4) \cup (-C''_3) \cup \varphi^2 (C_4)$ and $(-C'_4) \cup j (\varphi^2(C'_4))$ 
instead of $C_1 \cup \{-j(C_1)\}$ and $C_2 \cup \{-j(C_2)\}$, respectively. 

Setting $x=(z+1/z)/2=\cos t:1/2\to 1\to 1/2$ along $C''_3$, we find 
\[
\dfrac{z^2-1}{2i z} =\sin t = 
\begin{cases}
-\sqrt{1-x^2} & \left( t:-\frac{\pi}{3} \to 0 \right) \\
\sqrt{1-x^2} &  \left( t: 0\to \frac{\pi}{3} \right)
\end{cases}
\]
It follows from \eqref{H-branch1} that 
\begin{align*}
&\int_{C''_3} \dfrac{1-z^2}{w}dz =-2i \int_{\frac{1}{2}}^1 \dfrac{ dx}{\sqrt{a^3+\frac{1}{a^3}+6x-8x^3}} 
 -2i  \int^{\frac{1}{2}}_1 \dfrac{ dx}{\sqrt{a^3+\frac{1}{a^3}+6x-8x^3}}=0, \\
&\int_{C''_3} \dfrac{i(1+z^2)}{w}dz =-2i \int_{\frac{1}{2}}^1 \dfrac{ x }{\sqrt{(a^3+\frac{1}{a^3}+6x-8x^3)(1-x^2)}} dx\\
& +2i \int^{\frac{1}{2}}_1 \dfrac{ x }{\sqrt{(a^3+\frac{1}{a^3}+6x-8x^3)(1-x^2)}} dx 
=-4i \int_{\frac{1}{2}}^1 \dfrac{ x }{\sqrt{(a^3+\frac{1}{a^3}+6x-8x^3)(1-x^2)}} dx, \\
&\int_{C''_3} \dfrac{2 z}{w}dz = -4 \int_{\frac{1}{2}}^1 \dfrac{ dx }{\sqrt{(a^3+\frac{1}{a^3}+6x-8x^3)(1-x^2)}}, \\
&\int_{C''_3} \dfrac{z^4-z^6}{w^3}dz = 0, \\
&\int_{C''_3} \dfrac{i(z^4+z^6)}{w^3}dz = 4 i \int_{\frac{1}{2}}^1 \dfrac{x}{\sqrt{ a^3+\frac{1}{a^3}+6x-8x^3 }^3 \sqrt{1-x^2}} dx,  
\end{align*}
\begin{align*}
&\int_{C''_3} \dfrac{z^5}{w^3}dz = 2 \int_{\frac{1}{2}}^1 \dfrac{dx}{\sqrt{ a^3+\frac{1}{a^3}+6x-8x^3 }^3 \sqrt{1-x^2}}. 
\end{align*}
Straightforward calculation yields 
\begin{align*}
& \int_{C_4} \dfrac{i(1+z^2)}{w}dz = \dfrac{1}{2} \int_0^1 \dfrac{1+t^2}{\sqrt{t(t^3+a^3)(t^3+\frac{1}{a^3})}} dt 
  -\dfrac{\sqrt{3}}{2} i \int_0^1 \dfrac{1-t^2}{\sqrt{t(t^3+a^3)(t^3+\frac{1}{a^3})}} dt, \\
&\int_{C_4} \dfrac{2 z}{w}dz =  -2 i \int_0^1 \dfrac{t}{\sqrt{t(t^3+a^3)(t^3+\frac{1}{a^3})}} dt,  \\
& \int_{C_4} \dfrac{z^4-z^6}{w^3}dz = \dfrac{i}{2} \int_0^1\dfrac{t^4-t^6}{\sqrt{t(t^3+a^3)(t^3+\frac{1}{a^3})}^3}dt 
 +\dfrac{\sqrt{3}}{2}\int_0^1 \dfrac{t^4+t^6}{\sqrt{t(t^3+a^3)(t^3+\frac{1}{a^3})}^3}dt, \\
& \int_{C_4} \dfrac{i(z^4+z^6)}{w^3}dz = \dfrac{\sqrt{3}}{2} i \int_0^1\dfrac{t^4-t^6}{\sqrt{t(t^3+a^3)(t^3+\frac{1}{a^3})}^3}dt 
 -\dfrac{1}{2}\int_0^1 \dfrac{t^4+t^6}{\sqrt{t(t^3+a^3)(t^3+\frac{1}{a^3})}^3}dt, \\
&\int_{C_4} \dfrac{z^5}{w^3}dz = i \int_0^1\dfrac{t^5}{\sqrt{t(t^3+a^3)(t^3+\frac{1}{a^3})}^3}dt.
\end{align*}
Combining these equations and \eqref{H-inform3}, we have 
\begin{align}
\int_{(-C_4) \cup (-C''_3) \cup \varphi^2 (C_4)} \dfrac{1-z^2}{w}dz &=  
\sqrt{3} \int_0^1 \dfrac{1+t^2}{\sqrt{t(t^3+a^3)(t^3+\frac{1}{a^3})}} dt,   \label{H-key-integral1}  \\
\nonumber \int_{(-C_4) \cup (-C''_3) \cup \varphi^2 (C_4)} \dfrac{i( 1+z^2)}{w}dz &=    
\sqrt{3} i  \int_0^1 \dfrac{1-t^2}{\sqrt{t(t^3+a^3)(t^3+\frac{1}{a^3})}} dt  \\
& +4i \int_{\frac{1}{2}}^1 \dfrac{ x }{\sqrt{(a^3+\frac{1}{a^3}+6x-8x^3)(1-x^2)}} dx, \label{H-key-integral2} \\
\int_{(-C_4) \cup (-C''_3) \cup \varphi^2 (C_4)} \dfrac{2 z}{w}dz &=   
4 \int_{\frac{1}{2}}^1 \dfrac{ d x }{\sqrt{(a^3+\frac{1}{a^3}+6x-8x^3)(1-x^2)}},  \label{H-key-integral3}   \\
\int_{(-C_4) \cup (-C''_3) \cup \varphi^2 (C_4)} \dfrac{z^4-z^6}{w^3}dz &=
-\sqrt{3} \int_0^1 \dfrac{t^4+t^6}{\sqrt{t(t^3+a^3)(t^3+\frac{1}{a^3})}^3}dt, \label{H-key-integral4}  \\
\nonumber \int_{(-C_4) \cup (-C''_3) \cup \varphi^2 (C_4)} \dfrac{i(z^4+z^6)}{w^3}dz &= 
-\sqrt{3} i \int_0^1 \dfrac{t^4-t^6}{\sqrt{t(t^3+a^3)(t^3+\frac{1}{a^3})}^3}dt 
\end{align}
\begin{align}
& -4 i \int_{\frac{1}{2}}^1 \dfrac{x}{\sqrt{ a^3+\frac{1}{a^3}+6x-8x^3 }^3 \sqrt{1-x^2}} dx,  \label{H-key-integral5}  \\
\int_{(-C_4) \cup (-C''_3) \cup \varphi^2 (C_4)} \dfrac{z^5}{w^3}dz &= 
-2 \int_{\frac{1}{2}}^1 \dfrac{dx}{\sqrt{ a^3+\frac{1}{a^3}+6x-8x^3 }^3 \sqrt{1-x^2}}. \label{H-key-integral6} 
\end{align}
Similarly we obtain 
\begin{align*}
&\int_{C'_4} \dfrac{1-z^2}{w}dz = -\dfrac{\sqrt{3}}{2} \int_0^{\infty} \dfrac{ 1+t^2 }{\sqrt{ t(t^3+a^3)(t^3+\frac{1}{a^3}) }}dt 
 -\dfrac{i}{2} \int_0^{\infty} \dfrac{ 1-t^2 }{\sqrt{ t(t^3+a^3)(t^3+\frac{1}{a^3}) }}dt    \\
&= -\sqrt{3} \int_0^1 \dfrac{ 1+t^2 }{\sqrt{ t(t^3+a^3)(t^3+\frac{1}{a^3}) }}dt, \\
&\int_{C'_4} \dfrac{i(1+z^2)}{w}dz =  \int_0^1 \dfrac{ 1+t^2 }{\sqrt{ t(t^3+a^3)(t^3+\frac{1}{a^3}) }}dt,  \\
&\int_{C'_4} \dfrac{ 2 z}{w}dz =  -4 i \int_0^1 \dfrac{ t }{\sqrt{ t(t^3+a^3)(t^3+\frac{1}{a^3}) }}dt,  \\
&\int_{C'_4} \dfrac{z^4-z^6}{w^3}dz =   \sqrt{3} \int_0^1 \dfrac{ t^4+t^6 }{\sqrt{ t(t^3+a^3)(t^3+\frac{1}{a^3}) }^3}dt,  \\
&\int_{C'_4} \dfrac{i(z^4+z^6)}{w^3}dz =  -\int_0^1 \dfrac{ t^4+t^6 }{\sqrt{ t(t^3+a^3)(t^3+\frac{1}{a^3}) }^3}dt,  \\
&\int_{C'_4} \dfrac{z^5}{w^3}dz = 2 i \int_0^1 \dfrac{ t^5 }{\sqrt{ t(t^3+a^3)(t^3+\frac{1}{a^3}) }^3}dt.     
\end{align*}
Hence we find 
\begin{align}
\int_{(-C'_4) \cup j (\varphi^2(C'_4))} \dfrac{1-z^2}{w}dz &= 0, \label{H-key-integral7} \\
\int_{(-C'_4) \cup j (\varphi^2(C'_4))} \dfrac{i(1+z^2)}{w}dz &= -2 \int_0^1 \dfrac{ 1+t^2 }{\sqrt{ t(t^3+a^3)(t^3+\frac{1}{a^3}) }}dt, 
\label{H-key-integral8}      \\ 
\int_{(-C'_4) \cup j (\varphi^2(C'_4))} \dfrac{ 2 z}{w}dz &=  8 i \int_0^1 \dfrac{ t }{\sqrt{ t(t^3+a^3)(t^3+\frac{1}{a^3}) }}dt, 
\label{H-key-integral9}  \\
\int_{(-C'_4) \cup j (\varphi^2(C'_4))} \dfrac{z^4-z^6}{w^3}dz &=  0, \label{H-key-integral10} \\
\int_{(-C'_4) \cup j (\varphi^2(C'_4))} \dfrac{i(z^4+z^6)}{w^3}dz &=  2 \int_0^1 \dfrac{ t^4+t^6 }{\sqrt{ t(t^3+a^3)(t^3+\frac{1}{a^3}) }^3}dt, 
\label{H-key-integral11} 
\end{align}
\begin{align}
\int_{(-C'_4) \cup j (\varphi^2(C'_4))} \dfrac{z^5}{w^3}dz &= -4 i \int_0^1 \dfrac{ t^5 }{\sqrt{ t(t^3+a^3)(t^3+\frac{1}{a^3}) }^3}dt. 
\label{H-key-integral12}   
\end{align}
Set 
\begin{align*}
A &:= \int_0^1 \dfrac{1+t^2}{\sqrt{t(t^3+a^3)(t^3+\frac{1}{a^3})}} dt, \\ 
B &:= \sqrt{3} \int_0^1 \dfrac{1-t^2}{\sqrt{t(t^3+a^3)(t^3+\frac{1}{a^3})}} dt
+ 4\int_{\frac{1}{2}}^1 \dfrac{x}{\sqrt{ (a^3+\frac{1}{a^3}+6x-8x^3)(1-x^2) }} dx , \\
C &:= 4 \int_{\frac{1}{2}}^1 \dfrac{dx}{\sqrt{ (a^3+\frac{1}{a^3}+6x-8x^3)(1-x^2) }}, \\
D &:= 8\int_0^1 \dfrac{t}{\sqrt{t(t^3+a^3)(t^3+\frac{1}{a^3})}} dt, \quad 
E :=  \int_0^1 \dfrac{t^4+t^6}{\sqrt{t(t^3+a^3)(t^3+\frac{1}{a^3})}^3} dt, \\
F &:= \sqrt{3} \int_0^1 \dfrac{t^4-t^6}{\sqrt{t(t^3+a^3)(t^3+\frac{1}{a^3})}^3} dt
+ 4\int_{\frac{1}{2}}^1 \dfrac{x}{\sqrt{ a^3+\frac{1}{a^3}+6x-8x^3}^3\sqrt{1-x^2}} dx, \\
H &:= 2 \int_{\frac{1}{2}}^1 \dfrac{dx}{\sqrt{ a^3+\frac{1}{a^3}+6x-8x^3}^3\sqrt{1-x^2}}, \quad 
I := 4 \int_0^1 \dfrac{t^5}{\sqrt{t(t^3+a^3)(t^3+\frac{1}{a^3})}^3} dt. 
\end{align*}
Using \eqref{H-key-integral1}--\eqref{H-key-integral12}, we have 
\begin{align}
\int_{(-C_4) \cup (-C''_3) \cup \varphi^2 (C_4)} G
=
i\begin{pmatrix}
\sqrt{3} A \\
i B \\
C \\
-\sqrt{3} E \\
-i F \\
-H
\end{pmatrix},\quad 
\int_{(-C'_4) \cup j (\varphi^2(C'_4))} G
=i
\begin{pmatrix}
0 \\
-2 A \\
i D \\
0 \\
2 E \\
-i I
\end{pmatrix}. 
\end{align}
Therefore, the period matrix of the abelian differentials of the second kind, that is, 
$
\begin{pmatrix}
\displaystyle{\int}_{A_1} G & \displaystyle{\int_{A_2}} G & \displaystyle{\int_{A_3}} G & 
\displaystyle{\int_{B_1}} G & \displaystyle{\int_{B_2}} G & \displaystyle{\int_{B_3}} G  
\end{pmatrix}
$ is given by 
\[
i
\begin{pmatrix}
0  &  \dfrac{\sqrt{3}}{2} (A+i B)  &   0  &  -\sqrt{3} A  &  -2\sqrt{3} A  &  -\sqrt{3} A  \\
2 A  &   \dfrac{-3 A+ i B}{2}   &  A- i B  & i B  &  0   &  A  \\
-i D  &   -C  &   2C+i D  &  C  &  0  &  i D  \\
0  &  -\dfrac{\sqrt{3}}{2} (E+i F)   &  0  &  \sqrt{3} E  &  2\sqrt{3} E  &  \sqrt{3} E  \\
-2 E  &   \dfrac{3 E- i F}{2}   &   -E+iF  &  -i F  &  0  &  -E  \\
i I  &  H   &   -2H-i I   &  -H  &  0  &  -i I
\end{pmatrix}.
\]

The arguments in next subsections are based on the above techniques, so we omit some details. 

\subsection{rPD family}\label{rPD-detail}

\subsubsection{Canonical homology basis}\label{rPD-canonical}

Let $M$ be a hyperelliptic Riemann surface of genus $3$ defined as the completion of 
$\{(z,\,w)\,|\, w^2=z(z^3-a^3) \left( z^3+\frac{1}{a^3} \right)\} \subset \mathbb{C}^2$ for $a\in (0,\,1]$. 
The three differentials 
\[\dfrac{dz}{w},\,z\dfrac{dz}{w},\,z^2\dfrac{dz}{w}\]
form a basis for the abelian differentials of the first kind. 
Up to exact forms, the abelian differentials of the second kind are given by the following six differentials: 
\[\dfrac{dz}{w},\,z\dfrac{dz}{w},\,z^2\dfrac{dz}{w},\,\dfrac{z^4}{w^3}dz,\,\dfrac{z^5}{w^3}dz,\,\dfrac{z^6}{w^3}dz.\]

Let 
\[G=i \left(\dfrac{1-z^2}{w},\,\dfrac{i(1+z^2)}{w},\,\dfrac{2z}{w},\,
\dfrac{z^4-z^6}{w^3},\,\dfrac{i(z^4+z^6)}{w^3},\,\dfrac{z^5}{w^3}\right)^t dz \]
and consider the biholomorphisms 
\begin{align*}
j(z,\,w)&=(z,\,-w),\>\>\varphi(z,\,w)=(e^{\frac{2\pi}{3}i}z,\,e^{\frac{\pi}{3}i}w)
\end{align*}
on $M$. Then it is straightforward to compute that 
\begin{align*}
j^{*}G=-G, \quad 
\varphi^{*}G & =
\begin{pmatrix}
\dfrac{1}{2} & \dfrac{\sqrt{3}}{2}  & 0 & 0 & 0 & 0 \\
-\dfrac{\sqrt{3}}{2} & \dfrac{1}{2}  & 0 & 0 & 0 & 0 \\
0 & 0 & -1 & 0 & 0 & 0 \\
0 & 0 & 0 & \dfrac{1}{2} & \dfrac{\sqrt{3}}{2}  & 0 \\
0 & 0 & 0 & -\dfrac{\sqrt{3}}{2} & \dfrac{1}{2}  & 0 \\
0 & 0 & 0 & 0 & 0 & -1 
\end{pmatrix}G.  
\end{align*}
Now we determine a canonical homology basis on $M$. Recall that 
\begin{align*}
\pi_{{\rm rPD}}:&\quad M \longrightarrow \overline{\mathbb{C}}:=\mathbb{C} \cup \{\infty\} \\
& (z,\,w) \> \longmapsto \> z
\end{align*}
defines a two-sheeted branched covering having branch locus
\[
\{ (0,\,0),\>(\infty,\,\infty),\> (a,\,0),\> (e^{\pm \frac{2}{3}\pi i}a,\,0),\>  
(-1/a,\,0),\> (e^{\pm \frac{\pi}{3} i}/a,\,0) \}
\]
and $j$ is its deck transformation. 
So $M$ can be expressed as a $2$-sheeted branched cover of $\overline{\mathbb{C}}$ in the following way. 

\begin{picture}(340,150)

\put(10,90){\line(1,0){140}} \put(170,90){\line(1,0){140}}
\put(70,30){\line(0,1){120}} \put(230,30){\line(0,1){120}}

\put(70,90){\circle*{2}} \put(90,90){\circle*{2}} \put(20,90){\circle*{2}} 
\put(57,108){\circle*{2}} \put(110,131){\circle*{2}} \put(57,72){\circle*{2}} \put(110,49){\circle*{2}} 

\put(230,90){\circle*{2}} \put(250,90){\circle*{2}} \put(180,90){\circle*{2}} 
\put(217,108){\circle*{2}} \put(270,131){\circle*{2}} \put(217,72){\circle*{2}} \put(270,49){\circle*{2}} 
 
\put(60,80){$O$} \put(88,94){$a$} \put(0,96){$-1/a$} \put(35,112){$e^{\frac{2}{3} \pi i} a$} \put(85,135){$e^{\frac{\pi}{3} i} /a$}
\put(25,60){$e^{-\frac{2}{3} \pi i} a$} \put(80,40){$e^{-\frac{\pi}{3} i} /a$}

\put(110,131){\line(-53,-23){53}} \put(20,90){\line(37,-18){37}} \put(110,49){\line(1,-1){15}} 
\put(270,131){\line(-53,-23){53}} \put(180,90){\line(37,-18){37}} \put(270,49){\line(1,-1){15}} 

\put(90,105){$+$} \put(80,110){$-$} \put(115,94){$+$} \put(115,82){$-$} \put(30,100){$+$} \put(35,92){$-$} 
\put(80,65){$-$} \put(73,52){$+$} 
\put(250,105){$+$} \put(240,110){$-$} \put(275,94){$+$} \put(275,82){$-$} \put(190,100){$+$} \put(195,92){$-$} 
\put(240,65){$-$} \put(233,52){$+$} 

\put(120,140){{\rm (i) }} \put(280,140){{\rm (ii) }}

\put(143,96){$(\infty)$} \put(303,96){$(\infty)$} \put(15,23){$(\infty)$} \put(175,23){$(\infty)$} 

\put(130,5){{\rm \textbf{figure (rPD)}}}

\thicklines

\put(90,90){\line(1,0){60}} \put(70,90){\line(40,41){40}} \put(57,108){\line(-37,-18){37}} \put(57,72){\line(53,-23){53}} 
\put(250,90){\line(1,0){60}} \put(230,90){\line(40,41){40}} \put(217,108){\line(-37,-18){37}} \put(217,72){\line(53,-23){53}} 
 
\end{picture}
$\,$\\
We prepare two copies of $\overline{\mathbb{C}}$ and slit them along the thick lines in {\rm figure (rPD)}. 
Identifying each of the upper (resp. lower) edges of the thick lines in {\rm (i)} with 
each of the lower (resp. upper) edges of the thick lines in (ii), we obtain the hyperelliptic Riemann surface $M$ of genus $3$ 
(see the following figure). 
Note that each of thin lines joining two branch points in {\rm figure (rPD)} is corresponding to each of thick lines 
joining two branch points in the following figure. \\

\begin{picture}(340,160)

\put(135,148){$(\infty,\infty)$} \put(138,122){$(a,\,0)$} \put(138,105){$(0,\,0)$} 
\put(130,82){$(e^{\frac{\pi}{3} i}/a,\,0)$} \put(130,65){$(e^{\frac{2}{3}\pi i} a ,\,0)$} \put(130,40){$(-1/a,\,0)$}
\put(130,24){$(e^{-\frac{2}{3} \pi i} a ,\,0)$} \put(130,0){$(e^{-\frac{\pi}{3} i}/a,\,0)$} 

\put(122,134){$-$} \put(122,93){$-$} \put(122,53){$-$} \put(122,13){$-$} 
\put(212,133){$-$} \put(212,93){$-$} \put(212,53){$-$} \put(212,13){$-$} 

\put(90,133){$+$} \put(90,93){$+$} \put(90,53){$+$} \put(90,13){$+$} 
\put(180,133){$+$} \put(180,93){$+$} \put(180,53){$+$} \put(180,13){$+$} 

\put(110,150){\circle*{2}} \put(110,120){\circle*{2}} \put(110,110){\circle*{2}} \put(110,80){\circle*{2}} 
\put(110,70){\circle*{2}} \put(110,40){\circle*{2}} \put(110,30){\circle*{2}} \put(110,0){\circle*{2}}
\put(200,150){\circle*{2}} \put(200,120){\circle*{2}} \put(200,110){\circle*{2}} \put(200,80){\circle*{2}} 
\put(200,70){\circle*{2}} \put(200,40){\circle*{2}} \put(200,30){\circle*{2}} \put(200,0){\circle*{2}}

\qbezier(110,150)(100,150)(100,135) \qbezier(100,135)(100,120)(110,120)
\qbezier[10](110,150)(120,150)(120,135) \qbezier[10](120,135)(120,120)(110,120)

\qbezier(110,110)(100,110)(100,95) \qbezier(100,95)(100,80)(110,80)
\qbezier[10](110,110)(120,110)(120,95) \qbezier[10](120,95)(120,80)(110,80)

\qbezier(110,70)(100,70)(100,55) \qbezier(100,55)(100,40)(110,40)
\qbezier[10](110,70)(120,70)(120,55) \qbezier[10](120,55)(120,40)(110,40)

\qbezier(110,30)(100,30)(100,15) \qbezier(100,15)(100,0)(110,0)
\qbezier[10](110,30)(120,30)(120,15) \qbezier[10](120,15)(120,0)(110,0)

\qbezier(200,150)(190,150)(190,135) \qbezier(190,135)(190,120)(200,120)
\qbezier[10](200,150)(210,150)(210,135) \qbezier[10](210,135)(210,120)(200,120)

\qbezier(200,110)(190,110)(190,95) \qbezier(190,95)(190,80)(200,80)
\qbezier[10](200,110)(210,110)(210,95) \qbezier[10](210,95)(210,80)(200,80)

\qbezier(200,70)(190,70)(190,55) \qbezier(190,55)(190,40)(200,40)
\qbezier[10](200,70)(210,70)(210,55) \qbezier[10](210,55)(210,40)(200,40)

\qbezier(200,30)(190,30)(190,15) \qbezier(190,15)(190,0)(200,0)
\qbezier[10](200,30)(210,30)(210,15) \qbezier[10](210,15)(210,0)(200,0)

\put(50,150){(i)} \put(245,150){(ii)}

\thicklines

\qbezier(110,150)(30,150)(30,75) \qbezier(30,75)(30,0)(110,0)

\qbezier(200,150)(280,150)(280,75) \qbezier(280,75)(280,0)(200,0)

\qbezier(110,120)(80,115)(110,110) \qbezier(110,80)(80,75)(110,70)
\qbezier(110,40)(80,35)(110,30) 

\qbezier(200,120)(230,115)(200,110) \qbezier(200,80)(230,75)(200,70)
\qbezier(200,40)(230,35)(200,30) 

\end{picture}

$\,$\\
To describe $1$-cycles on $M$, we consider the following key paths:
\begin{align*}
C_1 &= \{(z,\,w)=(e^{\frac{\pi}{3}i} t/a,\, i e^{\frac{\pi}{6}i} \sqrt{t(1-t^3)(a^3+t^3 / a^3)} /a^2) \>|\>t: 0\to 1,\> \sqrt{*}>0 \}, \\
C_2 &= \{(z,\,w)=(e^{\frac{\pi}{3}i} t/a,\, - e^{\frac{\pi}{6}i} \sqrt{t(t^3-1)(a^3+t^3 / a^3)}/a^2 )\>|\>t: 1\to \infty ,\> \sqrt{*}>0 \}, \\
C_3 &= \{(z,\,w)=(a t,\, ia^2 \sqrt{t(1-t^3)(a^3t^3+1/a^3)})\>|\>t: 0 \to 1,\> \sqrt{*}>0 \}, \\
C_4 &= \{(z,\,w)=(a t ,\, a^2 \sqrt{t(t^3-1)(a^3 t^3+1/a^3)})\>|\>t:1\to \infty, \> \sqrt{*}>0\}. 
\end{align*}
We first choose $C_1$ in the following figure. After that, we shall see that $C_1\cup C_2$ is homotopic to $C_3 \cup C_4 $ 
by using path-integrals of the holomorphic differentials $\dfrac{2z}{w}dz$ and $\dfrac{1-z^2}{w} dz$. 

\begin{picture}(340,100)

\put(10,10){\line(1,0){140}} 
\put(30,0){\line(0,1){80}} 

\put(30,10){\circle*{2}} \put(50,10){\circle*{2}}  

\put(95,2){$-$} \put(10,60){{\rm (i) }} \put(40,0){$C_3$} \put(110,0){$C_4$} 
\put(64,30){$C_1$} \put(92,58){$C_2$} \put(40,30){$-$} 

\put(245,35){$C_1$} \put(205,15){$C_2$} \put(242,52){$C_3$} \put(245,70){$C_4$} 

\put(285,90){$(\infty,\,\infty)$} \put(290,60){$(a,\,0)$}
\put(290,45){$(0,\,0)$} \put(283,22){$(e^{\frac{\pi}{3} i}/a,\,0)$}

\put(270,50){\circle*{2}} \put(270,20){\circle*{2}} 

\put(30,10){\line(40,41){40}} \put(70,51){\circle*{2}} 
\put(70,51){\line(40,41){20}}

\thicklines

\put(50,10){\line(1,0){100}} 

\put(40,14){\line(1,0){20}} \put(53,14){\line(-2,1){5}} \put(53,14){\line(-2,-1){5}}
\put(60,14){\line(1,0){90}} \put(113,14){\line(-2,1){5}} \put(113,14){\line(-2,-1){5}}

\thinlines

\qbezier(270,90)(260,90)(260,75) \qbezier(260,75)(260,60)(270,60)
\qbezier[10](270,90)(280,90)(280,75) \qbezier[10](280,75)(280,60)(270,60)

\qbezier[10](270,50)(280,50)(280,35) \qbezier[10](280,35)(280,20)(270,20)

\put(210,90){(i)} 

\qbezier(270,90)(190,90)(190,15) \qbezier(270,60)(240,55)(270,50) \qbezier(270,20)(240,15)(270,10)

\thicklines

\put(30,10){\line(40,41){40}} \put(59.5,34){\line(-2,-1){7}} \put(59.5,34){\line(-1,-3){2}} 
\put(92,67.5){\line(-2,-1){7}} \put(92,67.5){\line(-1,-3){2}}
\put(40,14){\line(40,41){60}}

\qbezier(270,90)(260,90)(260,75) \qbezier(260,75)(260,60)(270,60)
\put(260,78){\line(-1,-2){3}} \put(260,78){\line(1,-2){3}} 

\qbezier(270,50)(260,50)(260,35) \qbezier(260,35)(260,20)(270,20)
\put(260,32){\line(-1,2){3}} \put(260,32){\line(1,2){3}} 

\qbezier(270,60)(240,55)(270,50) \put(258,53){\line(1,-2){2}} \put(258,53){\line(5,1){5}} 

\qbezier(270,90)(210,90)(210,30) \put(213,55){\line(1,-3){2.5}} \put(213,55){\line(-1,-1){6}}
\qbezier(210,30)(210,20)(270,20) 

\end{picture}

$\,$\\ 
Taking suitable $l,\,m,\,n\in \{0,\,1\}$, we find that $C_1\cup j^l (C_2)$ is homotopic to $j^m (C_3) \cup j^n (C_4)$. 
Now we have 
\begin{align}
\int_{C_1} \dfrac{2z}{w} dz &= 2\int_0^1 \dfrac{t}{\sqrt{t(1-t^3)(a^3+\frac{t^3}{a^3})}} dt 
= 2\int_1^{\infty} \dfrac{t}{\sqrt{t(t^3-1)(a^3 t^3+\frac{1}{a^3})}} dt, \label{eq-rPD31} \\
\nonumber \int_{C_2} \dfrac{2z}{w} dz &= -2 i \int_0^1 \dfrac{t}{\sqrt{t(1-t^3)(a^3 t^3+\frac{1}{a^3})}} dt, \\
\int_{C_3} \dfrac{2z}{w} dz &=-2 i \int_0^1 \dfrac{t}{\sqrt{t(1-t^3)(a^3 t^3+\frac{1}{a^3})}} dt, \label{eq-rPD32} \\
\nonumber \int_{C_4} \dfrac{2z}{w} dz &= 2\int_1^{\infty} \dfrac{t}{\sqrt{t(t^3-1)(a^3 t^3+\frac{1}{a^3})}} dt. 
\end{align}
It follows from the equation $\int_{C_1\cup j^l (C_2)} \frac{2z}{w} dz=\int_{j^m (C_3) \cup j^n (C_4)} \frac{2z}{w} dz$ that 
$n=0$ and $l=m$. Similarly, we find 
\begin{align}
&\int_{C_1} \dfrac{1-z^2}{w} dz = -\dfrac{\sqrt{3} }{2 a} i \int_1^{\infty} \dfrac{1+a^2 t^2}{\sqrt{t(t^3-1)(a^3 t^3+\frac{1}{a^3})}} dt
-\dfrac{1}{2 a} \int_1^{\infty} \dfrac{1-a^2 t^2}{\sqrt{t(t^3-1)(a^3 t^3+\frac{1}{a^3})}} dt, \label{eq-rPD11} \\
\nonumber &\int_{C_2} \dfrac{1-z^2}{w} dz = 
-\dfrac{\sqrt{3} }{2 a} \int_0^1 \dfrac{1+a^2 t^2}{\sqrt{t(1-t^3)(a^3 t^3+\frac{1}{a^3})}} dt
+\dfrac{i}{2 a} \int_0^1 \dfrac{1-a^2 t^2}{\sqrt{t(1-t^3)(a^3 t^3+\frac{1}{a^3})}} dt, \\
&\int_{C_3} \dfrac{1-z^2}{w} dz =-\dfrac{i}{a} \int_0^1 \dfrac{1-a^2 t^2}{\sqrt{t(1-t^3)(a^3 t^3+\frac{1}{a^3})}} dt, \label{eq-rPD12} \\
\nonumber &\int_{C_4} \dfrac{1-z^2}{w} dz = \dfrac{1}{a} \int_1^{\infty} \dfrac{1-a^2 t^2}{\sqrt{t(t^3-1)(a^3 t^3+\frac{1}{a^3})}} dt. 
\end{align}
The equation $\int_{C_1\cup j^m (C_2)} \frac{1-z^2}{w} dz=\int_{j^m (C_3) \cup C_4} \frac{1-z^2}{w} dz$ implies $m=0$. 
Thus, $C_1\cup C_2$ is homotopic to $C_3 \cup C_4$. 
In the process, we have also showed that the two equations 
\begin{align}
\sqrt{3} \int_0^1 \dfrac{1-a^2 t^2}{\sqrt{t(1-t^3)(a^3 t^3 + \frac{1}{a^3} )}}dt &= 
\int_1^{\infty} \dfrac{1+ a ^2 t^2}{\sqrt{t(t^3-1)(a^3 t^3+\frac{1}{a^3})}}dt, \label{eq-rPD100}\\
\sqrt{3} \int_1^{\infty} \dfrac{1-a^2 t^2}{\sqrt{t(t^3-1)(a^3 t^3 + \frac{1}{a^3} )}}dt &= 
-\int_0^1 \dfrac{1+ a ^2 t^2}{\sqrt{t(1-t^3)(a^3 t^3+\frac{1}{a^3})}}dt. \label{eq-rPD101}
\end{align}
Next we shall show that $j(C_1) \cup C_2$ is homotopic to $j(\varphi(C_3)) \cup \varphi (C_4)$. 
Choosing suitable $m,\,n \in \{ 0,\,1 \}$, we obtain that $j(C_1) \cup C_2$ is homotopic to 
$j^m (\varphi(C_3)) \cup j^n (\varphi (C_4))$. 
The equation $\int_{j(C_1) \cup C_2} \frac{2 z}{w} dz=\int_{j^m (\varphi(C_3)) \cup j^n (\varphi (C_4))} \frac{2 z}{w} dz$ 
yields that $m=1$ and $n=0$. Hence we have the following figure. 

\begin{picture}(340,110)

\put(10,20){\line(1,0){140}} 
\put(70,0){\line(0,1){90}} 

\put(70,20){\circle*{2}} \put(90,20){\circle*{2}} \put(20,20){\circle*{2}} 
\put(57,38){\circle*{2}} \put(110,61){\circle*{2}} 
 
\put(60,10){$O$} \put(88,24){$a$} \put(0,26){$-1/a$} \put(25,42){$e^{\frac{2}{3} \pi i} a$} \put(115,55){$e^{\frac{\pi}{3} i} /a$}

\put(110,61){\line(-53,-23){53}}

\put(90,35){$+$} \put(115,24){$+$} \put(115,12){$-$} \put(30,30){$+$} \put(35,22){$-$} 

\put(10,70){{\rm (i) }} 

\thicklines

\put(90,20){\line(1,0){60}} \put(70,20){\line(40,41){40}} \put(57,38){\line(-37,-18){37}} 

\put(70,26){\line(40,41){55}} \put(89.5,46){\line(-2,-1){6}} \put(89.5,46){\line(-1,-3){2}} 
\put(110,67){\line(-2,-1){6}} \put(110,67){\line(-1,-3){2}} 
\put(70,26){\line(-13,18){35}} \put(64,34){\line(0,-1){5}} \put(64,34){\line(3,-1){5}} 
\put(44,62){\line(0,-1){5}} \put(44,62){\line(3,-1){5}}

\thinlines 

\put(285,45){$j(C_1)$} \put(225,40){$C_2$} \put(220,7){$j(\varphi (C_3))$} \put(192,30){$\varphi( C_4)$} 

\put(270,60){\circle*{2}} \put(270,30){\circle*{2}} \put(270,20){\circle*{2}} 

\qbezier(270,100)(260,100)(260,85) \qbezier(260,85)(260,70)(270,70)
\qbezier[10](270,100)(280,100)(280,85) \qbezier[10](280,85)(280,70)(270,70)

\qbezier[10](270,60)(280,60)(280,45) \qbezier[10](280,45)(280,30)(270,30)

\put(210,100){(i)} 

\qbezier(270,100)(190,100)(190,25) \qbezier(270,70)(240,65)(270,60) \qbezier(270,30)(240,25)(270,20)
\qbezier(270,20)(260,20)(260,5) \qbezier[12](270,20)(280,20)(280,5) 

\qbezier(270,60)(260,60)(260,45) \qbezier(260,45)(260,30)(270,30)

\thicklines
 
\qbezier(270,60)(280,60)(280,45) \qbezier(280,45)(280,30)(270,30)
\put(280,42){\line(-1,2){3}} \put(280,42){\line(1,2){3}} 

\qbezier(270,100)(220,100)(220,40) \qbezier(220,40)(220,30)(270,30) 
\put(223,65){\line(1,-3){2.5}} \put(223,65){\line(-1,-1){6}}
 
\qbezier[20](270,60)(230,40)(240,20) \qbezier[10](240,20)(250,15)(270,20) 
\qbezier[3](240,20)(236,21)(232,22) \qbezier[4](240,20)(242,24)(244,28)
 
\qbezier(210,60)(210,20)(270,20) \qbezier(210,60)(220,100)(270,100) 
\put(210,60){\line(-1,-3){3}} \put(210,60){\line(1,-2){4}} 
 
\end{picture}

$\,$

\begin{picture}(340,170)

\put(10,130){\line(1,0){140}} 
\put(20,10){\line(0,1){140}} 

\put(50,130){\circle*{2}} \put(20,130){\circle*{2}} \put(80,59){\circle*{2}} 
 
\put(10,120){$O$} \put(54,134){$a$} \put(40,54){$e^{-\frac{\pi}{3} i} /a$}

\put(80,59){\line(1,-1){35}} 

\put(95,134){$+$}  
\put(30,82){$-$} \put(25,72){$+$} 

\put(120,145){{\rm (i) }} 

\thicklines

\put(20,135){\line(1,0){30}} \put(36,135){\line(-2,1){7}}  \put(36,135){\line(-2,-1){7}} \put(36,140){$C_3$}
\put(50,125){\line(1,0){100}} \put(106,125){\line(-2,1){7}}  \put(106,125){\line(-2,-1){7}} \put(110,115){$j(C_4)$} 

\put(30,120){\line(1,-1){56}} \put(60,90){\line(-2,1){10}} \put(60,90){\line(-1,4){3}} \put(69,88){$j( \varphi^2 (C_1) )$}
\put(73,52){\line(1,-1){35}} \put(93,32){\line(-2,1){10}} \put(93,32){\line(-1,4){3}} \put(57,25){$ \varphi^2 (C_2) $}

\put(50,130){\line(1,0){100}} \put(27,82){\line(53,-23){53}} \put(27,82){\line(-53,23){30}}

\thinlines
 
\put(285,132){$j (C_4)$} \put(230,112){$C_3$} 
\put(280,24){$(e^{-\frac{2}{3} \pi i} a ,\,0)$} \put(280,0){$(e^{-\frac{\pi}{3} i}/a,\,0)$} 

\put(274,134){$-$} \put(272,93){$-$} \put(272,53){$-$} \put(272,13){$-$} 

\put(240,133){$+$} \put(240,93){$+$} \put(240,53){$+$} \put(240,13){$+$} 

\put(260,150){\circle*{2}} \put(260,120){\circle*{2}} \put(260,110){\circle*{2}} \put(260,80){\circle*{2}} 
\put(260,70){\circle*{2}} \put(260,40){\circle*{2}} \put(260,30){\circle*{2}} \put(260,0){\circle*{2}}

\qbezier(260,150)(250,150)(250,135) \qbezier(250,135)(250,120)(260,120)
\qbezier[10](260,150)(270,150)(270,135) \qbezier[10](270,135)(270,120)(260,120)

\qbezier(260,110)(250,110)(250,95) \qbezier(250,95)(250,80)(260,80)
\qbezier[10](260,110)(270,110)(270,95) \qbezier[10](270,95)(270,80)(260,80)

\qbezier(260,70)(250,70)(250,55) \qbezier(250,55)(250,40)(260,40)
\qbezier[10](260,70)(270,70)(270,55) \qbezier[10](270,55)(270,40)(260,40)

\qbezier(260,30)(250,30)(250,15) \qbezier(250,15)(250,0)(260,0)
\qbezier[10](260,30)(270,30)(270,15) \qbezier[10](270,15)(270,0)(260,0)

\put(200,150){(i)}

\qbezier(260,120)(230,115)(260,110) \qbezier(260,80)(230,75)(260,70)
\qbezier(260,40)(230,35)(260,30) 

\thicklines
 
\qbezier(260,120)(230,115)(260,110) \put(248,113){\line(1,-2){2}} \put(248,113){\line(5,1){5}} 

\qbezier(260,150)(270,150)(270,135) \qbezier(270,135)(270,120)(260,120)
\put(270,138){\line(-1,-2){3}} \put(270,138){\line(1,-2){3}} 
 
\qbezier[20](260,110)(230,110)(230,55) \qbezier[20](230,55)(230,0)(260,0) 
\qbezier[3](230,55)(228,59)(226,63) \qbezier[3](230,55)(232,59)(234,63) 
\put(185,50){$j(\varphi^2(C_1))$}

\qbezier(260,150)(180,150)(180,75) \qbezier(180,75)(180,0)(260,0)
\put(180,75){\line(-1,-2){4}} \put(180,75){\line(1,-2){4}} \put(147,78){$\varphi^2(C_2)$}
 
\end{picture}

$\,$\\
Moreover, it follows from $j(C_1) \cup C_2 \sim j(\varphi(C_3)) \cup \varphi (C_4)$ that 
$ j ( \varphi^2 (C_1) ) \cup \varphi^2 (C_2) \sim C_3 \cup j (C_4)$ and 
$ \varphi (C_1) \cup j(\varphi (C_2)) \sim \varphi^2 (C_3) \cup j (\varphi^2(C_4))$. 
From the first relation, we find the above figure. 

We shall write an accurate figure from the second relation. 
Choosing a suitable $m \in \{0,\,1\}$, we obtain the following figure. 

\begin{picture}(340,170)

\put(10,100){\line(1,0){120}} 
\put(120,10){\line(0,1){110}} 

\put(50,100){\circle*{2}} \put(120,100){\circle*{2}} \put(100,65){\circle*{2}} 

\put(50,100){\line(50,-35){50}} 
 
\put(124,90){$O$} \put(24,90){$-1/a$} \put(60,60){$e^{-\frac{\pi}{3} i} /a$}

\put(110,60){$-$} \put(105,50){$+$} 

\put(20,145){{\rm (i) }} 

\thicklines

\put(100,65){\line(3,-2){30}}

\put(10,105){\line(1,0){40}} \put(32,105){\line(2,1){7}}  \put(32,105){\line(2,-1){7}} \put(0,115){$j^m(j(\varphi (C_2)))$}
\put(50,105){\line(1,0){70}} \put(86,105){\line(2,1){7}}  \put(86,105){\line(2,-1){7}} \put(70,115){$j^m (\varphi (C_1))$}

\put(120,95){\line(-2,-3){50}} \put(110,80){\line(0,1){6}} \put(110,80){\line(2,1){5}}
\put(80,35){\line(0,1){6}} \put(80,35){\line(2,1){5}}

\put(10,30){$ j^m ( j ( \varphi^2 (C_4) ) ) $}

\thinlines
 
\put(280,24){$(e^{-\frac{2}{3} \pi i} a ,\,0)$} \put(280,40){$(-1/a,\,0)$} 

\put(274,134){$-$} \put(272,93){$-$} \put(272,53){$-$} \put(272,13){$-$} 

\put(240,133){$+$} \put(240,93){$+$} \put(240,53){$+$} \put(240,13){$+$} 

\put(260,150){\circle*{2}} \put(260,120){\circle*{2}} \put(260,110){\circle*{2}} \put(260,80){\circle*{2}} 
\put(260,70){\circle*{2}} \put(260,40){\circle*{2}} \put(260,30){\circle*{2}} \put(260,0){\circle*{2}}

\qbezier(260,150)(250,150)(250,135) \qbezier(250,135)(250,120)(260,120)
\qbezier[10](260,150)(270,150)(270,135) \qbezier[10](270,135)(270,120)(260,120)

\qbezier(260,110)(250,110)(250,95) \qbezier(250,95)(250,80)(260,80)
\qbezier[10](260,110)(270,110)(270,95) \qbezier[10](270,95)(270,80)(260,80)

\qbezier(260,70)(250,70)(250,55) \qbezier(250,55)(250,40)(260,40)
\qbezier[10](260,70)(270,70)(270,55) \qbezier[10](270,55)(270,40)(260,40)

\qbezier(260,30)(250,30)(250,15) \qbezier(250,15)(250,0)(260,0)
\qbezier[10](260,30)(270,30)(270,15) \qbezier[10](270,15)(270,0)(260,0)

\put(200,150){(i)}

\qbezier(260,120)(230,115)(260,110) \qbezier(260,80)(230,75)(260,70)
\qbezier(260,40)(230,35)(260,30) 

\qbezier(260,150)(180,150)(180,75) \qbezier(180,75)(180,0)(260,0)

\put(268,113){$j^m (\varphi(C_1))$} \qbezier(285,108)(280,90)(232,90) 

\put(268,73){$j^m (j(\varphi(C_2)))$} \qbezier(285,68)(280,60)(219,60) 

\put(178,13){$j^m (\varphi^2 (C_3))$} \qbezier(207,47)(195,40)(190,23) 

\put(130,110){$j^m (j(\varphi^2 (C_4)))$} \qbezier(160,105)(165,80)(190,70) 
\qbezier(180,23)(150,70)(104,71)

\thicklines

\qbezier[20](260,110)(230,110)(230,75) \qbezier[20](230,75)(230,40)(260,40) 
\qbezier[3](230,75)(228,79)(226,83) \qbezier[3](230,75)(232,79)(234,83) 

\qbezier[20](260,110)(200,110)(200,70) \qbezier[20](200,70)(200,30)(260,30) 
\qbezier[4](200,70)(198,79)(196,83) \qbezier[4](200,70)(203,79)(206,83) 

\qbezier(260,150)(215,150)(215,80) \qbezier(215,80)(215,40)(260,40)
\put(215,80){\line(-1,-2){4}} \put(215,80){\line(1,-2){4}} 

\qbezier(260,150)(190,150)(190,60) \qbezier(190,60)(190,30)(260,30)
\put(190,65){\line(-1,-2){4}} \put(190,65){\line(1,-2){4}} 
 
\end{picture}

$\,$\\
Thus we have $j^m (j(\varphi^2(C_4))) \cup \{ -C_4 \} \cup \{-j(C_3)\} \cup j(j^m (\varphi^2 (C_3))) \sim 
\varphi^2 (C_2) \cup \{-j (\varphi^2 (C_2))\}$. The equation 
$\int_{ j^m (j(\varphi^2(C_4))) \cup \{ -C_4 \} \cup \{-j(C_3)\} \cup j(j^m (\varphi^2 (C_3))) } \frac{2 z}{w} dz
= \int_{\varphi^2 (C_2) \cup \{-j (\varphi^2 (C_2))  } \frac{2 z}{w} dz $ yields $m=1$. 

Therefore, we find a canonical homology basis as follows. 
\begin{align*}
A_3 &= \varphi (C_1) - j( \varphi (C_1) )-\varphi^2 (C_3) +j( \varphi^2 (C_3) ), \\
B_3 &= -\varphi^2 (C_1) + j( \varphi^2 (C_1) )+\varphi^2 (C_3) -j( \varphi^2 (C_3) ), \\
A_2 &= C_1 - j( C_1 )-\varphi (C_3) +j( \varphi (C_3) ), \\
B_2 &= -\varphi (C_1) + j( \varphi (C_1) )+\varphi (C_3) -j( \varphi (C_3) ) + B_3, \\
A_1 &= -C_3+j(C_3), \quad B_1 = -C_1+j(C_1)+B_2.
\end{align*}

\begin{picture}(340,100)

\qbezier(20,50)(20,90)(160,90) \qbezier(160,90)(300,90)(300,50) 
\qbezier(20,50)(20,10)(160,10) \qbezier(160,10)(300,10)(300,50) 

\qbezier(55,55)(80,35)(105,55) \qbezier(135,55)(160,35)(185,55) \qbezier(215,55)(240,35)(265,55) 
\qbezier(60,52)(80,62)(100,52) \qbezier(140,52)(160,62)(180,52) \qbezier(220,52)(240,62)(260,52)

\put(30,85){{\rm (ii)}} \put(30,15){{\rm (i)}} 

\qbezier(45,50)(45,30)(80,30) \qbezier(115,50)(115,30)(80,30) 
\qbezier(45,50)(45,70)(80,70) \qbezier(115,50)(115,70)(80,70) 

\qbezier(125,50)(125,30)(160,30) \qbezier(195,50)(195,30)(160,30) 
\qbezier(125,50)(125,70)(160,70) \qbezier(195,50)(195,70)(160,70) 

\qbezier(205,50)(205,30)(240,30) \qbezier(275,50)(275,30)(240,30) 
\qbezier(205,50)(205,70)(240,70) \qbezier(275,50)(275,70)(240,70) 

\qbezier(71,46)(55,38)(55,20) \qbezier[15](55,20)(75,25)(71,46) 

\qbezier(151,46)(135,35)(140,10) \qbezier[15](140,10)(160,25)(151,46) 

\qbezier(249,46)(265,38)(265,20) \qbezier[15](249,46)(243,25)(265,20)

\put(80,30){\line(-2,1){5}} \put(80,30){\line(-2,-1){5}} \put(100,25){$A_1$}
\put(160,30){\line(-2,1){5}} \put(160,30){\line(-2,-1){5}} \put(180,25){$A_2$}
\put(240,30){\line(-2,1){5}} \put(240,30){\line(-2,-1){5}} \put(275,35){$A_3$}

\put(62.5,40){\line(-2,-1){5}}  \put(62.5,40){\line(0,-1){5}}  \put(50,10){$B_1$}
\put(144.5,40){\line(-2,-1){5}} \put(144.5,40){\line(0,-1){5}} \put(135,0){$B_2$}
\put(257.5,40){\line(2,-1){5}} \put(257.5,40){\line(0,-1){5}} \put(260,10){$B_3$}

\end{picture}

\subsubsection{Period matrix}

Key $1$-cycles of the canonical homology basis as in \S~\ref{rPD-canonical} are given by $C_1 \cup \{-j(C_1)\}$ and 
$C_3 \cup \{-j(C_3)\}$. First, we have from \eqref{eq-rPD100} and \eqref{eq-rPD101}: 
\begin{align}
\int_{C_1} \dfrac{i(1+z^2)}{w} dz &= \dfrac{1}{2 a} \int_0^1 \dfrac{1+a^2 t^2}{\sqrt{t(1-t^3)(a^3 t^3+\frac{1}{a^3})}} dt 
+ \dfrac{i}{2 a} \int_1^{\infty} \dfrac{1+a^2 t^2}{\sqrt{t(t^3-1)(a^3 t^3+\frac{1}{a^3})}} dt, \label{eq-rPD21}
\end{align}
\begin{align}
\int_{C_3} \dfrac{i(1+z^2)}{w} dz &= \dfrac{1}{a} \int_0^1 \dfrac{1+a^2 t^2}{\sqrt{t(1-t^3)(a^3 t^3+\frac{1}{a^3})}} dt.  \label{eq-rPD22}
\end{align}
Thus, by \eqref{eq-rPD31}--\eqref{eq-rPD12}, \eqref{eq-rPD21}, and \eqref{eq-rPD22}, we find 
\begin{align}
\nonumber \int_{C_1\cup \{-j(C_1)\}} \dfrac{1-z^2}{w} dz &= 
-\dfrac{\sqrt{3}}{a} i \int_1^{\infty} \dfrac{1+a^2 t^2}{\sqrt{t(t^3-1)(a^3 t^3+\frac{1}{a^3})}} dt \\
&\qquad +\dfrac{1}{\sqrt{3} a} \int_0^1 \dfrac{1+a^2 t^2}{\sqrt{t(1-t^3)(a^3 t^3+\frac{1}{a^3})}} dt, \label{rPD-periods-final1}\\
\int_{C_3\cup \{-j(C_3)\}} \dfrac{1-z^2}{w} dz &= 
-\dfrac{2}{\sqrt{3} a} i \int_1^{\infty} \dfrac{1+a^2 t^2}{\sqrt{t(t^3-1)(a^3 t^3+\frac{1}{a^3})}} dt, \label{rPD-periods-final2}\\
\nonumber \int_{C_1\cup \{-j(C_1)\}} \dfrac{i(1+z^2)}{w} dz &= 
\dfrac{1}{a} \int_0^1 \dfrac{1+a^2 t^2}{\sqrt{t(1-t^3)(a^3 t^3+\frac{1}{a^3})}} dt \\
&\qquad + \dfrac{i}{a} \int_1^{\infty} \dfrac{1+a^2 t^2}{\sqrt{t(t^3-1)(a^3 t^3+\frac{1}{a^3})}} dt, \label{rPD-periods-final3}\\
\int_{C_3\cup \{-j(C_3)\}} \dfrac{i(1+z^2)}{w} dz &= 
\dfrac{2}{a} \int_0^1 \dfrac{1+a^2 t^2}{\sqrt{t(1-t^3)(a^3 t^3+\frac{1}{a^3})}} dt, 
\label{rPD-periods-final4} \\
\int_{C_1\cup \{-j(C_1)\}} \dfrac{2 z}{w} dz &= 4 \int_1^{\infty} \dfrac{t}{\sqrt{t(t^3-1)(a^3 t^3+\frac{1}{a^3})}} dt, 
\label{rPD-periods-final5} \\
\int_{C_3\cup \{-j(C_3)\}} \dfrac{2 z}{w} dz &= -4 i \int_0^1 \dfrac{t}{\sqrt{t(1-t^3)(a^3 t^3+\frac{1}{a^3})}} dt. \label{rPD-periods-final6}
\end{align}
Next, we calculate periods of the meromorphic differentials with poles. From now on, let $\gamma$ be an arbitrary $1$-cycle 
on $M$. 

Substituting $\alpha=4$, $\beta= -1$ to \eqref{eq:TTexact1} yields 
\begin{align*}
\int_{\gamma} \dfrac{z^4}{w^3} dz &= \dfrac{1}{6} \int_{\gamma} \dfrac{z^3}{w} dz
-\dfrac{1}{2} \left( a^3-\dfrac{1}{a^3} \right) \int_{\gamma} \dfrac{z^7}{w^3} dz \\
&\underbrace{=}_{\eqref{eq:TTintegral5}} 
\dfrac{1}{6} \int_{\gamma} \dfrac{z^3}{w} dz-\dfrac{1}{2} \left( a^3-\dfrac{1}{a^3} \right) 
\left\{  \dfrac{1}{6} \int_{\gamma} \dfrac{dz}{w}
+\dfrac{1}{2}\left(a^3-\dfrac{1}{a^3}\right)\int_{\gamma} \dfrac{z^4}{w^3}dz \right\}. 
\end{align*}
Thus we find 
\begin{align}\label{rPD-2nd-diff1}
\int_{\gamma} \dfrac{z^4}{w^3}dz = \dfrac{a^6}{3(a^6+1)^2} 
\left\{  2 \int_{\gamma} \dfrac{z^3}{w} dz-\left( a^3-\dfrac{1}{a^3} \right)  \int_{\gamma} \dfrac{dz}{w}  \right\}. 
\end{align}
Substituting $\alpha=5$, $\beta= -1$ to \eqref{eq:TTexact1} implies 
\begin{align*}
\int_{\gamma} \dfrac{z^5}{w^3} dz &= \dfrac{1}{2} \int_{\gamma} \dfrac{z^4}{w} dz
-\dfrac{1}{2} \left( a^3-\dfrac{1}{a^3} \right) \int_{\gamma} \dfrac{z^8}{w^3} dz 
\end{align*}
\begin{align*}
&\underbrace{=}_{\eqref{eq:TTintegral6}} 
\dfrac{1}{2} \int_{\gamma} \dfrac{z^4}{w} dz+\dfrac{1}{2} \left( a^3-\dfrac{1}{a^3} \right) \int_{\gamma} \dfrac{z^2}{w^3}dz \\
&\underbrace{=}_{\eqref{eq:TTintegral2}} \dfrac{1}{2} \int_{\gamma} \dfrac{z^4}{w} dz
+\dfrac{1}{2} \left( a^3-\dfrac{1}{a^3} \right) \left\{  -\dfrac{1}{2} \int_{\gamma} \dfrac{z}{w}dz
+\dfrac{1}{2}\left(-a^3+\dfrac{1}{a^3}\right)\int_{\gamma} \dfrac{z^5}{w^3}dz \right\}. 
\end{align*}
Hence we have 
\begin{align}\label{rPD-2nd-diff2}
\int_{\gamma} \dfrac{z^5}{w^3}dz = \dfrac{a^6}{(a^6+1)^2} 
\left\{  2 \int_{\gamma} \dfrac{z^4}{w} dz-\left( a^3-\dfrac{1}{a^3} \right)  \int_{\gamma} \dfrac{z}{w} dz  \right\}. 
\end{align}
Substituting $\alpha=6$, $\beta= -1$ to \eqref{eq:TTexact1} implies 
\begin{align*}
\int_{\gamma} \dfrac{z^6}{w^3} dz &= \dfrac{5}{6} \int_{\gamma} \dfrac{z^5}{w} dz
-\dfrac{1}{2} \left( a^3-\dfrac{1}{a^3} \right) \int_{\gamma} \dfrac{z^9}{w^3} dz \\
&\underbrace{=}_{\eqref{eq:TTintegral7}} \dfrac{5}{6} \int_{\gamma} \dfrac{z^5}{w} dz
-\dfrac{1}{2} \left( a^3-\dfrac{1}{a^3} \right) \left\{  \dfrac{5}{6} \int_{\gamma} \dfrac{z^2}{w}dz
+\dfrac{1}{2}\left(a^3-\dfrac{1}{a^3}\right)\int_{\gamma} \dfrac{z^6}{w^3}dz \right\}. 
\end{align*}
So we obtain 
\begin{align}\label{rPD-2nd-diff3}
\int_{\gamma} \dfrac{z^6}{w^3}dz = \dfrac{5 a^6}{3 (a^6+1)^2} 
\left\{  2 \int_{\gamma} \dfrac{z^5}{w} dz-\left( a^3-\dfrac{1}{a^3} \right)  \int_{\gamma} \dfrac{z^2}{w} dz  \right\}. 
\end{align}
From \eqref{rPD-2nd-diff1}, \eqref{rPD-2nd-diff2}, and \eqref{rPD-2nd-diff3}, we find 
\begin{align}
\int_{C_1-j(C_1)} \dfrac{z^4}{w^3} dz 
&= \dfrac{2e^{-\frac{\pi}{3}i}a^4}{3(a^6+1)^2} \int_0^1 \dfrac{-2t^3+1-a^6}{\sqrt{t(1-t^3)(a^3+\frac{t^3}{a^3})}} dt, 
\label{rPD-periods1} \\
\int_{C_3-j(C_3)} \dfrac{z^4}{w^3} dz 
&= -\dfrac{2 i a^2}{3(a^6+1)^2} \int_0^1 \dfrac{2 a^6 t^3+1-a^6}{\sqrt{t(1-t^3)(a^3 t^3+\frac{1}{a^3})}} dt, 
\label{rPD-periods2} \\
\int_{C_1-j(C_1)} \dfrac{z^5}{w^3} dz 
&= \dfrac{2a^3}{(a^6+1)^2} \int_0^1 \dfrac{-2t^4+(1-a^6) t}{\sqrt{t(1-t^3)(a^3+\frac{t^3}{a^3})}} dt, 
\label{rPD-periods3} \\
\int_{C_3-j(C_3)} \dfrac{z^5}{w^3} dz 
&= -\dfrac{2 i a^3}{(a^6+1)^2} \int_0^1 \dfrac{2 a^6 t^4+(1-a^6) t}{\sqrt{t(1-t^3)(a^3 t^3+\frac{1}{a^3})}} dt, 
\label{rPD-periods4} \\
\int_{C_1-j(C_1)} \dfrac{z^6}{w^3} dz 
&= \dfrac{10 e^{-\frac{2}{3} \pi i} a^2}{3(a^6+1)^2} \int_0^1 \dfrac{2 t^5+(a^6-1) t^2}{\sqrt{t(1-t^3)(a^3+\frac{t^3}{a^3})}} dt, 
\label{rPD-periods5} \\
\int_{C_3-j(C_3)} \dfrac{z^6}{w^3} dz 
&= -\dfrac{10 i a^4}{3(a^6+1)^2} \int_0^1 \dfrac{2 a^6 t^5+(1-a^6) t^2}{\sqrt{t(1-t^3)(a^3 t^3+\frac{1}{a^3})}} dt. \label{rPD-periods6}
\end{align}
Now, by using the relation $C_1 \cup C_2 \sim C_3\cup C_4$, we consider relations between the above path-integrals. 
Note that $M$ can be constructed via a glueing the following two curves by the relation $x=-1/z$, $y=w/z^4$: 
\[
w^2=z(z^3-a^3)\left(z^3+\dfrac{1}{a^3}\right)\underbrace{\longleftrightarrow}_{
x=-\dfrac{1}{z},\,y=\dfrac{w}{z^4}} y^2=x(x^3-a^3)\left(x^3+\dfrac{1}{a^3}\right).
\]
So $C_2$ and $C_4$ can be rewritten as 
\begin{align*}
C_2 &= \{ (x,\,y) = (-a e^{-\frac{\pi}{3} i } t,\, a^2 e^{-\frac{\pi}{6} i } \sqrt{t(1-t^3)(a^3 t^3+1/a^3)} ) \>|\> t:1\to 0,\>\sqrt{*}>0\}, \\
C_4 &= \{ (x,\,y) = (-t / a,\,  \sqrt{t(1-t^3)(a^3 +t^3/a^3)} /a^2 ) \>|\> t:1\to 0,\>\sqrt{*}>0\}. 
\end{align*}
Also, we have Lemma~\ref{rPD-lemma1} and Lemma~\ref{rPD-lemma2} in the coordinates $(x,\,y)$. Hence, 
by \eqref{rPD-2nd-diff1}, \eqref{rPD-2nd-diff2}, and \eqref{rPD-2nd-diff3}, we find 
\begin{align}
\int_{C_2-j(C_2)} \dfrac{z^4}{w^3}dz &= \int_{C_2-j(C_2)} \dfrac{x^6}{y^3} dx = 
\dfrac{10 e^{\frac{\pi}{6} i} a^4}{3(a^6+1)^2} \int_0^1 \dfrac{-2 a^6 t^5+(a^6-1) t^2}{\sqrt{t(1-t^3)(a^3 t^3+\frac{1}{a^3})}} dt, 
\label{rPD-periods7} \\
\int_{C_4-j(C_4)} \dfrac{z^4}{w^3}dz &= \int_{C_4-j(C_4)} \dfrac{x^6}{y^3} dx = 
\dfrac{10 a^2}{3(a^6+1)^2} \int_0^1 \dfrac{-2 t^5+(1-a^6) t^2}{\sqrt{t(1-t^3)(a^3 +\frac{t^3}{a^3})}} dt, 
\label{rPD-periods8} \\
\int_{C_2-j(C_2)} \dfrac{z^5}{w^3}dz &= -\int_{C_2-j(C_2)} \dfrac{x^5}{y^3} dx = 
-\dfrac{2 i a^3}{(a^6+1)^2} \int_0^1 \dfrac{2 a^6 t^4+(1-a^6) t}{\sqrt{t(1-t^3)(a^3 t^3+\frac{1}{a^3})}} dt, 
\label{rPD-periods9} \\
\int_{C_4-j(C_4)} \dfrac{z^5}{w^3}dz &= -\int_{C_4-j(C_4)} \dfrac{x^5}{y^3} dx = 
\dfrac{2a^3}{(a^6+1)^2} \int_0^1 \dfrac{-2t^4+(1-a^6) t}{\sqrt{t(1-t^3)(a^3+\frac{t^3}{a^3})}} dt, 
\label{rPD-periods10} \\
\int_{C_2-j(C_2)} \dfrac{z^6}{w^3}dz &= \int_{C_2-j(C_2)} \dfrac{x^4}{y^3} dx = 
\dfrac{2e^{-\frac{\pi}{6}i}a^2}{3(a^6+1)^2} \int_0^1 \dfrac{2 a^6 t^3+1-a^6}{\sqrt{t(1-t^3)(a^3 t^3+\frac{1}{a^3})}} dt, 
\label{rPD-periods11} \\
\int_{C_4-j(C_4)} \dfrac{z^6}{w^3}dz &= \int_{C_4-j(C_4)} \dfrac{x^4}{y^3} dx = 
\dfrac{2a^4}{3(a^6+1)^2} \int_0^1 \dfrac{-2t^3+1-a^6}{\sqrt{t(1-t^3)(a^3+\frac{t^3}{a^3})}} dt. \label{rPD-periods12}
\end{align}
Combining \eqref{rPD-periods1}--\eqref{rPD-periods12}, and the relation $C_1 \cup C_2 \sim C_3 \cup C_4$ yields 
\begin{align}
\nonumber & -\dfrac{5}{2\sqrt{3}} a^2 \int_0^1 \dfrac{2 a^6 t^5+(1-a^6)t^2}{\sqrt{t(1-t^3)(a^3 t^3+\frac{1}{a^3})}}dt
+\dfrac{1}{\sqrt{3}} \int_0^1 \dfrac{2a^6 t^3+1-a^6}{\sqrt{t(1-t^3)(a^3 t^3 +\frac{1}{a^3})}} dt \\
&=\dfrac{a^2}{2} \int_0^1 \dfrac{-2 t^3+1-a^6}{\sqrt{t(1-t^3)(a^3+\frac{t^3}{a^3})}} dt, \label{rPD-periods13} 
\end{align}
\begin{align}
\nonumber &-\dfrac{10}{\sqrt{3}} a^2 \int_0^1 \dfrac{2 a^6 t^5+(1-a^6)t^2}{\sqrt{t(1-t^3)(a^3 t^3+\frac{1}{a^3})}}dt
+\dfrac{1}{\sqrt{3}} \int_0^1 \dfrac{2a^6 t^3+1-a^6}{\sqrt{t(1-t^3)(a^3 t^3 +\frac{1}{a^3})}} dt \\
&= 5\int_0^1 \dfrac{-2 t^5+(1-a^6) t^2}{\sqrt{t(1-t^3)(a^3+\frac{t^3}{a^3})}} dt. \label{rPD-periods14}
\end{align}
By \eqref{rPD-periods1}--\eqref{rPD-periods6}, \eqref{rPD-periods13}, and \eqref{rPD-periods14}, we have 
\begin{align}
\nonumber \int_{C_1\cup\{-j(C_1)\}} \dfrac{z^4-z^6}{w^3} dz &= \dfrac{a^2}{3(a^6+1)^2} \bigg\{ 
\dfrac{1}{\sqrt{3}} \int_0^1 \dfrac{(2a^6 t^3+1-a^6)(5 a^2 t^2+1)}{\sqrt{t(1-t^3)(a^3 t^3+\frac{1}{a^3})}} dt \\
&\quad + 3 i  \int_0^1 \dfrac{(2a^6 t^3+1-a^6)(5 a^2 t^2-1)}{\sqrt{t(1-t^3)(a^3 t^3+\frac{1}{a^3})}} dt \bigg\}, 
\label{rPD-periods-final7}\\
\int_{C_3\cup\{-j(C_3)\}} \dfrac{z^4-z^6}{w^3} dz &= \dfrac{2 i a^2}{3(a^6+1)^2} 
\int_0^1 \dfrac{(2a^6 t^3+1-a^6)(5 a^2 t^2-1)}{\sqrt{t(1-t^3)(a^3 t^3+\frac{1}{a^3})}} dt, \label{rPD-periods-final8}\\
\nonumber \int_{C_1\cup\{-j(C_1)\}} \dfrac{i( z^4+z^6)}{w^3} dz &= \dfrac{a^2}{3(a^6+1)^2} \bigg\{ 
\int_0^1 \dfrac{(2a^6 t^3+1-a^6)(5 a^2 t^2+1)}{\sqrt{t(1-t^3)(a^3 t^3+\frac{1}{a^3})}} dt \\
&\quad - \sqrt{3} i  \int_0^1 \dfrac{(2a^6 t^3+1-a^6)(5 a^2 t^2-1)}{\sqrt{t(1-t^3)(a^3 t^3+\frac{1}{a^3})}} dt \bigg\}, 
\label{rPD-periods-final9} \\
\int_{C_3\cup\{-j(C_3)\}} \dfrac{i (z^4+z^6)}{w^3} dz &= \dfrac{2 a^2}{3(a^6+1)^2} 
\int_0^1 \dfrac{(2a^6 t^3+1-a^6)(5 a^2 t^2+1)}{\sqrt{t(1-t^3)(a^3 t^3+\frac{1}{a^3})}} dt, 
\label{rPD-periods-final10} \\
\int_{C_1\cup\{-j(C_1)\}} \dfrac{z^5}{w^3} dz &= \dfrac{2 a^3}{(a^6+1)^2} 
\int_0^1 \dfrac{-2 t^4+(1-a^6) t }{\sqrt{t(1-t^3)(a^3 +\frac{t^3}{a^3})}} dt, \label{rPD-periods-final11}\\
\int_{C_3\cup\{-j(C_3)\}} \dfrac{z^5}{w^3} dz &= -\dfrac{2 i a^3}{(a^6+1)^2} 
\int_0^1 \dfrac{2 a^6 t^4+(1-a^6) t }{\sqrt{t(1-t^3)(a^3 t^3+\frac{1}{a^3})}} dt. \label{rPD-periods-final12}
\end{align}
Set 
\begin{align*}
A &= \dfrac{1}{\sqrt{3} a} \int_0^1 \dfrac{1+a^2 t^2}{\sqrt{t(1-t^3)(a^3 t^3+\frac{1}{a^3})}} dt, \>
B = \dfrac{1}{\sqrt{3} a} \int_1^{\infty} \dfrac{1+a^2 t^2}{\sqrt{t(t^3-1)(a^3 t^3+\frac{1}{a^3})}} dt, \\
C &= 4 \int_1^{\infty} \dfrac{t}{\sqrt{t(t^3-1)(a^3 t^3+\frac{1}{a^3})}} dt, \> 
D = 4 \int_0^1 \dfrac{t}{\sqrt{t(1-t^3)(a^3 t^3+\frac{1}{a^3})}} dt, \\
E &= \dfrac{a^2}{3 \sqrt{3} (a^6+1)^2} \int_0^1 \dfrac{(2a^6 t^3+1-a^6)(5 a^2 t^2+1)}{\sqrt{t(1-t^3)(a^3 t^3+\frac{1}{a^3})}} dt, 
\end{align*}
\begin{align*}
F &= \dfrac{a^2}{3(a^6+1)^2} \int_0^1 \dfrac{(2a^6 t^3+1-a^6)(5 a^2 t^2-1)}{\sqrt{t(1-t^3)(a^3 t^3+\frac{1}{a^3})}} dt, \\
H &= \dfrac{2 a^3}{(a^6+1)^2} \int_0^1 \dfrac{-2 t^4+(1-a^6) t }{\sqrt{t(1-t^3)(a^3 +\frac{t^3}{a^3})}} dt, \> 
I = \dfrac{2 a^3}{(a^6+1)^2} \int_0^1 \dfrac{2 a^6 t^4+(1-a^6) t }{\sqrt{t(1-t^3)(a^3 t^3+\frac{1}{a^3})}} dt. 
\end{align*}
From \eqref{rPD-periods-final1}--\eqref{rPD-periods-final6}, \eqref{rPD-periods-final7}--\eqref{rPD-periods-final12}, 
we find 
\begin{align*}
\int_{C_1\cup\{-j(C_1)\}} G = i
\begin{pmatrix}
A- 3 i B \\
\sqrt{3} (A+ i B) \\
C \\
E+ 3 i F \\
\sqrt{3} (E- i F) \\
H
\end{pmatrix}, \quad 
\int_{C_3\cup\{-j(C_3)\}} G = i
\begin{pmatrix}
- 2 i B \\
2 \sqrt{3} A \\
- i D \\
2 i F \\
2 \sqrt{3} E \\
- i I
\end{pmatrix}. 
\end{align*}
Therefore, the period matrix of the abelian differentials of the second kind is given by 
\[i
\begin{pmatrix}
2 i B   &   -2 (A+i B)  &  -(A+i B)  & 2 A  &   3(A-i B)   &  2(A-i B) \\
-2\sqrt{3} A  &   0   &  \sqrt{3}  (A+i B)  &  -2\sqrt{3} i B  &  \sqrt{3} (A-i B)  &  0 \\
i D   &   C- i D  &  -C + i D  &  -C  &  0  &  -(C+i D) \\
- 2 i F   &   2 (-E+ i F)  &  -E+i F  &  2 E  & 3(E+i F)  &   2(E+i F) \\ 
-2\sqrt{3} E  &  0   &  \sqrt{3} (E-i F)  &  2\sqrt{3}  i  F  & \sqrt{3} (E+i F)  &  0  \\ 
i I  &   H - i I  &   -H+i I   &  -H  &  0  &  -(H+i I)
\end{pmatrix}. 
\]

\subsection{tP family, tD family}\label{tP-detail}

\subsubsection{Canonical homology basis}\label{tP-canonical} 

Let $M$ be a hyperelliptic Riemann surface of genus $3$ defined as the completion of 
$\{(z,\,w)\,|\, w^2=z^8 + a z^4+1 \} \subset \mathbb{C}^2$ for $a\in (-\infty,\,-2) \cup (2,\,\infty)$. 
It suffices to consider the case $a\in (2,\,\infty)$ because we obtain the same result for $a\in (-\infty,\,-2)$. 
The three differentials 
\[\dfrac{dz}{w},\,z\dfrac{dz}{w},\,z^2\dfrac{dz}{w}\]
form a basis for the abelian differentials of the first kind. 
Up to exact forms, the abelian differentials of the second kind are given by the following six differentials: 
\[\dfrac{dz}{w},\,z\dfrac{dz}{w},\,z^2\dfrac{dz}{w},\,\dfrac{z^4}{w^3}dz,\,\dfrac{z^5}{w^3}dz,\,\dfrac{z^6}{w^3}dz.\]

Let 
\[G=\left(\dfrac{1-z^2}{w},\,\dfrac{i(1+z^2)}{w},\,\dfrac{2z}{w},\,
\dfrac{z^4-z^6}{w^3},\,\dfrac{i(z^4+z^6)}{w^3},\,\dfrac{z^5}{w^3}\right)^t dz \]
and consider the biholomorphisms 
\begin{align*}
j(z,\,w)&=(z,\,-w),\>\>\varphi(z,\,w)=(i z,\,w)
\end{align*}
on $M$. Then it is straightforward to compute that 
\begin{align*}
j^{*}G=-G, \quad 
\varphi^{*} G=
\begin{pmatrix}
0 & 1 & 0 & 0 & 0 & 0 \\
-1 & 0 & 0 & 0 & 0 & 0 \\
0 & 0 & -1 & 0 & 0 & 0 \\
0 & 0 & 0 & 0 & 1 & 0 \\
0 & 0 & 0 & -1 & 0 & 0 \\
0 & 0 & 0 & 0 & 0 & -1
\end{pmatrix}G. 
\end{align*}
Now we determine a canonical homology basis on $M$. Recall that 
\begin{align*}
\pi_{{\rm tP}}:&\quad M \longrightarrow \overline{\mathbb{C}}:=\mathbb{C} \cup \{\infty\} \\
& (z,\,w) \> \longmapsto \> z
\end{align*}
defines a two-sheeted branched covering and $j$ is its deck transformation. 
Set $\alpha := \sqrt{\dfrac{\sqrt{a+2}+\sqrt{a-2}}{2}}>1$. $\pi_{{\rm tP}}$ has branch locus
\[
\{\alpha e^{\pm \frac{\pi}{4}i},\> \alpha e^{\pm \frac{3}{4}\pi i},\> 
\alpha^{-1} e^{\pm \frac{\pi}{4}i}, \> \alpha^{-1} e^{\pm \frac{3}{4}\pi i}\}. 
\]

\begin{picture}(340,150)

\put(10,90){\line(1,0){130}} \put(170,90){\line(1,0){130}}
\put(70,30){\line(0,1){120}} \put(230,30){\line(0,1){120}}

\put(85,105){\circle*{2}} \put(105,125){\circle*{2}} \put(85,75){\circle*{2}} \put(105,55){\circle*{2}} 
\put(55,75){\circle*{2}} \put(35,55){\circle*{2}} \put(55,105){\circle*{2}} \put(35,125){\circle*{2}} 

\put(245,105){\circle*{2}} \put(265,125){\circle*{2}} \put(245,75){\circle*{2}} \put(265,55){\circle*{2}} 
\put(215,75){\circle*{2}} \put(195,55){\circle*{2}} \put(215,105){\circle*{2}} \put(195,125){\circle*{2}} 

\put(35,125){\line(1,0){70}} \put(35,55){\line(1,0){70}} \put(55,75){\line(0,1){30}} \put(85,75){\line(0,1){30}} 
\put(195,125){\line(1,0){70}} \put(195,55){\line(1,0){70}} \put(215,75){\line(0,1){30}} \put(245,75){\line(0,1){30}} 

\put(120,140){{\rm (i) }} \put(280,140){{\rm (ii) }}

\put(130,10){{\rm \textbf{figure (tP)}}}

\put(95,107){$-$} \put(85,114){$+$} \put(95,67){$-$} \put(85,60){$+$} 
\put(38,107){$-$} \put(47,114){$+$} \put(38,67){$-$} \put(47,60){$+$} 

\put(255,107){$-$} \put(245,114){$+$} \put(255,67){$-$} \put(245,60){$+$} 
\put(198,107){$-$} \put(207,114){$+$} \put(198,67){$-$} \put(207,60){$+$} 

\put(88,95){$\alpha^{-1} e^{\frac{\pi}{4}i}$} \put(110,124){$\alpha e^{\frac{\pi}{4}i}$} 
\put(88,78){$\alpha^{-1} e^{-\frac{\pi}{4}i}$} \put(110,50){$\alpha e^{-\frac{\pi}{4}i}$} 
\put(10,78){$\alpha^{-1} e^{-\frac{3}{4}\pi i}$} \put(0,50){$\alpha e^{-\frac{3}{4}\pi i}$} 
\put(15,95){$\alpha^{-1} e^{\frac{3}{4}\pi i}$} \put(5,124){$\alpha e^{\frac{3}{4}\pi i}$} 

\thicklines

\put(85,105){\line(1,1){20}} \put(85,75){\line(1,-1){20}} \put(55,105){\line(-1,1){20}} \put(55,75){\line(-1,-1){20}} 
\put(245,105){\line(1,1){20}} \put(245,75){\line(1,-1){20}} \put(215,105){\line(-1,1){20}} \put(215,75){\line(-1,-1){20}} 

\end{picture}

So $M$ can be expressed as a $2$-sheeted branched cover of $\overline{\mathbb{C}}$ as 
the above. 
We prepare two copies of $\overline{\mathbb{C}}$ and slit them along the thick lines in {\rm figure (tP)}. 
Identifying each of the upper (resp. lower) edges of the thick lines in {\rm (i)} with 
each of the lower (resp. upper) edges of the thick lines in (ii), we obtain the hyperelliptic Riemann surface $M$ of genus $3$ 
(see the following figure). 
Note that each of thin lines joining two branch points in {\rm figure (tP)} is corresponding to each of thick lines 
joining two branch points in the following figure. \\

\begin{picture}(340,160)

\put(130,148){$(\alpha e^{\frac{\pi}{4}i},0)$} \put(128,122){$(\alpha^{-1} e^{\frac{\pi}{4}i},\,0)$} 
\put(126,105){$(\alpha^{-1} e^{-\frac{\pi}{4}i},\,0)$} 
\put(130,82){$(\alpha e^{-\frac{\pi}{4}i},\,0)$} \put(128,65){$(\alpha e^{-\frac{3}{4}\pi i} ,\,0)$} 
\put(124,40){$(\alpha^{-1} e^{-\frac{3}{4}\pi i},\,0)$}
\put(125,24){$(\alpha^{-1} e^{\frac{3}{4}\pi i} ,\,0)$} \put(130,0){$(\alpha e^{\frac{3}{4}\pi i},\,0)$} 

\put(122,134){$-$} \put(122,93){$-$} \put(122,53){$-$} \put(122,13){$-$} 
\put(212,133){$-$} \put(212,93){$-$} \put(212,53){$-$} \put(212,13){$-$} 

\put(90,133){$+$} \put(90,93){$+$} \put(90,53){$+$} \put(90,13){$+$} 
\put(180,133){$+$} \put(180,93){$+$} \put(180,53){$+$} \put(180,13){$+$} 

\put(110,150){\circle*{2}} \put(110,120){\circle*{2}} \put(110,110){\circle*{2}} \put(110,80){\circle*{2}} 
\put(110,70){\circle*{2}} \put(110,40){\circle*{2}} \put(110,30){\circle*{2}} \put(110,0){\circle*{2}}
\put(200,150){\circle*{2}} \put(200,120){\circle*{2}} \put(200,110){\circle*{2}} \put(200,80){\circle*{2}} 
\put(200,70){\circle*{2}} \put(200,40){\circle*{2}} \put(200,30){\circle*{2}} \put(200,0){\circle*{2}}

\qbezier(110,150)(100,150)(100,135) \qbezier(100,135)(100,120)(110,120)
\qbezier[10](110,150)(120,150)(120,135) \qbezier[10](120,135)(120,120)(110,120)

\qbezier(110,110)(100,110)(100,95) \qbezier(100,95)(100,80)(110,80)
\qbezier[10](110,110)(120,110)(120,95) \qbezier[10](120,95)(120,80)(110,80)

\qbezier(110,70)(100,70)(100,55) \qbezier(100,55)(100,40)(110,40)
\qbezier[10](110,70)(120,70)(120,55) \qbezier[10](120,55)(120,40)(110,40)

\qbezier(110,30)(100,30)(100,15) \qbezier(100,15)(100,0)(110,0)
\qbezier[10](110,30)(120,30)(120,15) \qbezier[10](120,15)(120,0)(110,0)

\qbezier(200,150)(190,150)(190,135) \qbezier(190,135)(190,120)(200,120)
\qbezier[10](200,150)(210,150)(210,135) \qbezier[10](210,135)(210,120)(200,120)

\qbezier(200,110)(190,110)(190,95) \qbezier(190,95)(190,80)(200,80)
\qbezier[10](200,110)(210,110)(210,95) \qbezier[10](210,95)(210,80)(200,80)

\qbezier(200,70)(190,70)(190,55) \qbezier(190,55)(190,40)(200,40)
\qbezier[10](200,70)(210,70)(210,55) \qbezier[10](210,55)(210,40)(200,40)

\qbezier(200,30)(190,30)(190,15) \qbezier(190,15)(190,0)(200,0)
\qbezier[10](200,30)(210,30)(210,15) \qbezier[10](210,15)(210,0)(200,0)

\put(50,150){(i)} \put(245,150){(ii)}

\thicklines

\qbezier(110,150)(30,150)(30,75) \qbezier(30,75)(30,0)(110,0)

\qbezier(200,150)(280,150)(280,75) \qbezier(280,75)(280,0)(200,0)

\qbezier(110,120)(80,115)(110,110) \qbezier(110,80)(80,75)(110,70)
\qbezier(110,40)(80,35)(110,30) 

\qbezier(200,120)(230,115)(200,110) \qbezier(200,80)(230,75)(200,70)
\qbezier(200,40)(230,35)(200,30) 

\end{picture}

$\,$\\
To describe $1$-cycles on $M$, we consider the following key $1$-cycles:
\begin{align*}
C_1 &= \{(z,\,w)=(t,\,\sqrt{t^8+a t^4+1}) \>|\>t: 0\to \infty,\> \sqrt{*}>0 \}\\
& \qquad \cup \{(z,\,w)=(i t,\, \sqrt{t^8+a t^4+1} )\>|\>t: \infty\to 0 ,\> \sqrt{*}>0 \}, \\
C_2 &= \{(z,\,w)=(i t,\, \sqrt{t^8+a t^4+1})\>|\>t: 1 \to -1,\> \sqrt{*}>0 \} \\
 & \qquad \cup \{(z,\,w)=(e^{i t} ,\, w(t))\>| \> t:-\pi/2\to \pi/2, \> w (0) < 0\}. 
\end{align*}
We shall verify that $C_2$ defines a connected $1$-cycle later. Since $C_1\cap C_2\neq  \o $, 
we may choose $C_1$ and $C_2$ in the following figure. 

\begin{picture}(340,130)

\put(10,60){\line(1,0){60}} \put(20,0){\line(0,1){120}} 
\put(90,60){\line(1,0){60}} \put(100,0){\line(0,1){120}} 

\put(35,75){\circle*{2}} \put(55,95){\circle*{2}} \put(35,45){\circle*{2}} \put(55,25){\circle*{2}} 
\put(115,75){\circle*{2}} \put(135,95){\circle*{2}} \put(115,45){\circle*{2}} \put(135,25){\circle*{2}} 

\put(10,95){\line(1,0){45}} \put(10,25){\line(1,0){45}} \put(35,45){\line(0,1){30}} 
\put(90,95){\line(1,0){45}} \put(90,25){\line(1,0){45}} \put(115,45){\line(0,1){30}} 

\put(60,110){{\rm (i) }} \put(140,110){{\rm (ii) }}

\put(45,77){$-$} \put(35,84){$+$} \put(45,37){$-$} \put(35,30){$+$} 
\put(125,77){$-$} \put(115,84){$+$} \put(125,37){$-$} \put(115,30){$+$}

\put(230,110){\circle*{2}} \put(230,80){\circle*{2}} \put(230,70){\circle*{2}} \put(230,40){\circle*{2}} 
\put(260,110){\circle*{2}} \put(260,80){\circle*{2}} \put(260,70){\circle*{2}} \put(260,40){\circle*{2}} 

\qbezier(230,110)(220,110)(220,95) \qbezier(220,95)(220,80)(230,80)
\qbezier[10](230,110)(240,110)(240,95) \qbezier[10](240,95)(240,80)(230,80)

\qbezier[10](230,70)(240,70)(240,55) \qbezier[10](240,55)(240,40)(230,40)

\put(180,110){(i)} \put(290,110){(ii)} 

\qbezier(230,110)(170,110)(170,35) \qbezier(230,80)(200,75)(230,70) \qbezier(230,40)(200,35)(230,30)
\qbezier(230,30)(220,30)(220,15) \qbezier[12](230,30)(240,30)(240,15) 

\qbezier(230,70)(220,70)(220,55) \qbezier(220,55)(220,40)(230,40)

\qbezier(260,110)(250,110)(250,95) \qbezier(250,95)(250,80)(260,80)
\qbezier[10](260,110)(270,110)(270,95) \qbezier[10](270,95)(270,80)(260,80)

\qbezier[10](260,70)(270,70)(270,55) \qbezier[10](270,55)(270,40)(260,40)

\qbezier(260,110)(320,110)(320,35) \qbezier(260,80)(290,75)(260,70) \qbezier(260,40)(290,35)(260,30)
\qbezier(260,30)(250,30)(250,15) \qbezier[12](260,30)(270,30)(270,15) 

\qbezier(260,70)(250,70)(250,55) \qbezier(250,55)(250,40)(260,40)

\thicklines

\put(35,75){\line(1,1){20}} \put(35,45){\line(1,-1){20}} 
\put(115,75){\line(1,1){20}} \put(115,45){\line(1,-1){20}} 

\put(24,64){\line(1,0){46}} \put(50,64){\line(-2,1){5}} \put(50,64){\line(-2,-1){5}} 
\put(24,64){\line(0,1){56}} \put(24,90){\line(-1,2){3}} \put(24,90){\line(1,2){3}} \put(28,105){$C_1$} 
\put(16,34){\line(0,1){52}} \put(16,60){\line(-1,2){3}} \put(16,60){\line(1,2){3}} 
\qbezier(16,34)(30,34)(40,40) \qbezier(16,86)(30,86)(40,80) 
\put(28,35){\line(-4,1){6}} \put(28,35){\line(-3,-2){5}} 
\put(28,85){\line(3,1){5}} \put(28,85){\line(2,-3){3}} \put(22,40){$C_2$} 

\qbezier(120,40)(130,60)(120,80) \put(125,60){\line(-1,-2){3}} \put(125,60){\line(1,-2){3}} \put(126,65){$C_2$} 

\qbezier(190,60)(210,60)(215,75) \qbezier[5](215,75)(208,80)(200,80) 
\qbezier[3](203,80)(207,81)(211,82) \qbezier[3](203,80)(204,77)(205,74) 
\qbezier[8](200,80)(200,92)(195,100) \qbezier(195,100)(185,90)(190,60) 
\put(188,75){\line(-1,2){3}} \put(188,75){\line(1,2){3}} \put(185,50){$C_1$} 
\qbezier(220,95)(200,90)(195,80) \qbezier(195,80)(195,72)(196,65) \put(195,70){\line(-1,2){3}} \put(195,70){\line(1,2){3}} 
\qbezier(196,65)(210,55)(220,55) \put(205,97){$C_2$} 
\qbezier(270,55)(290,75)(270,95) \put(280,75){\line(1,-2){3}} \put(280,75){\line(-1,-1){5}} \put(282,80){$C_2$} 

\end{picture}
$\,$\\
After that, we shall give other $1$-cycles. From $C_1 \cap \varphi^2 (C_1) \neq \o$ and $C_1 \cap \varphi^3 (C_1) \neq \o$, 
we have the following two figures. 

\begin{picture}(340,150)

\put(0,70){\line(1,0){60}} \put(50,10){\line(0,1){120}} 
\put(35,55){\circle*{2}} \put(15,35){\circle*{2}} \put(35,85){\circle*{2}} \put(15,105){\circle*{2}} 
\put(15,105){\line(1,0){45}} \put(15,35){\line(1,0){45}} \put(35,55){\line(0,1){30}}  
\put(0,120){{\rm (i) }} 
\put(18,87){$-$} \put(27,94){$+$} \put(18,47){$-$} \put(27,40){$+$} 

\put(80,70){\line(1,0){60}} \put(130,10){\line(0,1){120}} 
\put(115,55){\circle*{2}} \put(95,35){\circle*{2}} \put(115,85){\circle*{2}} \put(95,105){\circle*{2}} 
\put(95,105){\line(1,0){45}} \put(95,35){\line(1,0){45}} \put(115,55){\line(0,1){30}}  
\put(80,120){{\rm (ii) }} 
\put(98,87){$-$} \put(107,94){$+$} \put(98,47){$-$} \put(107,40){$+$} 

\put(10,20){$\varphi^2 (C_1)$} 
\put(57,45){$\varphi^2 (C_2)$} \put(71,77){$\varphi^2 (C_2)$}

\thicklines

\put(35,85){\line(-1,1){20}} \put(35,55){\line(-1,-1){20}} 

\put(115,85){\line(-1,1){20}} \put(115,55){\line(-1,-1){20}} 

\put(45,10){\line(0,1){55}} \put(45,65){\line(-1,0){45}} \put(45,30){\line(-1,-2){3}} \put(45,30){\line(1,-2){3}} 
\put(15,65){\line(2,-1){6}} \put(15,65){\line(2,1){6}}

\put(54,44){\line(0,1){52}} \put(54,70){\line(-1,-2){3}} \put(54,70){\line(1,-2){3}} 
\qbezier(54,44)(40,44)(30,50) \qbezier(54,96)(40,96)(30,90) 
\put(42,45){\line(-1,2){2}} \put(42,45){\line(-2,-1){4}} 
\put(42,95){\line(1,1){4}} \put(42,95){\line(3,-2){5}} 

\qbezier(110,50)(100,70)(110,90) \put(105,70){\line(-1,2){3}} \put(105,70){\line(1,2){3}}

\thinlines 

\put(230,10){\circle*{2}} \put(230,40){\circle*{2}} \put(230,50){\circle*{2}} \put(230,80){\circle*{2}} 
\put(260,10){\circle*{2}} \put(260,40){\circle*{2}} \put(260,50){\circle*{2}} \put(260,80){\circle*{2}} 

\qbezier(230,120)(220,120)(220,105) \qbezier(220,105)(220,90)(230,90)
\qbezier[10](230,120)(240,120)(240,105) \qbezier[10](240,105)(240,90)(230,90)

\qbezier(230,80)(220,80)(220,65) \qbezier(220,65)(220,50)(230,50)
\qbezier[10](230,80)(240,80)(240,65) \qbezier[10](240,65)(240,50)(230,50)

\qbezier(230,40)(220,40)(220,25) \qbezier(220,25)(220,10)(230,10)
\qbezier[10](230,40)(240,40)(240,25) \qbezier[10](240,25)(240,10)(230,10)

\qbezier(260,120)(250,120)(250,105) \qbezier(250,105)(250,90)(260,90)
\qbezier[10](260,120)(270,120)(270,105) \qbezier[10](270,105)(270,90)(260,90)

\qbezier(260,80)(250,80)(250,65) \qbezier(250,65)(250,50)(260,50)
\qbezier[10](260,80)(270,80)(270,65) \qbezier[10](270,65)(270,50)(260,50)

\qbezier(260,40)(250,40)(250,25) \qbezier(250,25)(250,10)(260,10)
\qbezier[10](260,40)(270,40)(270,25) \qbezier[10](270,25)(270,10)(260,10)

\put(180,130){(i)} \put(285,130){(ii)}

\qbezier(165,85)(165,10)(230,10)

\qbezier(315,85)(315,10)(260,10)

\qbezier(230,90)(200,85)(230,80)
\qbezier(230,50)(200,45)(230,40) 

\qbezier(260,90)(290,85)(260,80)
\qbezier(260,50)(290,45)(260,40) 

\thicklines

\qbezier[10](200,115)(200,95)(215,85) \qbezier(215,85)(190,90)(180,105) 
\qbezier(180,105)(190,60)(215,45) \qbezier[20](215,45)(190,70)(200,115)
\qbezier[4](197,86)(196,81)(195,76) \qbezier[4](197,86)(201,82)(205,78)
\put(195,65){\line(-3,2){7}} \put(195,65){\line(0,1){8}} \put(166,110){$\varphi^2 (C_1)$}

\qbezier(220,65)(190,70)(175,95) \qbezier(175,95)(180,40)(220,25) 
\put(189,50){\line(-3,2){7}} \put(189,50){\line(0,1){8}} \put(170,30){$\varphi^2 (C_2)$} 

\qbezier(270,25)(290,45)(270,65) \put(280,45){\line(1,-2){3}} \put(280,45){\line(-1,-1){5}} \put(277,60){$\varphi^2 (C_2)$} 

\end{picture}

\begin{picture}(340,150)

\put(10,130){\line(1,0){105}} \put(60,80){\line(0,1){60}} 

\put(75,115){\circle*{2}} \put(95,95){\circle*{2}} \put(45,115){\circle*{2}} \put(25,95){\circle*{2}} 

\put(25,95){\line(1,0){70}} \put(45,115){\line(0,1){25}} \put(75,115){\line(0,1){25}} 

\put(20,140){{\rm (i) }} 

\put(85,107){$-$} \put(75,100){$+$} \put(28,107){$-$} \put(37,100){$+$}

\put(10,60){\line(1,0){105}} \put(60,10){\line(0,1){60}} 

\put(75,45){\circle*{2}} \put(95,25){\circle*{2}} \put(45,45){\circle*{2}} \put(25,25){\circle*{2}} 

\put(25,25){\line(1,0){70}} \put(45,45){\line(0,1){25}} \put(75,45){\line(0,1){25}} 

\put(20,70){{\rm (ii) }} 

\put(85,37){$-$} \put(75,30){$+$} \put(28,37){$-$} \put(37,30){$+$}

\put(100,115){$\varphi^3 (C_1)$} \put(0,120){$\varphi^3 (C_2)$} \put(45,22){$\varphi^3 (C_2)$} 

\thicklines

\put(75,115){\line(1,-1){20}} \put(45,115){\line(-1,-1){20}}

\put(75,45){\line(1,-1){20}} \put(45,45){\line(-1,-1){20}}

\put(65,125){\line(1,0){50}} \put(90,125){\line(2,1){6}} \put(90,125){\line(2,-1){6}} 
\put(65,125){\line(0,-1){45}} \put(65,105){\line(-1,2){3}} \put(65,105){\line(1,2){3}} 

\put(34,134){\line(1,0){52}} \put(60,134){\line(2,1){6}} \put(60,134){\line(2,-1){6}} 
\qbezier(34,134)(34,120)(40,110) \put(36,119){\line(-1,1){4}} \put(36,119){\line(1,3){2}} 
\qbezier(86,134)(86,120)(80,110) \put(84,119){\line(-3,-2){6}} \put(84,119){\line(1,-2){3}} 

\qbezier(40,40)(60,30)(80,40) \put(60,35){\line(-2,1){6}} \put(60,35){\line(-2,-1){6}}

\thinlines 

\put(230,40){\circle*{2}} \put(230,70){\circle*{2}} \put(230,80){\circle*{2}} \put(230,110){\circle*{2}} 
\put(260,40){\circle*{2}} \put(260,70){\circle*{2}} \put(260,80){\circle*{2}} \put(260,110){\circle*{2}} 

\qbezier(230,110)(220,110)(220,95) \qbezier(220,95)(220,80)(230,80)
\qbezier[10](230,110)(240,110)(240,95) \qbezier[10](240,95)(240,80)(230,80)

\qbezier(230,70)(220,70)(220,55) \qbezier(220,55)(220,40)(230,40)
\qbezier[10](230,70)(240,70)(240,55) \qbezier[10](240,55)(240,40)(230,40)

\qbezier(260,110)(250,110)(250,95) \qbezier(250,95)(250,80)(260,80)
\qbezier[10](260,110)(270,110)(270,95) \qbezier[10](270,95)(270,80)(260,80)

\qbezier(260,70)(250,70)(250,55) \qbezier(250,55)(250,40)(260,40)
\qbezier[10](260,70)(270,70)(270,55) \qbezier[10](270,55)(270,40)(260,40)

\put(180,140){(i)} \put(285,140){(ii)}

\qbezier(165,75)(165,40)(180,20) \qbezier(165,75)(165,110)(180,130) 

\qbezier(315,75)(315,40)(300,20) \qbezier(315,75)(315,110)(300,130) 

\qbezier(230,120)(200,115)(230,110)
\qbezier(230,80)(200,75)(230,70)
\qbezier(230,40)(200,35)(230,30) 

\qbezier(260,120)(290,115)(260,110)
\qbezier(260,80)(290,75)(260,70)
\qbezier(260,40)(290,35)(260,30) 

\put(215,70){$\varphi^3 (C_1)$} \put(170,45){$\varphi^3 (C_2)$} 
\put(282,80){$\varphi^3 (C_2)$} 

\thicklines 

\qbezier[10](240,95)(215,100)(215,115) 
\qbezier(215,115)(150,110)(215,35) \put(184,100){\line(0,1){10}} \put(184,100){\line(2,1){10}} 
\qbezier[10](215,35)(215,55)(240,55) 

\qbezier[10](250,55)(280,55)(280,75) \qbezier[10](280,75)(280,95)(250,95)
\qbezier[3](280,75)(277,72)(274,69) \qbezier[3](280,75)(282,71)(284,67) 

\qbezier(215,75)(175,100)(215,115) \put(197,90){\line(-2,1){8}} \put(197,90){\line(1,2){4}} 
\qbezier[10](215,75)(200,100)(215,115) 
\qbezier[3](208,101)(206,97)(204,93) \qbezier[3](208,101)(210,97)(212,93) 

\end{picture}

Therefore, we find a canonical homology basis as follows. 
\begin{align*}
A_1 &= C_2, \quad B_1 = C_1, \quad 
A_2 = -\varphi^2 (C_1)-\varphi^3 (C_1)+ \varphi^3 (C_2), \\
B_2 &= B_1 +\varphi^3 (C_1), \quad 
A_3 = \varphi^2 (C_2), \quad B_3 = B_2 +\varphi^2 (C_1). 
\end{align*}

\begin{picture}(340,100)

\qbezier(20,50)(20,90)(160,90) \qbezier(160,90)(300,90)(300,50) 
\qbezier(20,50)(20,10)(160,10) \qbezier(160,10)(300,10)(300,50) 

\qbezier(55,55)(80,35)(105,55) \qbezier(135,55)(160,35)(185,55) \qbezier(215,55)(240,35)(265,55) 
\qbezier(60,52)(80,62)(100,52) \qbezier(140,52)(160,62)(180,52) \qbezier(220,52)(240,62)(260,52)

\put(30,85){{\rm (ii)}} \put(30,15){{\rm (i)}} 

\qbezier(45,50)(45,30)(80,30) \qbezier(115,50)(115,30)(80,30) 
\qbezier(45,50)(45,70)(80,70) \qbezier(115,50)(115,70)(80,70) 

\qbezier(125,50)(125,30)(160,30) \qbezier(195,50)(195,30)(160,30) 
\qbezier(125,50)(125,70)(160,70) \qbezier(195,50)(195,70)(160,70) 

\qbezier(205,50)(205,30)(240,30) \qbezier(275,50)(275,30)(240,30) 
\qbezier(205,50)(205,70)(240,70) \qbezier(275,50)(275,70)(240,70) 

\qbezier(71,46)(55,38)(55,20) \qbezier[15](55,20)(75,25)(71,46) 

\qbezier(151,46)(135,35)(140,10) \qbezier[15](140,10)(160,25)(151,46) 

\qbezier(249,46)(265,38)(265,20) \qbezier[15](249,46)(243,25)(265,20)

\put(80,30){\line(-2,1){5}} \put(80,30){\line(-2,-1){5}} \put(100,25){$A_1$}
\put(160,30){\line(-2,1){5}} \put(160,30){\line(-2,-1){5}} \put(180,25){$A_2$}
\put(240,30){\line(-2,1){5}} \put(240,30){\line(-2,-1){5}} \put(275,35){$A_3$}

\put(62.5,40){\line(-2,-1){5}}  \put(62.5,40){\line(0,-1){5}}  \put(50,10){$B_1$}
\put(144.5,40){\line(-2,-1){5}} \put(144.5,40){\line(0,-1){5}} \put(135,0){$B_2$}
\put(257.5,40){\line(2,-1){5}} \put(257.5,40){\line(0,-1){5}} \put(260,10){$B_3$}

\end{picture}

\subsubsection{Period matrix}

Key $1$-cycles of the canonical homology basis as in \S~\ref{tP-canonical} are given by $C_1$ and $C_2$. 
First, we have 
\begin{align}
\int_{C_1} \dfrac{1-z^2}{w} dz &= -i \int_0^{\infty} \dfrac{1+t^2}{\sqrt{t^8+a t^4+1}} dt
= -2i \int_0^1 \dfrac{1+t^2}{\sqrt{t^8+a t^4+1}} dt, \label{tP-period1} 
\end{align}
\begin{align}
\int_{C_1} \dfrac{i (1+z^2)}{w} dz &= 
2i \int_0^1 \dfrac{1+t^2}{\sqrt{t^8+a t^4+1}} dt, \label{tP-period2}\\
\int_{C_1} \dfrac{2 z}{w} dz &= 
8 \int_0^1 \dfrac{t}{\sqrt{t^8+a t^4+1}} dt, \label{tP-period3}\\
\int_{C_1} \dfrac{z^4-z^6}{w^3} dz &= 
-2i \int_0^1 \dfrac{t^4+t^6}{\sqrt{t^8+a t^4+1}^3} dt, \label{tP-period4}\\
\int_{C_1} \dfrac{i(z^4+z^6)}{w^3} dz &= 
2i \int_0^1 \dfrac{t^4+t^6}{\sqrt{t^8+a t^4+1}^3} dt, \label{tP-period5}\\
\int_{C_1} \dfrac{z^5}{w^3} dz &= 
4 \int_0^1 \dfrac{t^5}{\sqrt{t^8+a t^4+1}^3} dt. \label{tP-period6}
\end{align}
Next we calculate periods along $C_2$. 
Set $x=(z+1/z)/2=\cos t:0\to 1\to 0$ ($t:-\pi/2 \to 0 \to \pi/2$) along a circle part of $C_2$. 
Then $\dfrac{1-z^2}{w}dz=-2\dfrac{z^2}{w} dx$ and $\dfrac{z^4-z^6}{w^3}dz=-2\left(\dfrac{z^2}{w}\right)^3 dx$. 
$z+1/z=2x$ implies that $z^2+1/z^2=4x^2-2$, and thus $z^4+1/z^4=16 x^4-16 x^2+2$. Hence, 
\begin{align*}
\left( \dfrac{z^2}{w}  \right)^2&=\dfrac{z^4}{z^8+a z^4+1} =\dfrac{1}{z^4+\frac{1}{z^4}+a}
=\dfrac{1}{16x^4-16x^2+2+a}>0.
\end{align*}
To choose a suitable branch, we substitute $t=0$ to $z^2/w$. Then 
$\dfrac{z^2}{w}(t=0)=\dfrac{1}{w(0)}<0$. As a result, 
\begin{align}
\dfrac{z^2}{w} = -\dfrac{1}{\sqrt{ 16x^4-16x^2+2+a }} <0.  \label{tP-branch1}
\end{align}
Consequently, we have 
\begin{align}
\nonumber \int_{C_2} \dfrac{1-z^2}{w}dz&= i \int^{-1}_1  \dfrac{1+t^2}{\sqrt{ t^8+a t^4+1}}dt 
+ \int_0^1 \dfrac{2 }{\sqrt{ 16 x^4-16 x^2+2+a}}dx \\
\nonumber &\qquad +\int^0_1 \dfrac{2 }{\sqrt{ 16 x^4-16 x^2+2+a}}dx \\
&= -2 i \int^{1}_0  \dfrac{1+t^2}{\sqrt{ t^8+a t^4+1}}dt, \label{tP-period7}\\
\nonumber \int_{C_2} \dfrac{z^4-z^6}{w^3}dz&= i \int^{-1}_1  \dfrac{t^4+t^6}{\sqrt{ t^8+a t^4+1}^3}dt 
+ \int_0^1 \dfrac{2 }{\sqrt{ 16 x^4-16 x^2+2+a}^3}dx \\
\nonumber &\qquad +\int^0_1 \dfrac{2 }{\sqrt{ 16 x^4-16 x^2+2+a}^3}dx \\
&= -2 i \int^{1}_0  \dfrac{t^4+t^6}{\sqrt{ t^8+a t^4+1}^3}dt. \label{tP-period8}
\end{align}
Moreover, from \eqref{tP-branch1} and substituting $t=\pm \pi/2$ to $z^2/w$, we find 
$\dfrac{-1}{w(\pm \pi/2)} <0$. 
So $w(\pm \pi/2) >0$. 
Thus, $C_2$ defines a connected $1$-cycle. 

We set $x=(z-1/z)/(2 i)=\sin t:-1\to 1$ ($t:-\pi/2 \to \pi/2$) along a circle part of $C_2$. 
Then $\dfrac{i(1+z^2)}{w}dz=-2\dfrac{z^2}{w} dx$ and $\dfrac{i(z^4+z^6)}{w^3}dz=-2\left(\dfrac{z^2}{w}\right)^3 dx$. 
$z-1/z=2 i x$ implies that $z^2+1/z^2=-4x^2+2$, and thus $z^4+1/z^4=16 x^4-16 x^2+2$. Hence, 
\begin{align*}
\left( \dfrac{z^2}{w}  \right)^2&=\dfrac{z^4}{z^8+a z^4+1} =\dfrac{1}{z^4+\frac{1}{z^4}+a}
=\dfrac{1}{16x^4-16x^2+2+a}>0.
\end{align*}
To choose a suitable branch, we substitute $t=0$ to $z^2/w$. Then 
$\dfrac{z^2}{w}(t=0)=\dfrac{1}{w(0)}<0$. As a result, 
\begin{align*}
\dfrac{z^2}{w} = -\dfrac{1}{\sqrt{ 16x^4-16x^2+2+a }} <0.  
\end{align*}
Consequently, we have 
\begin{align}
\nonumber \int_{C_2} \dfrac{i(1+z^2)}{w}dz&= \int^{1}_{-1}  \dfrac{1-t^2}{\sqrt{ t^8+a t^4+1}}dt 
+ \int_{-1}^1 \dfrac{2 }{\sqrt{ 16 x^4-16 x^2+2+a}}dx \\
&= 2 \int^{1}_{0}  \dfrac{1-t^2}{\sqrt{ t^8+a t^4+1}}dt + 4 \int_{0}^1 \dfrac{dx }{\sqrt{ 16 x^4-16 x^2+2+a}}, \label{tP-period9}\\
\nonumber \int_{C_2} \dfrac{i(z^4+z^6)}{w^3}dz&= \int_{-1}^1  \dfrac{t^4-t^6}{\sqrt{ t^8+a t^4+1}^3}dt 
+ \int_{-1}^1 \dfrac{2 }{\sqrt{ 16 x^4-16 x^2+2+a}^3}dx \\
&= 2 \int_{0}^1  \dfrac{t^4-t^6}{\sqrt{ t^8+a t^4+1}^3}dt 
+ 4\int_{0}^1 \dfrac{dx }{\sqrt{ 16 x^4-16 x^2+2+a}^3}. \label{tP-period10}
\end{align}

Setting $x=(z^2-1)/(i(z^2+1))=\tan t:-\infty\to \infty$ ($t:-\pi/2 \to \pi/2$) along a circle part of $C_2$, we find 
$\dfrac{2 z}{w}dz=\dfrac{i}{2}\dfrac{(z^2+1)^2}{w} dx$ and $\dfrac{z^5}{w^3}dz=\dfrac{i}{4}\dfrac{z^4}{w^2}\dfrac{(z^2+1)^2}{w} dx$. 
$z^2=\dfrac{1-x^2}{1+x^2}+i\dfrac{2x}{1+x^2}$ implies that $z^2+1/z^2=\dfrac{2(1-x^2)}{1+x^2}$, 
and thus $z^4+1/z^4=\dfrac{2x^4-12x^2+2}{(1+x^2)^2}$. Hence, 
\begin{align*}
\left( \dfrac{(z^2+1)^2}{w}  \right)^2&=\dfrac{16}{(2+a)x^4+(2a-12)x^2+2+a}>0.
\end{align*}
To choose a suitable branch, we substitute $t=0$ to $(z^2+1)^2/w$. Then 
$\dfrac{(z^2+1)^2}{w}(t=0)=\dfrac{4}{w(0)}<0$. As a result, 
\begin{align*}
\dfrac{(z^2+1)^2}{w} = -\dfrac{4}{\sqrt{ (2+a)x^4+(2a-12)x^2+2+a }} <0.  
\end{align*}
Similarly, we obtain 
\[
\dfrac{z^4}{w^2}=\dfrac{1}{z^4+\frac{1}{z^4}+a}=\dfrac{(1+x^2)^2}{(2+a)x^4+(2a-12)x^2+2+a}. 
\]
Consequently, we have 
\begin{align}
\nonumber &\int_{C_2} \dfrac{2 z}{w}dz= \int^{1}_{-1}  \dfrac{2 t}{\sqrt{ t^8+a t^4+1}}dt 
-2i \int_{-\infty}^{\infty} \dfrac{dx }{\sqrt{ (2+a) x^4+(2a-12) x^2+2+a}} \\
&\qquad = -8i \int_{0}^{1} \dfrac{dx }{\sqrt{ (2+a) x^4+(2a-12) x^2+2+a}},\label{tP-period11}\\
\nonumber &\int_{C_2} \dfrac{z^5}{w^3}dz = \int_{-1}^1  \dfrac{t^5}{\sqrt{ t^8+a t^4+1}^3}dt 
-i \int_{-\infty}^{\infty} \dfrac{(1+x^2)^2 }{\sqrt{ (2+a) x^4+(2a-12) x^2+2+a}^3}dx \\
&\qquad = -4i \int_{0}^{1} \dfrac{(1+x^2)^2 }{\sqrt{ (2+a) x^4+(2a-12) x^2+2+a}^3}dx. \label{tP-period12}
\end{align}
Set 
\begin{align*}
A &= 2 \int^{1}_{0}  \dfrac{1-t^2}{\sqrt{ t^8+a t^4+1}}dt + 4 \int_{0}^1 \dfrac{dt }{\sqrt{ 16 t^4-16 t^2+2+a}},\\
B &= 2 \int^{1}_0  \dfrac{1+t^2}{\sqrt{ t^8+a t^4+1}}dt, \quad C= 8 \int_0^1 \dfrac{t}{\sqrt{t^8+a t^4+1}} dt, \\
D &= 8 \int_{0}^{1} \dfrac{dt }{\sqrt{ (2+a) t^4+(2a-12) t^2+2+a}}, \\
E &=  2 \int_{0}^1  \dfrac{t^4-t^6}{\sqrt{ t^8+a t^4+1}^3}dt + 4\int_{0}^1 \dfrac{dt }{\sqrt{ 16 t^4-16 t^2+2+a}^3}, \\
F &= 2 \int^{1}_0  \dfrac{t^4+t^6}{\sqrt{ t^8+a t^4+1}^3}dt, \quad H= 4 \int_0^1 \dfrac{t^5}{\sqrt{t^8+a t^4+1}^3} dt, \\
I &= 4 \int_{0}^{1} \dfrac{(1+t^2)^2 }{\sqrt{ (2+a) t^4+(2a-12) t^2+2+a}^3}dt. 
\end{align*}
By \eqref{tP-period1}--\eqref{tP-period6} and \eqref{tP-period7}--\eqref{tP-period12}, we have 
\[
\int_{C_1} G=
\begin{pmatrix}
-i B \\
i B \\
C \\
-i F \\
i F \\
H
\end{pmatrix}, \quad 
\int_{C_2} G=
\begin{pmatrix}
-i B \\
A \\
-i D \\
-i F \\
E \\
-i I
\end{pmatrix}. 
\]
Therefore, the period matrix of the abelian differentials of the second kind is given by 
\[
\begin{pmatrix}
- i B   &   -A  &  i B  & -i B  &   -2 i B   &  -i B \\
A  &   i B  &  -A  &  i B  &  0  &  -i B \\
-i D   &   i D  &  -i D  &  C  &  0  &  C \\
- i F   &   -E  &  i F  &  -i F  & -2 i F  &   -i F \\ 
E  &  i F   &  -E  &  i F  & 0  &  -i F  \\ 
-i I  &   i I  &   -i I   &  H  &  0  &  H
\end{pmatrix}. 
\]

\subsection{tCLP family}\label{tCLP-detail}

\subsubsection{canonical homology basis}\label{tCLP-canonical} 

Let $M$ be a hyperelliptic Riemann surface of genus $3$ defined as the completion of 
$\{(z,\,w)\,|\, w^2=z^8 + a z^4+1 \} \subset \mathbb{C}^2$ for $a\in (-2,\,2)$. 
It suffices to consider the case $a\in [0,\,2)$ because we obtain the same result for $a\in (-2,\,0]$. 
The three differentials 
\[\dfrac{dz}{w},\,z\dfrac{dz}{w},\,z^2\dfrac{dz}{w}\]
form a basis for the abelian differentials of the first kind. 
Up to exact forms, the abelian differentials of the second kind are given by the following six differentials: 
\[\dfrac{dz}{w},\,z\dfrac{dz}{w},\,z^2\dfrac{dz}{w},\,\dfrac{z^4}{w^3}dz,\,\dfrac{z^5}{w^3}dz,\,\dfrac{z^6}{w^3}dz.\]

Let 
\[G=\left(\dfrac{1-z^2}{w},\,\dfrac{i(1+z^2)}{w},\,\dfrac{2z}{w},\,
\dfrac{z^4-z^6}{w^3},\,\dfrac{i(z^4+z^6)}{w^3},\,\dfrac{z^5}{w^3}\right)^t dz \]
and consider the biholomorphisms 
\begin{align*}
j(z,\,w)&=(z,\,-w),\>\>\varphi(z,\,w)=(i z,\,w)
\end{align*}
on $M$. Then it is straightforward to compute that 
\begin{align*}
j^{*}G=-G, \quad 
\varphi^{*} G=
\begin{pmatrix}
0 & 1 & 0 & 0 & 0 & 0 \\
-1 & 0 & 0 & 0 & 0 & 0 \\
0 & 0 & -1 & 0 & 0 & 0 \\
0 & 0 & 0 & 0 & 1 & 0 \\
0 & 0 & 0 & -1 & 0 & 0 \\
0 & 0 & 0 & 0 & 0 & -1
\end{pmatrix}G. 
\end{align*}
Now we determine a canonical homology basis on $M$. Recall that 
\begin{align*}
\pi_{{\rm tCLP}}:&\quad M \longrightarrow \overline{\mathbb{C}}:=\mathbb{C} \cup \{\infty\} \\
& (z,\,w) \> \longmapsto \> z
\end{align*}
defines a two-sheeted branched covering and $j$ is its deck transformation. 
Set $e^{i\alpha} := -\frac{a}{2}+i \frac{\sqrt{4-a^2}}{2}\in S^1\subset \mathbb{C}$ ($\alpha \in [\pi/2,\,\pi) $). 
$\pi_{{\rm tCLP}}$ has branch locus
\[
\{ e^{\pm \frac{\alpha}{4}i},\> i e^{\pm \frac{\alpha}{4} i},\> 
- e^{\pm \frac{\alpha}{4}i}, \> -i e^{\pm \frac{\alpha}{4} i}\}. 
\]
So $M$ can be expressed as a $2$-sheeted branched cover of $\overline{\mathbb{C}}$ in the following way. \\

\begin{picture}(340,150)

\put(10,90){\line(1,0){130}} \put(170,90){\line(1,0){130}}
\put(70,30){\line(0,1){120}} \put(230,30){\line(0,1){120}}

\put(70,90){\circle{90}} \put(230,90){\circle{90}} 

\put(110.5,110){\circle*{2}} \put(110.5,70){\circle*{2}} \put(29.5,110){\circle*{2}} \put(29.5,70){\circle*{2}} 
\put(50,130.5){\circle*{2}} \put(90,130.5){\circle*{2}} \put(50,49.5){\circle*{2}} \put(90,49.5){\circle*{2}} 

\put(270.5,110){\circle*{2}} \put(270.5,70){\circle*{2}} \put(189.5,110){\circle*{2}} \put(189.5,70){\circle*{2}} 
\put(210,130.5){\circle*{2}} \put(250,130.5){\circle*{2}} \put(210,49.5){\circle*{2}} \put(250,49.5){\circle*{2}} 

\put(120,140){{\rm (i) }} \put(280,140){{\rm (ii) }}

\put(130,10){{\rm \textbf{figure (tCLP)}}}

\put(103,95){$+$} \put(118,95){$-$} \put(29,95){$+$} \put(14,95){$-$} \put(73,126){$+$} \put(73,137){$-$} 
\put(73,49){$+$} \put(73,38){$-$} 

\put(263,95){$+$} \put(278,95){$-$} \put(189,95){$+$} \put(174,95){$-$} \put(233,126){$+$} \put(233,137){$-$} 
\put(233,49){$+$} \put(233,38){$-$} 

\put(115,110){$e^{\frac{\alpha}{4}i}$} \put(115,70){$e^{-\frac{\alpha}{4}i}$} \put(93,131){$i e^{-\frac{\alpha}{4}i}$} 
\put(30,131){$i e^{\frac{\alpha}{4}i}$} \put(0,110){$-e^{-\frac{\alpha}{4}i}$} \put(3,65){$-e^{\frac{\alpha}{4}i}$} 
\put(20,38){$-i e^{-\frac{\alpha}{4}i}$} \put(90,43){$-i e^{\frac{\alpha}{4}i}$} 

\thicklines

\qbezier(110.5,110)(120,90)(110.5,70) \qbezier(29.5,110)(20,90)(29.5,70) 
\qbezier(50,130.5)(70,140)(90,130.5) \qbezier(50,49.5)(70,40)(90,49.5) 

\qbezier(270.5,110)(280,90)(270.5,70) \qbezier(189.5,110)(180,90)(189.5,70) 
\qbezier(210,130.5)(230,140)(250,130.5) \qbezier(210,49.5)(230,40)(250,49.5) 

\end{picture}
$\,$\\
We prepare two copies of $\overline{\mathbb{C}}$ and slit them along the thick lines in {\rm figure (tCLP)}. 
Identifying each of the upper (resp. lower) edges of the thick lines in {\rm (i)} with 
each of the lower (resp. upper) edges of the thick lines in (ii), we obtain the hyperelliptic Riemann surface $M$ of genus $3$ 
(see the following figure). 
Note that each of thin lines joining two branch points in {\rm figure (tCLP)} is corresponding to each of thick lines 
joining two branch points in the following figure. \\

\begin{picture}(340,160)

\put(135,148){$(e^{\frac{\alpha}{4}i},0)$} \put(134,122){$(e^{-\frac{\alpha}{4}i},\,0)$} 
\put(132,105){$(-i e^{\frac{\alpha}{4}i},\,0)$} 
\put(130,82){$(-i e^{-\frac{\alpha}{4}i},\,0)$} \put(132,65){$(-e^{\frac{\alpha}{4}i} ,\,0)$} 
\put(130,40){$(-e^{-\frac{\alpha}{4}i},\,0)$}
\put(134,24){$(i e^{\frac{\alpha}{4}i},\,0)$} \put(133,0){$(i e^{-\frac{\alpha}{4}i},\,0)$} 

\put(122,134){$-$} \put(122,93){$-$} \put(122,53){$-$} \put(122,13){$-$} 
\put(212,133){$-$} \put(212,93){$-$} \put(212,53){$-$} \put(212,13){$-$} 

\put(90,133){$+$} \put(90,93){$+$} \put(90,53){$+$} \put(90,13){$+$} 
\put(180,133){$+$} \put(180,93){$+$} \put(180,53){$+$} \put(180,13){$+$} 

\put(110,150){\circle*{2}} \put(110,120){\circle*{2}} \put(110,110){\circle*{2}} \put(110,80){\circle*{2}} 
\put(110,70){\circle*{2}} \put(110,40){\circle*{2}} \put(110,30){\circle*{2}} \put(110,0){\circle*{2}}
\put(200,150){\circle*{2}} \put(200,120){\circle*{2}} \put(200,110){\circle*{2}} \put(200,80){\circle*{2}} 
\put(200,70){\circle*{2}} \put(200,40){\circle*{2}} \put(200,30){\circle*{2}} \put(200,0){\circle*{2}}

\qbezier(110,150)(100,150)(100,135) \qbezier(100,135)(100,120)(110,120)
\qbezier[10](110,150)(120,150)(120,135) \qbezier[10](120,135)(120,120)(110,120)

\qbezier(110,110)(100,110)(100,95) \qbezier(100,95)(100,80)(110,80)
\qbezier[10](110,110)(120,110)(120,95) \qbezier[10](120,95)(120,80)(110,80)

\qbezier(110,70)(100,70)(100,55) \qbezier(100,55)(100,40)(110,40)
\qbezier[10](110,70)(120,70)(120,55) \qbezier[10](120,55)(120,40)(110,40)

\qbezier(110,30)(100,30)(100,15) \qbezier(100,15)(100,0)(110,0)
\qbezier[10](110,30)(120,30)(120,15) \qbezier[10](120,15)(120,0)(110,0)

\qbezier(200,150)(190,150)(190,135) \qbezier(190,135)(190,120)(200,120)
\qbezier[10](200,150)(210,150)(210,135) \qbezier[10](210,135)(210,120)(200,120)

\qbezier(200,110)(190,110)(190,95) \qbezier(190,95)(190,80)(200,80)
\qbezier[10](200,110)(210,110)(210,95) \qbezier[10](210,95)(210,80)(200,80)

\qbezier(200,70)(190,70)(190,55) \qbezier(190,55)(190,40)(200,40)
\qbezier[10](200,70)(210,70)(210,55) \qbezier[10](210,55)(210,40)(200,40)

\qbezier(200,30)(190,30)(190,15) \qbezier(190,15)(190,0)(200,0)
\qbezier[10](200,30)(210,30)(210,15) \qbezier[10](210,15)(210,0)(200,0)

\put(50,150){(i)} \put(245,150){(ii)}

\thicklines

\qbezier(110,150)(30,150)(30,75) \qbezier(30,75)(30,0)(110,0)

\qbezier(200,150)(280,150)(280,75) \qbezier(280,75)(280,0)(200,0)

\qbezier(110,120)(80,115)(110,110) \qbezier(110,80)(80,75)(110,70)
\qbezier(110,40)(80,35)(110,30) 

\qbezier(200,120)(230,115)(200,110) \qbezier(200,80)(230,75)(200,70)
\qbezier(200,40)(230,35)(200,30) 

\end{picture}
$\,$\\
To describe $1$-cycles on $M$, we consider the following key $1$-cycles:
\begin{align*}
C_1 &= \{(z,\,w)=(t e^{\frac{\pi}{4}i},\,\sqrt{t^8-a t^4+1}) \>|\>t: \infty\to 0,\> \sqrt{*}>0 \}\\
& \qquad \cup \{(z,\,w)=(t e^{-\frac{\pi}{4}i},\, \sqrt{t^8-a t^4+1} )\>|\>t: 0 \to \infty ,\> \sqrt{*}>0 \}, \\
C_2 &= \{(z,\,w)=(t,\, \sqrt{t^8+a t^4+1})\>|\>t: \infty \to 0,\> \sqrt{*}>0 \} \\
 & \qquad \cup \{(z,\,w)=(-i t ,\, \sqrt{t^8+a t^4+1})\>| \> t:0\to \infty, \> \sqrt{*}>0\}. 
\end{align*}
Since $C_1\cap C_2\neq  \o $, we may choose $C_1$ and $C_2$ in the following figure. \\

\begin{picture}(340,120)

\put(0,60){\line(1,0){70}} \put(10,0){\line(0,1){120}} 
\qbezier(30,19.5)(43,26)(50.5,40) \qbezier(30,100.5)(43,94)(50.5,80) 
\put(50.5,80){\circle*{2}} \put(50.5,40){\circle*{2}} \put(30,19.5){\circle*{2}} 
\put(50,110){{\rm (i) }} 
\put(43,65){$+$} \put(58,65){$-$}  \put(13,19){$+$} \put(13,8){$-$} 

\put(90,60){\line(1,0){70}} \put(100,0){\line(0,1){120}} 
\qbezier(120,19.5)(133,26)(140.5,40) \qbezier(120,100.5)(133,94)(140.5,80) 
\put(140.5,80){\circle*{2}} \put(140.5,40){\circle*{2}} \put(120,19.5){\circle*{2}} 
\put(140,110){{\rm (ii) }} 
\put(133,65){$+$} \put(148,65){$-$}  \put(103,19){$+$} \put(103,8){$-$} 

\put(15,80){$C_1$} \put(20,30){$C_1$} \put(30,48){$C_2$} \put(-5,25){$C_2$} 
\put(85,4){$C_2$} \put(148,48){$C_2$} 

\thicklines

\qbezier(50.5,80)(60,60)(50.5,40) \qbezier(-10,19.5)(10,10)(30,19.5) 

\qbezier(140.5,80)(150,60)(140.5,40) \qbezier(80,19.5)(100,10)(120,19.5) 

\put(10,60){\line(1,1){42}} \put(30,80){\line(1,3){3}} \put(30,80){\line(3,1){9}} 
\put(10,60){\line(1,-1){42}} \put(30,40){\line(-3,1){9}} \put(30,40){\line(-1,3){3}} 

\put(10,60){\line(1,0){45}} \put(30,60){\line(2,1){7}} \put(30,60){\line(2,-1){7}} 
\put(10,60){\line(0,-1){45}} \put(10,35){\line(-1,2){4}} \put(10,35){\line(1,2){4}}

\put(100,15){\line(0,-1){15}} \put(100,4){\line(-1,2){3}} \put(100,4){\line(1,2){3}} 
\put(145,60){\line(1,0){15}} \put(150,60){\line(2,1){6}} \put(150,60){\line(2,-1){6}}

\thinlines

\put(230,110){\circle*{2}} \put(230,80){\circle*{2}} \put(230,70){\circle*{2}} \put(230,40){\circle*{2}} 
\put(260,110){\circle*{2}} \put(260,80){\circle*{2}} \put(260,70){\circle*{2}} \put(260,40){\circle*{2}} 

\qbezier(230,110)(220,110)(220,95) \qbezier(220,95)(220,80)(230,80)
\qbezier[10](230,110)(240,110)(240,95) \qbezier[10](240,95)(240,80)(230,80)

\qbezier[10](230,70)(240,70)(240,55) \qbezier[10](240,55)(240,40)(230,40)

\put(180,110){(i)} \put(290,110){(ii)} 

\qbezier(230,110)(170,110)(170,35) \qbezier(230,80)(200,75)(230,70) \qbezier(230,40)(200,35)(230,30)
\qbezier(230,30)(220,30)(220,15) \qbezier[12](230,30)(240,30)(240,15) 

\qbezier(230,70)(220,70)(220,55) \qbezier(220,55)(220,40)(230,40)

\qbezier(260,110)(250,110)(250,95) \qbezier(250,95)(250,80)(260,80)
\qbezier[10](260,110)(270,110)(270,95) \qbezier[10](270,95)(270,80)(260,80)

\qbezier[10](260,70)(270,70)(270,55) \qbezier[10](270,55)(270,40)(260,40)

\qbezier(260,110)(320,110)(320,35) \qbezier(260,80)(290,75)(260,70) \qbezier(260,40)(290,35)(260,30)
\qbezier(260,30)(250,30)(250,15) \qbezier[12](260,30)(270,30)(270,15) 

\qbezier(260,70)(250,70)(250,55) \qbezier(250,55)(250,40)(260,40)

\put(174,55){$C_1$} \put(204,47){$C_2$} \put(280,55){$C_2$} 

\thicklines

\qbezier(215,75)(160,40)(200,103) \qbezier[10](215,75)(220,100)(200,103) 
\put(192,89){\line(0,1){9}} \put(192,89){\line(2,1){8}} 

\qbezier(220,55)(170,60)(220,95) \put(200,78){\line(1,3){3}} \put(200,78){\line(3,1){9}} 

\qbezier(270,55)(300,75)(270,95) \put(285,75){\line(-1,-1){6}} \put(285,75){\line(1,-3){3}} 

\end{picture}

Next we shall consider other $1$-cycles. From $C_1\cap (\varphi^3 (C_1)\cap \varphi^3 (C_2))\neq \o$ and 
$C_1\cap (\varphi^2 (C_1)\cap \varphi^2 (C_2))\neq \o$, we have the following two figures. 

\begin{picture}(340,130)

\put(-10,60){\line(1,0){70}} \put(50,0){\line(0,1){110}} 

\qbezier(30,19.5)(17,26)(9.5,40) \qbezier(70,19.5)(83,26)(90.5,40) 
\put(9.5,80){\circle*{2}} \put(9.5,40){\circle*{2}} \put(30,19.5){\circle*{2}} 
\put(10,110){{\rm (i) }} 
\put(9,65){$+$} \put(-6,65){$-$}  \put(40,19){$+$} \put(40,8){$-$} 

\put(100,60){\line(1,0){70}} \put(160,0){\line(0,1){110}} 

\qbezier(140,19.5)(127,26)(119.5,40) 
\put(119.5,80){\circle*{2}} \put(119.5,40){\circle*{2}} \put(140,19.5){\circle*{2}} 
\put(120,110){{\rm (ii) }} 
\put(119,65){$+$} \put(104,65){$-$}  \put(150,19){$+$} \put(150,8){$-$} 

\put(68,47){$\varphi^3 (C_1)$} \put(0,10){$\varphi^3 (C_1)$} 
\put(19,68){$\varphi^3 (C_2)$} \put(53,20){$\varphi^3 (C_2)$} 

\put(165,5){$\varphi^3 (C_2)$} \put(104,48){$\varphi^3 (C_2)$} 

\thicklines

\qbezier(9.5,80)(0,60)(9.5,40) \qbezier(70,19.5)(50,10)(30,19.5) 

\qbezier(119.5,80)(110,60)(119.5,40) \qbezier(180,19.5)(160,10)(140,19.5) 

\put(50,60){\line(1,-1){40}} \put(70,40){\line(3,-1){9}} \put(70,40){\line(1,-3){3}} 
\put(50,60){\line(-1,-1){40}} \put(30,40){\line(1,3){3}} \put(30,40){\line(3,1){9}} 

\put(50,60){\line(0,-1){45}} \put(50,45){\line(1,-2){4}} \put(50,45){\line(-1,-2){4}} 
\put(50,60){\line(-1,0){45}} \put(25,60){\line(2,-1){8}} \put(25,60){\line(2,1){8}} 

\put(160,15){\line(0,-1){15}} \put(160,9){\line(1,-2){3}} \put(160,9){\line(-1,-2){3}} 
\put(115,60){\line(-1,0){15}} \put(105,60){\line(2,1){5}} \put(105,60){\line(2,-1){5}}

\thinlines

\put(250,20){\circle*{2}} \put(250,50){\circle*{2}} \put(250,60){\circle*{2}} \put(250,90){\circle*{2}} 
\put(280,20){\circle*{2}} \put(280,50){\circle*{2}} \put(280,60){\circle*{2}} \put(280,90){\circle*{2}} 

\qbezier(250,90)(240,90)(240,75) \qbezier(240,75)(240,60)(250,60)
\qbezier[10](250,90)(260,90)(260,75) \qbezier[10](260,75)(260,60)(250,60)

\qbezier(250,50)(240,50)(240,35) \qbezier(240,35)(240,20)(250,20)
\qbezier[10](250,50)(260,50)(260,35) \qbezier[10](260,35)(260,20)(250,20)

\qbezier(280,90)(270,90)(270,75) \qbezier(270,75)(270,60)(280,60)
\qbezier[10](280,90)(290,90)(290,75) \qbezier[10](290,75)(290,60)(280,60)

\qbezier(280,50)(270,50)(270,35) \qbezier(270,35)(270,20)(280,20)
\qbezier[10](280,50)(290,50)(290,35) \qbezier[10](290,35)(290,20)(280,20)

\put(200,120){(i)} \put(305,120){(ii)}

\qbezier(195,55)(195,20)(210,0) \qbezier(195,55)(195,90)(210,110) 

\qbezier(325,55)(325,20)(310,0) \qbezier(325,55)(325,90)(310,110) 

\qbezier(250,100)(220,95)(250,90)
\qbezier(250,60)(220,55)(250,50)
\qbezier(250,20)(220,15)(250,10) 

\qbezier(280,100)(310,95)(280,90)
\qbezier(280,60)(310,55)(280,50)
\qbezier(280,20)(310,15)(280,10) 

\put(190,65){$\varphi^3 (C_1)$} \put(205,30){$\varphi^3 (C_2)$} \put(290,27){$\varphi^3 (C_2)$} 

\thicklines

\qbezier(235,95)(190,100)(235,55) \qbezier[7](235,95)(225,100)(220,105) \qbezier[15](235,55)(225,70)(220,105) 
\put(216,78){\line(0,1){9}} \put(216,78){\line(-2,1){8}} 

\qbezier(240,75)(200,60)(240,35) \put(227,45){\line(0,1){6}} \put(227,45){\line(-3,1){6}} 

\qbezier(310,75)(310,40)(290,35) \qbezier(290,75)(310,110)(310,75) 
 \put(310,75){\line(1,-2){4}} \put(310,75){\line(-1,-2){4}} 

\end{picture}

\begin{picture}(340,150)

\put(-10,60){\line(1,0){70}} \put(50,10){\line(0,1){110}} 

\qbezier(30,19.5)(17,26)(9.5,40) \qbezier(30,100.5)(17,94)(9.5,80) 
\put(9.5,80){\circle*{2}} \put(9.5,40){\circle*{2}} \put(30,100.5){\circle*{2}} \put(70,100.5){\circle*{2}} 
\put(10,110){{\rm (i) }} 
\put(9,65){$+$} \put(-6,65){$-$}  \put(40,96){$+$} \put(40,107){$-$}

\put(80,60){\line(1,0){70}} \put(140,10){\line(0,1){110}} 

\qbezier(120,19.5)(107,26)(99.5,40) \qbezier(120,100.5)(107,94)(99.5,80) 
\put(99.5,80){\circle*{2}} \put(99.5,40){\circle*{2}} \put(120,100.5){\circle*{2}} \put(160,100.5){\circle*{2}} 
\put(100,110){{\rm (ii) }} 
\put(99,65){$+$} \put(84,65){$-$}  \put(128,96){$+$} \put(128,106){$-$} 

\put(24,26){$\varphi^2 (C_1)$} \put(-10,85){$\varphi^2 (C_1)$} \put(10,48){$\varphi^2 (C_2)$} \put(55,85){$\varphi^2 (C_2)$} \put(145,110){$\varphi^2 (C_2)$} 
\put(80,48){$\varphi^2 (C_2)$} 

\thicklines

\qbezier(9.5,80)(0,60)(9.5,40) \qbezier(70,100.5)(50,110)(30,100.5) 

\qbezier(99.5,80)(90,60)(99.5,40) \qbezier(160,100.5)(140,110)(120,100.5) 

\put(50,60){\line(-1,-1){40}} \put(35,45){\line(-3,-1){9}} \put(35,45){\line(-1,-3){3}}
\put(50,60){\line(-1,1){40}} \put(30,80){\line(3,-1){9}} \put(30,80){\line(1,-3){3}} 

\put(50,60){\line(-1,0){45}} \put(30,60){\line(-2,1){8}} \put(30,60){\line(-2,-1){8}} 
\put(50,60){\line(0,1){45}} \put(50,85){\line(-1,-2){4}} \put(50,85){\line(1,-2){4}} 

\put(140,105){\line(0,1){15}} \put(140,115){\line(1,-2){3}} \put(140,115){\line(-1,-2){3}} 

\put(95,60){\line(-1,0){15}} \put(90,60){\line(-2,1){6}} \put(90,60){\line(-2,-1){6}}

\thinlines 

\put(240,10){\circle*{2}} \put(240,40){\circle*{2}} \put(240,50){\circle*{2}} \put(240,80){\circle*{2}} 
\put(270,10){\circle*{2}} \put(270,40){\circle*{2}} \put(270,50){\circle*{2}} \put(270,80){\circle*{2}} 

\qbezier(240,120)(230,120)(230,105) \qbezier(230,105)(230,90)(240,90)
\qbezier[10](240,120)(250,120)(250,105) \qbezier[10](250,105)(250,90)(240,90)

\qbezier(240,80)(230,80)(230,65) \qbezier(230,65)(230,50)(240,50)
\qbezier[10](240,80)(250,80)(250,65) \qbezier[10](250,65)(250,50)(240,50)

\qbezier(240,40)(230,40)(230,25) \qbezier(230,25)(230,10)(240,10)
\qbezier[10](240,40)(250,40)(250,25) \qbezier[10](250,25)(250,10)(240,10)

\qbezier(270,120)(260,120)(260,105) \qbezier(260,105)(260,90)(270,90)
\qbezier[10](270,120)(280,120)(280,105) \qbezier[10](280,105)(280,90)(270,90)

\qbezier(270,80)(260,80)(260,65) \qbezier(260,65)(260,50)(270,50)
\qbezier[10](270,80)(280,80)(280,65) \qbezier[10](280,65)(280,50)(270,50)

\qbezier(270,40)(260,40)(260,25) \qbezier(260,25)(260,10)(270,10)
\qbezier[10](270,40)(280,40)(280,25) \qbezier[10](280,25)(280,10)(270,10)

\put(190,130){(i)} \put(295,130){(ii)}

\qbezier(175,85)(175,10)(240,10)

\qbezier(325,85)(325,10)(270,10)

\qbezier(240,90)(210,85)(240,80)
\qbezier(240,50)(210,45)(240,40) 

\qbezier(270,90)(300,85)(270,80)
\qbezier(270,50)(300,45)(270,40) 

\put(178,100){$\varphi^2 (C_1)$} \put(185,40){$\varphi^2 (C_2)$} 
\put(285,114){$\varphi^2 (C_2)$} 

\thicklines 

\qbezier[10](225,85)(215,100)(210,120) \qbezier[20](225,45)(210,70)(210,120)
\qbezier(225,85)(180,120)(225,45)

\put(206,80){\line(-1,1){7}} \put(206,80){\line(1,5){2}} 

\qbezier(230,65)(150,140)(230,25) \put(206,60){\line(-2,1){8}} \put(206,60){\line(0,1){8}} 
\qbezier(300,110)(320,40)(280,25) \qbezier(280,65)(290,70)(300,110)
\put(306,60){\line(-1,-1){7}} \put(306,60){\line(1,-2){4}}

\end{picture}

Therefore, we find a canonical homology basis as follows. 
\begin{align*}
A_1 &= C_2, \quad B_1 = C_1, \quad A_2= \varphi^3 (C_2),\quad 
B_2 = B_1+\varphi^3 (C_1), \\
A_3 &= \varphi^2 (C_2),\quad 
B_3 = B_2+\varphi^2 (C_1). 
\end{align*}

\begin{picture}(340,100)

\qbezier(20,50)(20,90)(160,90) \qbezier(160,90)(300,90)(300,50) 
\qbezier(20,50)(20,10)(160,10) \qbezier(160,10)(300,10)(300,50) 

\qbezier(55,55)(80,35)(105,55) \qbezier(135,55)(160,35)(185,55) \qbezier(215,55)(240,35)(265,55) 
\qbezier(60,52)(80,62)(100,52) \qbezier(140,52)(160,62)(180,52) \qbezier(220,52)(240,62)(260,52)

\put(30,85){{\rm (ii)}} \put(30,15){{\rm (i)}} 

\qbezier(45,50)(45,30)(80,30) \qbezier(115,50)(115,30)(80,30) 
\qbezier(45,50)(45,70)(80,70) \qbezier(115,50)(115,70)(80,70) 

\qbezier(125,50)(125,30)(160,30) \qbezier(195,50)(195,30)(160,30) 
\qbezier(125,50)(125,70)(160,70) \qbezier(195,50)(195,70)(160,70) 

\qbezier(205,50)(205,30)(240,30) \qbezier(275,50)(275,30)(240,30) 
\qbezier(205,50)(205,70)(240,70) \qbezier(275,50)(275,70)(240,70) 

\qbezier(71,46)(55,38)(55,20) \qbezier[15](55,20)(75,25)(71,46) 

\qbezier(151,46)(135,35)(140,10) \qbezier[15](140,10)(160,25)(151,46) 

\qbezier(249,46)(265,38)(265,20) \qbezier[15](249,46)(243,25)(265,20)

\put(80,30){\line(-2,1){5}} \put(80,30){\line(-2,-1){5}} \put(100,25){$A_1$}
\put(160,30){\line(-2,1){5}} \put(160,30){\line(-2,-1){5}} \put(180,25){$A_2$}
\put(240,30){\line(-2,1){5}} \put(240,30){\line(-2,-1){5}} \put(275,35){$A_3$}

\put(62.5,40){\line(-2,-1){5}}  \put(62.5,40){\line(0,-1){5}}  \put(50,10){$B_1$}
\put(144.5,40){\line(-2,-1){5}} \put(144.5,40){\line(0,-1){5}} \put(135,0){$B_2$}
\put(257.5,40){\line(2,-1){5}} \put(257.5,40){\line(0,-1){5}} \put(260,10){$B_3$}

\end{picture}

\subsubsection{period matrix}

Key $1$-cycles of the canonical homology basis as in \S~\ref{tCLP-canonical} are given by $C_1$ and $C_2$. 
Straightforward calculation yields 
\begin{align}
\int_{C_1} \dfrac{1-z^2}{w} dz &= -\sqrt{2} i \int_0^{\infty} \dfrac{1-t^2}{\sqrt{t^8-a t^4+1}} dt
= 0, \label{tCLP-period1}\\
\int_{C_1} \dfrac{i (1+z^2)}{w} dz &= 
2\sqrt{2} \int_0^1 \dfrac{1+t^2}{\sqrt{t^8-a t^4+1}} dt, \label{tCLP-period2}\\
\int_{C_1} \dfrac{2 z}{w} dz &= 
-8 i \int_0^1 \dfrac{t}{\sqrt{t^8-a t^4+1}} dt, \label{tCLP-period3}\\
\int_{C_1} \dfrac{z^4-z^6}{w^3} dz &= 
\sqrt{2} i \int_0^{\infty} \dfrac{t^4-t^6}{\sqrt{t^8-a t^4+1}^3} dt=0, \label{tCLP-period4}\\
\int_{C_1} \dfrac{i(z^4+z^6)}{w^3} dz &= 
-2 \sqrt{2} \int_0^1 \dfrac{t^4+t^6}{\sqrt{t^8-a t^4+1}^3} dt, \label{tCLP-period5}\\
\int_{C_1} \dfrac{z^5}{w^3} dz &= 
4 i \int_0^1 \dfrac{t^5}{\sqrt{t^8-a t^4+1}^3} dt, \label{tCLP-period6} \\
\int_{C_2} \dfrac{1-z^2}{w} dz &= -i \int_0^{\infty} \dfrac{1+t^2}{\sqrt{t^8+a t^4+1}} dt
= -2 i \int_0^{1} \dfrac{1+t^2}{\sqrt{t^8+a t^4+1}} dt, \label{tCLP-period7} \\
\int_{C_2} \dfrac{i (1+z^2)}{w} dz &= 
-2 i \int_0^{1} \dfrac{1+t^2}{\sqrt{t^8+a t^4+1}} dt,  \label{tCLP-period8}\\
\int_{C_2} \dfrac{2 z}{w} dz &= 
-8 \int_0^1 \dfrac{t}{\sqrt{t^8+a t^4+1}} dt, \label{tCLP-period9} \\
\int_{C_2} \dfrac{z^4-z^6}{w^3} dz &= 
- i \int_0^{\infty} \dfrac{t^4+t^6}{\sqrt{t^8+a t^4+1}^3} dt=
-2 i \int_0^{1} \dfrac{t^4+t^6}{\sqrt{t^8+a t^4+1}^3} dt, \label{tCLP-period10} 
\end{align}
\begin{align}
\int_{C_2} \dfrac{i(z^4+z^6)}{w^3} dz &= 
-2 i \int_0^{1} \dfrac{t^4+t^6}{\sqrt{t^8+a t^4+1}^3} dt, \label{tCLP-period11} \\
\int_{C_2} \dfrac{z^5}{w^3} dz &= 
-4 \int_0^1 \dfrac{t^5}{\sqrt{t^8+a t^4+1}^3} dt. \label{tCLP-period12} 
\end{align}
By setting 
\begin{align*}
A &= 2\sqrt{2} \int_0^1 \dfrac{1+t^2}{\sqrt{t^8-a t^4+1}} dt, \quad 
B = 2 \int_0^{1} \dfrac{1+t^2}{\sqrt{t^8+a t^4+1}} dt, \\
C &= 8 \int_0^1 \dfrac{t}{\sqrt{t^8+a t^4+1}} dt, \quad 
D = 8 \int_0^1 \dfrac{t}{\sqrt{t^8-a t^4+1}} dt, \\
E &= 2 \sqrt{2} \int_0^1 \dfrac{t^4+t^6}{\sqrt{t^8-a t^4+1}^3} dt, \quad 
F = 2 \int_0^{1} \dfrac{t^4+t^6}{\sqrt{t^8+a t^4+1}^3} dt, \\
H &= 4 \int_0^1 \dfrac{t^5}{\sqrt{t^8+a t^4+1}^3} dt, \quad 
I = 4 \int_0^1 \dfrac{t^5}{\sqrt{t^8-a t^4+1}^3} dt, 
\end{align*}
and \eqref{tCLP-period1}--\eqref{tCLP-period12}, we have 
\[
\int_{C_1} G = 
\begin{pmatrix}
0 \\
A \\
-i D \\
0 \\
-E \\
i I
\end{pmatrix}, \quad 
\int_{C_2} G = 
\begin{pmatrix}
-i B \\
-i B \\
-C \\
-i F \\
-i F \\
-H
\end{pmatrix}.
\]
Therefore, the period matrix of the abelian differentials of the second kind is given by 
\[
\begin{pmatrix}
- i B   &   i B  &  i B  & 0  &   -A   &  -A \\
- i B  &   -i B  &  i B  &  A  &  A  &  0 \\
- C   &   C  &  - C  &  - i D  &  0  &  -i D \\
- i F   &   i F  &  i F  &  0  & E  &   E \\ 
- i F  &  -i F   &  i F  &  -E  & -E  &  0  \\ 
- H &   H  &   -H  &  i I  &  0  &  i I
\end{pmatrix}. 
\]

$\,$\\
Norio Ejiri \\ 
Department of Mathematics, Meijo University \\
Tempaku, Nagoya 468-8502, Japan.\\
ejiri@meijo-u.ac.jp\\
$\,$\\
Toshihiro Shoda \\
Faculty of Education, Saga University \\
1 Honjo-machi, Saga-city, Saga, 840-8502, Japan. \\ 
tshoda@cc.saga-u.ac.jp

\end{document}